\documentclass[12pt,dvips]{article}
\usepackage{stmaryrd}% pour \llbracket et \rrbracket
\usepackage{latexsym}
\usepackage{srcltx}
\usepackage{amsmath,amsthm,amsfonts,amssymb, mathrsfs, wasysym}
\usepackage[T1]{fontenc}
\usepackage[latin1]{inputenc}
\usepackage{aeguill}% pour les guillemets francais et autres bizrreries du latin1
\usepackage{url}%utile pour les preprints avec url
\usepackage{enumerate}% \begin{enumerate}[cas 1]
\usepackage{mdwlist} % enumerate*
\usepackage{rotating} %utilise pour les fleches \actson \actedon
\usepackage{graphicx}
\usepackage[a4paper,textwidth=15cm,textheight=25cm]{geometry}
\usepackage{xspace}
\usepackage{export}%pour sauvecompteurs
\usepackage[vlined,ruled]{algorithm2e}
\nocaptionofalgo
\newtheorem{thm}{Theorem}[section]
\newtheorem{thmbis}{Theorem}
\newtheorem*{thm*}{Theorem}
\newtheorem{dfn}[thm]{Definition} 
\newtheorem*{dfn*}{Definition}

\newtheorem{cor}[thm]{Corollary}
\newtheorem*{cor*}{Corollary}
\newtheorem{corbis}[thmbis]{Corollary}
\newtheorem{prop}[thm]{Proposition} 
\newtheorem*{prop*}{Proposition} 
\newtheorem*{properties*}{Properties} 
 
\newtheorem{lem}[thm]{Lemma} 
\newtheorem*{lem*}{Lemma} 
 
\newtheorem{claim}[thm]{Claim} 
\newtheorem*{claim*}{Claim} 
 
\newtheorem*{fact*}{Fact}

\newtheorem*{qst*}{Question}

\newtheorem*{pb*}{Problem}

\newtheorem*{algo*}{Algorithm}

\theoremstyle{remark}
\newtheorem*{rem*}{Remark}
\newtheorem{rem}[thm]{Remark}
\newtheorem*{example*}{Example}
\newtheorem{example}[thm]{Example}

\newenvironment{SauveCompteurs}[1]{%
\newcommand{\monparametre}{#1}
\openexport{\monparametre_sauve}
  \Export{thmbis}%\Export{section}\Export{subsection}\Export{subsubsection} %(utile si \Export{thm})
\closeexport}{}

\newenvironment{UtiliseCompteurs}[1]{%
\newcommand{\monparametre}{#1}% utile car en fait #1 n'est pas accessible dans la deuxieme partie de l'environnement.
\openexport{\monparametre_aux}
  \Export{thmbis}%\Export{section}\Export{subsection}\Export{subsubsection}
\closeexport
\Import{\monparametre_sauve}%
\renewcommand{\label}[1]{}%But=eviter les multiply defined. 
}{\Import{\monparametre_aux}}

\newlength{\espaceavantspecialthm}
\newlength{\espaceapresspecialthm}
{\normalfont \vskip \espaceapresspecialthm}

\newenvironment{specialthm*}[1]{
\vskip\espaceavantspecialthm \noindent \textbf{#1} \itshape}%
{\normalfont \vskip \espaceapresspecialthm}
\newlength{\espaceavantenonce}
\newlength{\espaceapresenonce}
\newcommand{\fontetitreun}[1]{\textbf{#1}} %c'est le format du titre des enonce1 [Theoreme 1.1]
\newcommand{\fontetitredeux}[1]{\textit{#1}} %c'est le format du titre des enonce1 [Theoreme 1.1]

{\normalfont \vskip \espaceapresenonce}

\newenvironment{enonce1*}[1]{
\vskip\espaceavantenonce \noindent \fontetitreun{#1} \itshape}%
{\normalfont \vskip \espaceapresenonce}

{\vskip \espaceapresenonce}

\newenvironment{enonce2*}[1]{
\vskip\espaceavantenonce \noindent \fontetitredeux{#1} }%
{\vskip \espaceapresenonce}

\makeatletter
\edef\@tempa#1#2{\def#1{\mathaccent\string"\noexpand\accentclass@#2 }}
\@tempa\rond{017}
\makeatother

\newcommand{\es}{\emptyset}
\renewcommand{\phi}{\varphi} 
\newcommand{\m} {^{-1}} 
\newcommand{\eps} {\varepsilon} 
\newcommand{\st}{|\,}%such that

\newcommand {\ra} {\rightarrow}

\newcommand {\onto} {\twoheadrightarrow}
\newcommand {\xra} {\xrightarrow}

\newcommand{\actson}{\,\raisebox{1.8ex}[0pt][0pt]{\begin{turn}{-90}\ensuremath{\circlearrowright}\end{turn}}\,}
 
\newcommand{\ul}[1]{\underline{#1}} 
\newcommand{\ol}[1]{\overline{#1}}

\newcommand{\normal} {\vartriangleleft}
\renewcommand{\subsetneq}{\varsubsetneq}

\newcommand{\dunion}{\sqcup}% on peut utiliser \sqcup ou  \amalg

\newcommand{\Dunion}{\bigsqcup} % on peut utiliser \bigsqcup ou \coprod

\newcommand{\ie} {i.~e.\ }

\newcommand {\calb} {{\mathcal {B}}}   
\newcommand {\calc} {{\mathcal {C}}}   
\newcommand {\cald} {{\mathcal {D}}}   
\newcommand {\cale} {{\mathcal {E}}}   
\newcommand {\calf} {{\mathcal {F}}}   
\newcommand {\calg} {{\mathcal {G}}}

\newcommand {\calk} {{\mathcal {K}}}   
\newcommand {\call} {{\mathcal {L}}}   
\newcommand {\calm} {{\mathcal {M}}}   
   
\newcommand {\calo} {{\mathcal {O}}}   
\newcommand {\calp} {{\mathcal {P}}}   
\newcommand {\calq} {{\mathcal {Q}}}   
\newcommand {\calr} {{\mathcal {R}}}   
\newcommand {\cals} {{\mathcal {S}}}   
\newcommand {\calt} {{\mathcal {T}}}

\newcommand {\bbN} {{\mathbb {N}}}

\newcommand {\bbQ} {{\mathbb {Q}}}   
\newcommand {\bbR} {{\mathbb {R}}}

\newcommand {\bbZ} {{\mathbb {Z}}}   

\newcommand{\abs}[1]{\lvert#1\rvert} %recommande par amsldoc
\newcommand{\grp}[1]{\langle #1 \rangle}
\newcommand{\dom} {\mathop{\mathrm{dom}}} %mathop permet de l'utiliser comme un operateur, ie sans parentheses (\dom w) et il gere l'espace tout seul
\newcommand{\Isom} {{\mathrm{Isom}}}

\newcommand{\Fix}{{\mathrm{Fix}}}

\newcommand{\Out} {{\mathrm{Out}}}
\newcommand{\Aut} {{\mathrm{Aut}}}

\newcommand{\Homeo} {{\mathrm{Homeo}}}

\newcommand{\id} {\mathrm{id}}

\newcommand{\G}{\Gamma}

\newcommand{\tto}{\twoheadrightarrow}
\newcommand{\K}{\mathcal{K}}

\newcommand{\calT}{\mathcal{T}}

\newcommand{\Cay}{\mathrm{Cay\,}}
\newcommand{\Cyl}{\mathrm{Cyl}}

\newcommand{\lqg}{\call\calq\calg}
\newcommand{\qg}{\calq\calg}
\newcommand{\Hfmi}{\Hat S_\pm^*}
\newcommand{\qi}{quasi-isometrically }

\newcommand{\Spm}{S_{\pm}}
\newcommand{\fmi}{\Spm^*}
\newcommand{\fmip}{\Spm^+}
\newcommand{\AutS}{{\Aut(\Spm)}}
\newcommand{\Sol}{\mathrm{Sol}}

\begin{document}
\title{Foliations for solving equations in groups: \\
 free, virtually free, and hyperbolic groups.}
\author{Fran{\c c}ois Dahmani, Vincent Guirardel}
\date{}
%{\today.\\ \small Fichier \texttt{\jobname.tex}}

\maketitle

\begin{abstract}
We give an algorithm for solving equations and inequations with rational constraints
in virtually free groups. Our algorithm is based on Rips' classification of measured band complexes.
Using canonical representatives, we deduce 
 an algorithm for solving equations and inequations 
in all hyperbolic groups (possibly with torsion).
Additionally, we can deal with \qi embeddable rational constraints.
\end{abstract}

\setcounter{section}{-1}
\section{Introduction}

  \subsection{The equations problem}

Given a group $G$, the \emph{equations problem} in $G$ consists 
in deciding algorithmically whether a system of equations with constants has a solution in $G$ or not.
An equation is an equality $w=1$ where $w$ is a word on a set of variables an their inverses 
together with constants taken from $G$.
When inequations (\ie negation of equations) are allowed, we call this problem the \emph{problem of equations and inequations}.
In other words, the problem of equations and inequations is equivalent to the decidability of the existential (or universal)
theory of $G$ with constants in $G$.

A solution to the equations problem is quite powerful, as it vastly generalises 
the word problem, the conjugacy problem, the simultaneous conjugacy problem, etc.
When $G$ is  abelian, the problem of equations and inequations is easily solved using linear algebra.  
But already, if $G$ a free $3$-step nilpotent group of rank 2, the equations problem
is undecidable (\cite{Truss_equation}, see also \cite{Romankov_universal}).
This is based on Matiyasevich's Theorem 
saying that one cannot decide the solubility of polynomial equations in $\bbZ$  \cite{Matiyasevich}.
For a free $2$-step nilpotent group, 
the equations problem
is solvable if and only if one can decide 
the solubility 
of polynomial equations over $\bbQ$ \cite{Romankov_universal}.
Additionally, if $G$ is a non-commutative free metabelian group, 
the equations problem in $G$ is unsolvable \cite{Romankov_metabelian},
but the problem of equations and inequations \emph{without constants} is solvable \cite{Chapuis_universal}.

The problem of equations and inequations in free groups is
natural, and has attracted attention of many people (Lyndon, Appel,
Lorents...).  The solution of this problem by Makanin in 1982
\cite{Makanin_equations} (with the appropriate correction in \cite{Makanin_decidability}) certainly constitutes a milestone in the
theory.  It has been a source of inspiration for Rips for his study of
group actions on $\mathbb{R}$-trees, and his solution to Morgan and Shalen's
conjecture \cite{BF_stable,GLP1}, which found applications in many
branches of geometry.  It has been a decisive step towards
algorithmic, and theoretical description of the set of homomorphisms
of a group into a free group (Razborov \cite{Razborov_systems}).  It has also been,
together with these developments, a prelude to the far reaching recent
solutions to Tarski's problems on the elementary theory of free
groups, by Sela \cite{Sela_diophantine6}, and by Kharlampovich and Miasnikov
\cite{KhMy_elementary}. 
 Makanin's algorithm is also the basis of Rips and Sela's
solution to the  equations problem  in torsion free
hyperbolic groups that they manage to reduce to the equations problem in a free group
\cite{RiSe_canonical}.  
Finally, it is crucial in Sela's
solution of the isomorphism problem for torsion free hyperbolic groups
with finite outer-automorphism group: he makes a delicate use of
Makanin's algorithm, and of Rips' classification of actions on
$\mathbb{R}$-trees \cite{Sela_isomorphism}. 

Our main result is the following theorem, proved in section \ref{sec_endgame} (see below for definitions and discussion).

\begin{SauveCompteurs}{thm_eqn_hyp}
  \begin{thmbis}\label{thm_eqn_hyp}
    There exists an algorithm which takes as input
    \begin{itemize*}
    \item a presentation of a hyperbolic group $G$ (possibly with torsion),
    \item a finite system of equations and inequations with constants in
      $G$, and with \qi embeddable rational constraints
    \end{itemize*}
    and which decides whether there exists a solution or not.
  \end{thmbis}
\end{SauveCompteurs}

In a forthcoming paper, we will use this algorithm to give a solution to the isomorphism problem
for all hyperbolic groups (possibly with torsion) \cite{DG2}.

We pursue several goals in this paper. 
Our first goal is a proof of Theorem \ref{thm_eqn_hyp}.
As in \cite{RiSe_canonical,Dah_existential}, our proof is based on canonical representatives, which allow us
to reduce to the equations problem in a virtually free group,
\ie a finite extension of a free group.

Our second goal is therefore to 
give a solution of the equations problem in virtually free groups.
This problem easily reduces to a problem of \emph{twisted equations} in a free group.

Last, but not least, we present a new approach to Makanin's algorithm,
based on Rips' theory for \emph{foliated   band complexes}, allowing us to solve these twisted equations.
This occupies the major part of this paper.

We continue this introduction by reviewing different aspects and concepts involved in our strategy.

\subsection{Rational constraints}

In a group $G$, the class of \emph{rational subsets} is the smallest
class containing finite subsets, and closed under finite union 
$A\cup B$, product $A\cdot B$, and semi-group generation $A^*$.  Equivalently, if
$S$ is a finite generating set of $G$, a subset $A\subset G$ is
rational if it is the image in $G$ of a regular language $L$ of the
free monoid on $S\cup S\m$, a regular language being a language
recognised by a finite state automaton.  
Solving a system of equations with \emph{rational constraints} consists in
solving this system of equations with the requirement that
each variable $x$ lies in a rational subset $\calr_x$ given in advance.

In a hyperbolic group, we propose the class of \emph{\qi embeddable} rational subsets as the suitable class for constraints. 
A rational subset $A\subset G$ is \qi embeddable if one can choose the regular language $L$ in the free monoid 
to consist of quasigeodesics (with uniform constants).
For example, a quasi-convex subgroup $H$ of a hyperbolic group $G$ is a \qi embeddable rational subset. Note that by \cite{Kap_Detecting}, one can compute a finite state automaton representing $H$ from a finite generating set. 
In general, the class of rational subsets of a group is not closed under complementation or intersection.
However, the set of \qi embeddable rational subsets of a hyperbolic group is a Boolean algebra (Cor. \ref{cor_boolean}).
For instance, the complement of a quasi-convex subgroup is also a \qi embeddable rational subset.
In particular, inequations in a hyperbolic group can be encoded using \qi embeddable rational constraints of the form $\calr=G\setminus\{1\}$.
In a virtually free group, every rational subset is \qi embeddable,
and the set of all rational subsets is a Boolean algebra.

Equations with rational constraints in a free monoid were first considered by Schulz \cite{Schulz_Makanin}.
Following an approach of Plandowski for free monoids,  
Diekert, Gutierrez and  Hagenah proved that systems of equations, inequations, 
and rational constraints in free groups are algorithmically solvable \cite{Plandowski_satisfiability,DGH_pspace}. 

The use of rational constraints in systems of equations turns out to be rather powerful. 
The problem of equations and inequations for right angled Artin groups has been reduced by
Diekert and Muscholl to a problem of equations with rational constraints in free groups \cite{DiekertMuscholl_graph-groups}.
This has been generalised by Diekert and Lohrey to 
free products, direct products, and graph products of certain groups 
\cite{DiekertLohrey_graph-products}.
It is the key tool to extend Rips and Sela's solution to
the equations problem for torsion free hyperbolic groups into a
solution to the problem of equations and inequations \cite{Dah_existential}.
It greatly streamlines the solution of the isomorphism problem for
torsion free hyperbolic groups, and allows  substantial
generalisations  \cite{DaGr_isomorphism,DG2}. It plays an important role in our 
solution to the equations problem in virtually free groups (Th.\ref{thm_vf} below),
and in Lohrey and Senizergues' independent solution \cite{LoSe_equations},
even if the initial problem does not involve inequations or rational constraints.
\\

For the sake of illustration, let us present two elementary applications of the use of rational constraints in a system of equations. 
Let $F_2$ be a free group of rank $2$. Given a finitely generated subgroup $H<F_2$ (or more generally any rational subset), 
and an element $x\in F_2$, one can decide if there is an automorphism of $F_2$ sending $x$ into $H$. 
Indeed, consider $a,b$ a basis of $F_2$, and write $x=w(a,b)$ as a word on $a,b$.
The orbit of $x$ under $\Aut(F_2)$ intersects $H$
if and only if there exists a basis $u,v$ of
$F_2$ such that $w(u,v) \in H$. By a theorem of Dehn, Magnus and
Nielsen, $(u,v)$ is a basis if and only if 
$\exists g, \, [u,v]^g= [a,b]^{\pm1}$. Therefore, the orbit of
$x$ intersects $H$ if and only if the system of equations with rational constraints
$$
\left\{
\begin{aligned}{}
[u,v]^g= [a,b]^{\pm1}\\
z=w(u,v),\quad z\in H
\end{aligned}
\right.$$
in the variables $u,v,g,z$ has a solution.

Our second application is  an immediate consequence of Theorem \ref{thm_eqn_hyp}:

\begin{corbis} Let $G$ be a hyperbolic group. 
Given a quasiconvex subgroup $H<G$, one can decide its malnormality
by solving the system $z=yxy\m$, $z\neq 1$ 
with rational constraints $z,x\in H$, $y\notin H$.
\end{corbis}

In presence of torsion, almost malnormality (meaning that $yHy\m\cap H$ is finite when $y\notin H$)
can be checked similarly by replacing the inequality $z\neq 1$ by an inequality $z^{N!}\neq 1$ where $N$ is a bound
on the order of torsion in $G$.

This result does not seem to appear in the literature.
This is a variant of problem H14 in \cite{open_pbs} and of Question 3 in \cite{BrWi_malnormality}.
Without assumption of quasiconvexity, malnormality is undecidable \cite{BrWi_malnormality}.
\\

\subsection{Lifting equations and rational constraints to a virtually free group}

Following the strategy initiated in  \cite{RiSe_canonical}, and
continued in  \cite{Dah_existential}, we now explain how to
reduce Theorem \ref{thm_eqn_hyp} to the problem
of solving equations with rational constraints in a virtually free group (see Section \ref{sec_eqn_hyp}).

In \cite{RiSe_canonical}, Rips and Sela introduced \emph{canonical representatives} for torsion free hyperbolic groups, 
which enabled them to reduce the equations problem in a torsion free hyperbolic group
to the equations problem in a free group.
These canonical representatives are paths in the Cayley graph of $G$
which satisfy some \emph{path equations} representing the initial equations.
Such paths correspond to words on the generating system of $G$, and thus to elements of the corresponding free group.
In presence of torsion, canonical representatives need to be interpreted as paths in a barycentric subdivision $X$ of a Rips complex of $G$.
The action of $G$ on $X$ is not free in general. 
The quotient of the $1$-skeleton $X^{(1)}/G$ is a finite graph of finite groups, whose fundamental group is a virtually free group $V$.
Path equations in $X$ are then interpreted in terms of equations in $V$.

Another task that needs to be done, is to lift rational constraints to $V$.
This can be done because both canonical representatives and paths representing the rational subsets are quasi-geodesics.
This part of the argument is similar to Cannon's argument showing that the language of geodesics is a regular language \cite{Cannon_combinatorial,CDP}.

This allows to reduce Theorem \ref{thm_eqn_hyp} to the following result (proved in section \ref{sec_endgame}):
\begin{SauveCompteurs}{thmvf}
  \begin{thmbis}\label{thm_vf} 
  The problem of equations with rational constraints is solvable in finite extensions of free groups.

  More precisely, there exists an algorithm which takes as input a presentation of a virtually free group $G$, and a system of
  equations and inequations with constants in $G$, together with a set of rational constraints, and which decides whether there exists
  a solution or not.
\end{thmbis}
\end{SauveCompteurs}

This result was independently obtained in a more general form in \cite{LoSe_equations}.
Unlike \cite{LoSe_equations}, our approach does not rely on an existing solution of the equations problem in free groups.
Instead, we give a new proof, based on Rips' theory for foliated band complexes.

\subsection{Twisted equations and virtually free groups}

A \emph{twisted equation} in a group $G$ is an equation of the form 
$\phi_1(x_1)\dots\phi_n(x_n)=1$ where each $\phi_i$ is a fixed automorphism in $\Aut(G)$,
and $x_i$ is a variable or a constant.
For instance, twisted conjugacy involves a simple example of a twisted equation:
given an automorphism $\phi\in \Aut(G)$, two elements $a,b\in G$ are twisted conjugate
if there exists $x\in G$ such that  $xa\phi(x^{-1}) =b$. 

When one considers equations in a finite  
extension $1\to N\to G\to Q\to 1 $ ($Q$ is finite), 
twisted equations appear naturally.  Indeed one can
replace an equation $xyz=1$, by a disjunction of equations
$\left(\tilde{q}_xn_x \right) \left(\tilde{q}_yn_y \right) \left(\tilde{q}_zn_z\right) =1$, 
where the  constants $\tilde{q}_x,\tilde{q}_y,\tilde{q}_z$ are chosen
in a given cross-section of $Q$, and $n_x, n_y, n_z$ are new unknowns in $N$. Gathering the elements
$\tilde{q}_x,\tilde{q}_y,\tilde{q}_z$ to the left amounts to twist $n_x$ and $n_y$ by
suitable automorphisms.
Thus, the solubility of equations in $G$ reduces to the solubility of finitely many systems of twisted equations in $N$.
The twisting morphisms occurring in this manner are quite particular because
they generate a finite subgroup of the outer automorphism group $\Out(N)$.

In the Kourovka notebook, Makanin asked about the problem of twisted equations
\cite[Problem 10.26(b)]{Kourovka15}. 
We are able to give a positive answer 
to Makanin's question above assuming that the given automorphisms generate a finite subgroup of $\Out(F)$:

\begin{SauveCompteurs}{twistedpb}
\begin{thmbis}\label{thm_twistedpb}
  There exists an algorithm that takes as input a basis of a free group $F$, a finite set $\Phi$ of automorphisms of $F$ whose
  image in $\Out(F)$ generate a finite subgroup, and a system of twisted equations with rational constraints in $F$ (with twisting
  automorphisms in $\Phi$) and that decides whether there is a solution or not.
\end{thmbis}
\end{SauveCompteurs}

As explained above, Theorem \ref{thm_vf} easily follows from Theorem \ref{thm_twistedpb} (see Section \ref{sec_eqn}).
Moreover, we give a trick to reduce to the case where the twisting automorphisms of $F$
permute the elements of $S\cup S\m$ for some free basis $S$ of $F$.
This trick is related to the Zimmerman-Culler Theorem \cite{Zimmermann_CMH81,Culler_finite} which 
realises any finite subgroup of $\Out(F)$ as a finite group $H$ of automorphisms of a graph $X$ whose fundamental group is $F$. 
Because $H$ may fail to fix a point in $X$, $H$ may fail to lift to a finite subgroup of $\Aut(S)$.
This is why we need to embed $F$ into a larger free group $\Hat F$, whose basis is the set of edges of $X$.
The group $H$ is then realised as a subgroup of $\Aut(\Hat F)$ permuting the basis elements, and $H$ preserves the conjugacy class of $F$.
The initial system of equations gives a new system of equations in $\Hat F$, and we add rational constraints
saying that the variables should live in $F$.
 This reduction is the content of Proposition \ref{prop;trick_Out_Aut}
in the broader context of equations with rational constraints.

\subsection{Dynamical and geometric aspects}

We now focus on the equations problem in a free group.
Makanin and Razborov developed a combinatorial machinery to encode
equality of subwords occurring in a solution of an equation.  An
interesting and well written account on Makanin's algorithm for
equations in free monoids (a simplified version of the case of free
groups, that does not imply a solution for free groups) was given in
\cite{Diekert_encyclopedia}, and another one on Makanin and Razborov's
algorithm for free groups was given in \cite{KhMy_implicit}.  As we
said, Rips was inspired by this machinery for his study of foliated
band complexes, whose dynamics, on the other hand, reflect actions on
$\bbR$-trees.  Our strategy is to reverse this flow of ideas, and use
Rips' classification of foliated band complexes to prove that the
algorithm we propose always stops.

   We hope that this point of view on  Makanin's algorithm will be of 
interest to an audience concerned with  equations problems and also to an audience
concerned with geometry and dynamics of group actions.

Rips' theory is an understanding of actions of finitely generated
groups on $\mathbb{R}$-trees. Recall that an $\bbR$-tree is a geodesic metric space in which any two
points are joined by a unique injective path. 
Following Sela, let us try to explain how the equations problem in a free group $F$
is related to  $\bbR$-trees \cite{Sela_diophantine1}.

 The equations problem in $F$ is about homomorphisms of
finitely presented groups to $F$. Indeed, a system of
equations on a set of unknowns $X$ is a finite set $\cale$ of words in the free product $F*\grp{X}$,
and a solution of this system of equations corresponds to a morphism from $G_\cale=F*\grp{X}/\grp{\grp{\cale}}$ to $F$
which is the identity on $F$.
Each such morphism gives rise to an action of $G_\cale$ on the 
Cayley graph of $F$, a simplicial tree. 
One can rescale this tree in order to normalise the maximal displacement of the generators of $G_\cale$. 
 For an infinite sequence of solutions,
 the corresponding actions of $G_\cale$ on the Cayley graph
converge to an action of $G_\cale$ on some $\mathbb{R}$-tree. 

Rips' theory says that under suitable hypotheses, this action can be understood in terms
of actions on simplicial trees, and actions on $\bbR$-trees dual to minimal measured foliations on $2$-complexes.
The main result of Rips' theory is a classification of those minimal measured foliations into three types:
\begin{itemize}
\item homogeneous type (also known as axial or toral), whose dual $\bbR$-tree is a line
\item surface type (also known as interval exchange): the $\bbR$-tree is dual to a measured foliation on a surface (or a $2$-orbifold)
\item exotic type (also known as Levitt or thin).
\end{itemize}

 One can try to use these ideas 
to decide whether a  given system of equations has a solution, and to look for a \emph{shortest} solution 
(e.g.\ in terms of the maximal displacement of the generators).
If we have some solutions such that
the corresponding actions of $G_\cale$ on the Cayley graph  of $F$ are close enough to the limiting $\bbR$-tree,
we can apply Sela's \emph{shortening argument}.
This argument says that, under suitable hypotheses, 
there is a quotient of $G_\cale$ through which the actions factorise, 
and  automorphisms of this quotient that \emph{shorten} 
 all nearby solutions. 
These solutions are not shortest and  can therefore be ignored.
Using some compactness argument, only finitely many solutions 
do not lie in such a shortening neighbourhood of a limiting $\bbR$-tree.
One can hope to bound the length of these remaining solutions and check
by hand if such a solution exists.

\begin{figure}[htbp]
  \centering
\includegraphics{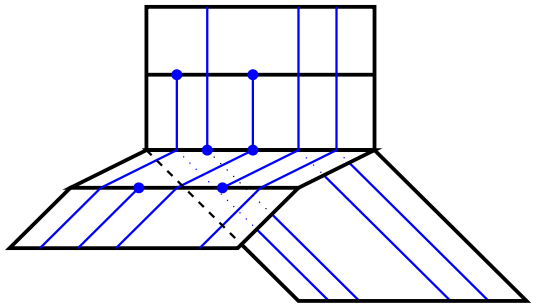}
  \caption{A prelamination is a finite disjoint union of \emph{leaf segments}.}
 \label{fig_prelamination}
\end{figure}

A major problem in this approach is that we cannot even start as
we don't know if our system of equations has any solution (this is what we have to decide).
Therefore, we cannot work with \emph{actual} solutions.
Instead, we work with \emph{potential} solutions.
An actual solution, \ie a morphism $G_\cale\ra F$, 
can  
be represented by a continuous map from a presentation complex $X$ of $G_\cale$ to a wedge of circles $Y$.
The preimage of midpoints of edges of $Y$ is a combinatorial \emph{lamination} of $X$.
Instead of working with laminations, we work with \emph{prelaminations} that play the role of potential solutions.
A subcomplex of a lamination is a typical example of a prelamination, but prelaminations are defined by local conditions
and are not required to extend to an actual lamination (See Figure \ref{fig_prelamination}).

Our proof shows that when a prelamination is very long, then any actual solution corresponding to this
prelamination can be shortened, and therefore ignored.
This is the key of the termination of our algorithm.

\subsection{Overview of the main algorithm}

Let us describe the approach to solving equations in free groups developed in this paper.

The first step toward Theorem \ref{thm_twistedpb} is to encode a system of
triangulated equations in a \emph{band complex}.
Band complexes are Makanin's \emph{generalised equations}, but they also play the role of a presentation complex of $G_\cale$ in the above discussion.
A band complex consists of $D$ a disjoint union of non-degenerate segments (one segment for each variable or constant involved in $\cale$)
together with a finite set of rectangles (called \emph{bands}) attached to $D$ by two opposite sides called its bases.
If one considers twisted equations in the context described above,
then each band carries a specific automorphism. If one considers
rational constraints, then subsegments of $D$ carry 
 regular languages.

A \emph{solution} of the band complex in the free group $\grp{S}$ is a labelling of $D$ by reduced words on $S\cup S\m$
so that both bases of each band are labelled by the same reduced word (up to composing by the automorphism of the band), 
and so that the label of each subsegment satisfies the associated rational constraint.
For instance, to encode constants, one can use regular languages consisting of a single word 
to impose the labelling on some subsets of $D$.
Solving twisted equations in free groups easily reduces to deciding the existence of a solution to
such a band complex (see Section \ref{sec_band_cplex}).
To simplify this presentation, we now forget about rational constraints.

Our \emph{main algorithm} \ref{algo;main} decides whether a given band complex admits a solution (\ie a labelling as above)
using prelaminations. 
A prelamination is a finite disjoint union of \emph{leaf segments},
and a leaf segment is a segment contained in a band and joining its two bases (see Figure \ref{fig_prelamination}).
A prelamination is \emph{induced} by a solution if its leaf segments join matching subwords of the labelling.
A prelamination is \emph{complete} if one cannot extend any leaf (\ie if it is an actual lamination).
Given a complete prelamination, it is easy to decide if there exists a solution that induces it.

We can explore the space of possible prelaminations, in quest of a complete prelamination induced by a shortest solution. 
The main concern is how to detect that there is no solution.

For each complete prelamination, we can decide if it is induced by a solution or not.
If it is, we are done, and otherwise, we reject this prelamination.
For each incomplete prelamination $\call$, we run a
\emph{prelamination analyser} 
which tries to 
find a certificate ensuring that no shortest solution can induce $\call$. 
In this case one can reject $\call$.
The analyser may fail to find any such certificate and might say ``I don't know''.
There are several kinds of certificates:
detection of an incompatibility with constants,
non-existence of an invariant (combinatorial) transverse measure,
or a \emph{shortening certificate} proving that if some solution induces $\call$,
then there is a shorter solution.

Note that if a lamination $\call'$ extends a rejected $\call$, then $\call'$ cannot be induced
by a shortest solution.
If all prelaminations remaining to analyse are extensions of rejected ones,
we know that there is no shortest solution, hence no solution at all, and the machine stops.

 To prove that this algorithm works, we assume that there is no solution,
and show that the algorithm rejects all sufficiently long prelaminations by 
 producing appropriate shortening certificates.

 This is where we use Rips' classification of measured foliations on
 band complexes.  Assume by contradiction that the algorithm does not stop. By
 extracting a limit of an increasing sequence of non rejected prelaminations $\call_1\subset\call_2\subset\dots$,
one can construct a topological foliation on the band
 complex with an invariant measure.  
We prove that this measure has no atom. Assume for simplicity that it has full support
(this is not true in general).
One can decompose the foliated band complex into minimal components. 
By Rips' theory, they are classified as homogeneous, exotic, and
 surface type components.  

In presence of a homogeneous component,  
for all $i$ large enough, $\call_i$ shows large repetition patterns that force
solutions inducing $\call_i$ to
have subwords that are arbitrarily high powers. 
By Bulitko's Lemma, there is a computable a priori bound on such powers in a shortest solution.
One can therefore produce a shortening certificate for $i$ large enough.

In presence of a surface or exotic component, one can perform, as in the Rips Machine, an infinite sequence of moves
on the foliated band complex $\Sigma$.
These moves don't increase the complexity of $\Sigma$, so finitely many unfoliated band complexes are visited.
Assume for simplicity that $\Sigma$ itself is visited twice.
For $i$ large enough, the corresponding moves are compatible with $\call_i$,
and transform any solution of $\Sigma$ inducing $\call_i$ into a shorter solution of $\Sigma$.
These moves provide a shortening certificate for $\call_i$:  if $\call_i$ is induced by some solution,
then this solution cannot be shortest.
When $\Sigma$ itself is not visited twice, this argument needs to be refined, and shortening must be
performed by at least a certain Lipschitz factor.

The running time of our algorithm does not seem to be very good.
Makanin's algorithm (in its corrected version \cite{Makanin_decidability}) is known not to be primitive recursive \cite{KoPa_Makanin},
and we don't see why ours should be better. 
Using  some data compression on words, Plandowski proposed an algorithm of much better complexity (polynomial in space)
\cite{Plandowski_satisfiability,DGH_pspace}.

\subsection{Organisation of the paper}

Sections \ref{sec_prelim} to \ref{sec_band_cplex} set up the vocabulary, and reduce the equations problem in virtually free groups
to band complexes.
Section \ref{sec_bulitko} is about Bulitko's Lemma on the periodicity exponent of minimal solutions.
Section \ref{sec_prelamin} introduces prelaminations and related notions.
Section \ref{sec_algos} describes the prelamination generator and the prelamination analyser, and
gives a precise definition of a \emph{shortening sequence of moves} that we use as a certificate for rejection.
Section \ref{sec_asym} contains the main part of the argument. Limits of sequences of prelaminations are analysed using Rips' theory,
and existence of shortening sequences of moves is proved.
Section \ref{sec_endgame} contains a detailed account on the main algorithm. 
This section can be read independently, 
admitting a certain number of well identified results of the previous sections.
Although based on Theorem \ref{thm_vf}, Section \ref{sec_eqn_hyp} is independent from the rest of the paper.
It deals with the equations problem in hyperbolic groups using canonical representatives, and 
with \qi embeddable rational constraints.

In a first reading, one can first forget about rational constraints.
Moreover, if one is only interested in the case of free groups, \ie with untwisted equations, one can safely ignore
all technical considerations about M\"obius strips (see Remark \ref{rem_strip}).

{\small
\tableofcontents
}

\section{Preliminaries}\label{sec_prelim}

\subsection{Regular languages}

Let $S$ be a finite set, $\ol S$ a copy of $S$, and $\Spm=S\dunion \ol S$
endowed with the canonical involution $s\mapsto \ol s$ exchanging $S$ and $\ol S$.
Let $\fmi$ be the free monoid on $\Spm$, endowed with 
the involution $w=s_1\dots s_n\mapsto \ol w=\ol s_n\dots \ol s_1$.

An \emph{automaton}  over $\Spm$ is a directed graph where each (oriented) edge is labelled by some element of $\Spm$
together with two finite subsets $Start,Accept$ of the set of vertices of this graph.
The \emph{language} $\calr_A\subset \fmi$ accepted by an automaton $A$ 
is the set of words labelled by directed paths starting from a vertex in $Start$
and finishing at a vertex in $Accept$.
A subset $\calr\subset \fmi$ is a \emph{regular language}
if $\calr=\calr_A$ for some automaton $A$.

If $\calr_1,\calr_2$ are regular languages, then so are $\calr_1\cup\calr_2$,
 $\calr_1\cdot\calr_2$ (the product set), and $\calr_1^*$ (the submonoid generated by $\calr_1$).
Kleene's theorem asserts that the class of regular languages is the smallest class containing finite subsets, and 
closed under these operations.\\

The class of regular languages is actually a Boolean algebra.
Moreover, one can algorithmically compute automata for  $\calr_1\cup\calr_2$,
  $\calr_1\cap\calr_2$,   $\fmi\setminus \calr_1$,
 $\calr_1\cdot\calr_2$ and $\calr_1^*$ from automata representing $\calr_1$ and $\calr_2$.
Similarly, if $\calr$ is rational, then so is its image $\ol \calr$ under the involution $w\mapsto \ol w$,
and the corresponding automaton can be computed.

\begin{example}\label{example_automate}
  Assume that $\pi:\fmi \ra F$ is a morphism to a finite group $F$, and consider a subset $E\subset F$.
Let $A$ be the Cayley graph of $F$ relative to $\pi(S)$.

Then $\pi\m(E)$ is a regular language, corresponding to the automaton $A$
where $Start=\{1\}$ and $Accept=E$.
\end{example}

Not all regular languages are of this form but this is true if one replaces the finite group $F$ by a finite monoid.
Let's be more specific.
Denote by $\calm_n$ the finite monoid of matrices with Boolean entries in $\{0,1\}=\{\mathtt{true,false}\}$, where the product of $C=AB$
of $A=(a_{ij})$ and $B=(b_{jk})$ is $(c_{ik})=(\bigvee_{j=1}^n a_{ij}\wedge b_{jk})$ and $\vee,\wedge$ denote $\mathtt{or}$ and $\mathtt{and}$.
Consider an automaton $A$ over $\Spm$ with vertices $v_1,\dots, v_n$. 
To each $s\in\Spm$ corresponds the matrix $M_s$ whose entry $(i,j)$ is $1$ if and only if
there is an edge labelled $s$ from $v_i$ to $v_j$.
The assignment $s\mapsto M_s$ extends to a morphism $\rho:\fmi\ra \calm_n$
and the language $\calr_A$ accepted by $A$ is the preimage under $\rho$ of the set of matrices
$\mu$ having a non-zero entry $(i,j)$ where $v_i\in Start$ and $v_j\in Accept$.
Conversely, given a morphism $\rho:\fmi\ra\calm$ to a finite monoid $\calm$ and a subset $\mu\subset \calm$,
$\rho\m(\mu)$ is a regular language corresponding to the following automaton: 
its vertex set is $\calm$, its edges labelled by $s$ are given by the multiplication by $\rho(s)$,
$Start=\{\id\}$ and $Accept=\mu$.
Clearly, one can compute $\rho$ and $\mu$ from an automaton, and conversely.

\begin{rem}\label{rem_deterministic}
  This (well known) construction implies the (well known) fact that one can make an automaton deterministic,
and in particular, assume that $Start$ is a single vertex of the automaton.
\end{rem}

\subsection{Rational subsets}

Consider a group $G$, $S$ a finite generating set,
and $\pi:\fmi\ra G$ the natural morphism mapping $S\cup \ol S$ to $S\cup S\m$.

\begin{dfn}
  A \emph{rational subset} of a group $G$ is the image under $\pi$ of a regular language of $\fmi$.

Equivalently, the class of rational subsets of $G$ is the smallest class
containing finite subsets, and 
closed under union, product, and submonoid generation.
\end{dfn}

The equivalence between the two definitions follows from Kleene's theorem.
In particular, the notion of rational subset does not depend on the choice of the generating set $S$. 
We say that a rational subset $\calr\subset G$ is \emph{represented} by an automaton $A$ 
over $\Spm$ when $\pi(\calr_A)=\calr$.

Note that if $\calr$ is rational, so is $\calr\m$.
Although the class of regular languages of $\fmi$ is a Boolean algebra,
the class of rational subsets of a group $G$ is not closed under intersection or complementation in general.
Indeed, it is easy to see that subgroups are rational subsets if and only if they are finitely generated,   
but it can happen that the intersection of two such subgroups is not finitely generated.

\begin{lem}\label{lem_morphism}
    Let $f:G\ra G'$ be a morphism, $S,S'$ be finite generating sets of $G,G'$, and $A$ an automaton on $S$ representing some
  rational subset $\calr\subset G$.

Then $f(\calr)$ is a rational subset of $G'$, and 
knowing an expression of $f(s)$ as an $S'$-word, one can compute an automaton for $f(\calr)$.
\end{lem}

\begin{proof}
For each $s\in S\cup \ol S$, write $f(s)=s'_1\dots s'_n$ where $s'_i\in S'\cup \ol S'$.
Replace each directed edge of $A$ labelled $s$ by a directed segment of length $n$ labelled $s'_1\dots s'_n$.
The automaton obtained in this way clearly represents $f(\calr)$.
\end{proof}

The following lemma follows immediately from the analogous fact concerning regular languages of $\fmi$.
It holds without assumption on the group $G$. 

\begin{lem}\label{lem_union}
Given two automata $A_1,A_2$ defining some rational subsets $\calr_1,\calr_2$ of $G$,
one can algorithmically compute some automata representing
 $\calr_1\cup\calr_2$,  $\calr_1\cdot\calr_2$, $\calr_1^*$ 
and $\calr_1\m$.
\end{lem}

When $G$ is a free group, or even a virtually free group,
the class of rational subsets is a Boolean algebra (see below). 
For the free group, this is based on the following fact.
We denote by $\grp{S}$ the free group with free basis $S$.

\begin{lem}[{\cite[Proposition 2.8 p59]{Berstel_transductions}}]\label{lem_reduit}
  For any rational subset $\calr$ in the free group $\grp{S}$, consider the set $\Tilde \calr\subset\fmi$
of reduced words representing elements of $\calr$.

Then $\Tilde\calr$ is a regular language, and an automaton representing $\Tilde\calr$ can be computed from an automaton representing $\calr$.
\end{lem}

\newcommand{\Red}{\mathrm{Red}}
It follows that the class of rational subsets of a free group 
is a Boolean algebra: denoting by $\Red$ the regular language of reduced words of $\fmi$, one has
$\pi(\Red \setminus\Tilde \calr)=\grp{S}\setminus\calr$
and $\pi(\Tilde \calr_1\cap\Tilde \calr_2)=\calr_1\cap\calr_2$.
Moreover, if $S$ is a basis of the free group $F$, 
automata over $\Spm$ representing the obtained rational subsets can be computed from automata over $\Spm$ for the initial ones.

\subsection{Rational subsets of a virtually free group}

When $G$ is virtually free, the fact that rational subsets form a Boolean algebra 
follows from recent work by Lohrey and Senizergues \cite{LoSe_rational} 
(their result is more general but quite complicated).  
We propose here a short proof of what we need.

\begin{rem}
  From a presentation of a virtually free group $V$, one can compute a free basis of a normal finite index free subgroup $F$ and
a representative  $g_a\in V$ of each element $a\in V/F$. Indeed, by the Reidemeister-Schreier process, one can enumerate all 
normal finite index subgroups of $V$, with presentations, and by Tietze transformations, all presentations of these subgroups.
One will find one that is a presentation with no relator, \ie a free basis of a free group $F$.
Then the finite group $V/F$ and lifts in $V$ are easy to compute.
\end{rem}

\begin{lem}\label{lem_chg_base}
  Let $V$ be a virtually free group given by some presentation with generating set $S_V$.
Let $F\normal V$ be a normal free subgroup of finite index,
given by a free basis $S$, expressed as a set of words on $S_V$.
Let $\calr$ be a rational subset of $F$ (hence of $V$).

Given an automaton over $(S_V\cup \ol S_V)$ representing $\calr$,
one can compute an automaton over $\Spm$ representing it, and vice-versa.
\end{lem}

\begin{proof}
Computing an automaton over $(S_V\cup \ol S_V)$ from one over $\Spm$ simply
follows from Lemma \ref{lem_morphism} applied to the embedding $F\ra V$.  

Let $\pi:(S_V\cup \ol S_V)^*\ra V$ be the natural map.
Let $A$ be an automaton over $S_V\cup \ol S_V$ accepting a language $\Tilde \calr\subset (S_V\cup \ol S_V)^*$
with $\pi(\Tilde\calr)=\calr$.
Let $V_A$ be its set of vertices, and assume that $Start$ is a single vertex (this can be assumed by Remark \ref{rem_deterministic}).
Denote by $\rho:(S_V\cup \ol S_V)^*\ra V/F$ the composition of $\pi$ with the quotient map. 

We construct a new automaton $A'$ with vertex set $V_{A'}=V_A\times V/F$.
We put an edge labelled $u\in S_V\cup \ol S_V$ from $(v_1,a)$ to $(v_2,b)$ if there is an edge labelled $u$ from $v_1$ to $v_2$ in $A$,
and $b=a \rho(u)$. We take $Start'=(Start,1)$ and $Accept'=Accept\times \{1\}$. 
Since $\calr\subset F$, the language accepted by $A'$ is precisely $\Tilde\calr$.

For each $a\in V/F$ choose a representative $g_a\in V$ with $g_a=1_V$ for $a=1$. We also think of $g_a$ as a word on $(S_V\cup \ol S_V)$,
and we note that such $g_a$ can be computed.
We construct an automaton $A''$ over $\Spm$ as follows.
Consider an edge $e$ labelled $u$ and joining $(v_1,a)$ to $(v_2,b)$.
Note that $u'=g_a u g_b\m \in F$, and write  $u'=s_1\dots s_k$ as a word on $\Spm$ (this can be done algorithmically).
Replace the edge $e$ by a directed segment labelled by $s_1\dots s_k$.
Do this for every edge of $A'$, take $Start''=Start'$ and $Accept''=Accept'$, and call $A''$ the obtained automaton.

We claim that the language $\Tilde\calr''$ accepted by $A''$ satisfies $\pi(\Tilde\calr'')=\calr$. 
Indeed, assume that  $u_1\dots u_n\in(S_V\cup \ol S_V)^*$ is accepted by $A'$,
and consider $a_i=\rho(u_1\dots u_i)$. One has $a_0=1$ since the path starts at $Start'$, and $a_n=1$ since $\pi(u_1\dots u_n)\in F$.  
Then the $\Spm$-word $u'_1\dots u'_n=(1 u_1 g_{a_1}\m)(g_{a_1} u_2 g_{a_2}\m)\dots (g_{a_{n-1}}u_n 1)$ 
is accepted by $A''$, and has the same image under $\pi$ as $u_1\dots u_n$. 
Similarly any word accepted by $A''$ maps to $\calr$.
\end{proof}

\begin{lem}\label{lem;rat_virt_free}
 Let $V$ be a virtually free group, and  
 $F\triangleleft V$ be a normal free 
  subgroup of finite index, 
  and let $g_1, \dots, g_k$ be representatives of the left cosets of $F$ in $V$.
  
  A subset $\calr\subset V$ is rational if and only if for all
  $i$, $(g_i^{-1} \calr )\cap F$ is rational in $F$.

Let $S_V$ be a generating set for $V$, and $S$ a basis of $F$. Given automata over $\Spm$ representing each $(g_i^{-1} \calr)\cap F$, 
one can explicitly compute an automaton over $S_V\cup \ol S_V$
  representing $\calr$, and conversely, given an automaton over $S_V\cup \ol S_V$
  representing $\calr$, and $g_1, \dots, g_k$, one can compute
  an automaton over $\Spm$ representing $(g_i^{-1} \calr )\cap F$.
\end{lem}

\begin{proof}
  If $(g_i^{-1} \calr )\cap F$ is a rational subset of $F$, it is rational in $V$.
Therefore, the sets $\calr \cap g_i F$ are rational in $V$, and so is their union.
Given an automaton over $\Spm$ representing $(g_i^{-1} \calr)\cap F$,
one can compute an automaton over $S_V\cup \ol S_V$ representing it by Lemma \ref{lem_chg_base},
hence an automaton over $S_V\cup \ol S_V$ representing $\calr$ by Lemma \ref{lem_union}.

Conversely, assume that $\calr$ is rational in $V$.
Denote by $\pi:(S_V\cup \ol S_V)^*\onto V$ the natural morphism,
and $\Tilde\calr\subset (S_V\cup \ol S_V)^*$ a rational language such that
$\calr=\pi(\Tilde\calr)$.
By Example \ref{example_automate}, $\pi\m(F)$ is a regular language,
and so is $\Tilde\calr\cap\pi\m(F)$.
Thus $\calr\cap F$ is rational,  and one can compute an automaton $A$ over $S_V\cup \ol S_V$ representing it. 
By Lemma \ref{lem_chg_base}, we can turn it into an automaton over $\Spm$.

\end{proof}

\begin{lem} \label{lem;Boolean} Let $V$ be a virtually free
  group, and $\calr,\calr'$ be rational subsets.

Then $V\setminus \calr$ and $\calr\cap\calr'$ are rational,
and one can compute automata representing them from automata representing $\calr$ and $\calr'$.
\end{lem}
        
\begin{proof}
       By Lemma \ref{lem;rat_virt_free}, one can write $\calr = g_1 \call_1
      \sqcup \dots \sqcup g_k \call_k$ with $\call_i$ rational in $F$. Then
      $V\setminus \calr = g_1 (F\setminus \call_1) \sqcup \dots \sqcup g_k
      (F\setminus \call_k)$ is rational since in the free group $F$, $(F\setminus \call_i)$ is rational.
      Moreover, one can compute automata for $\call_i$ by Lemma \ref{lem;rat_virt_free}, 
      hence for $F\setminus \call_i$, and for $V\setminus \calr$.
      
      For the second assertion, simply write $\calr\cap \calr' =(\calr^c \cup
      \calr'{}^c)^c$. 
\end{proof}

\section{Equations and twisted equations}\label{sec_eqn}

In this section, we start by giving a formal definition of a system of equations with rational constraints in a group $G$.
Then we explain how to translate a system of equations in a virtually free group
into a disjunction of systems of twisted equations in a free group.

\newcommand{\Xpm}{X_\pm}
\begin{dfn}[Equations with rational constraints]\label{dfn_equation}
Consider a finite set of variables $X$ 
and $\Xpm=X\dunion \ol X$ with the natural involution $x\mapsto \ol x$.

A \emph{system of equations} $\cale$ 
is a finite set of words $w=x_1\dots x_n$ in $\Xpm^*$ (representing the equation $x_1\dots x_n=1$).

A set of \emph{rational constraints} for $\cale$ is 
a tuple $\ul\calr=(\calr_x)_{x\in X}$ where for each $x\in X$, $\calr_x\subset G$ is  a rational subset.

A \emph{solution} of $(\cale,\ul\calr)$ is a tuple $\ul g=(g_x)_{x\in X}\in G^X$ 
with $g_x\in \calr_x$ for each $x\in X$, and such that 
 for each $w=x_1\dots x_n\in\cale$, $g_{x_1}\dots g_{x_n}=1$ in $G$
where we define $g_{\ol x}=g_x\m$ for each $x\in X$.
\end{dfn}

Abusing notations, if $(g_x)_{x\in X}$ is a solution of $\cale$, 
we will also call solution the corresponding tuple $(g_x)_{x\in X\cup\ol X}$ where $g_{\ol x}=g_x\m$ for each $x\in X$.

Constants in a system of equations can be encoded with rational constraints:
a constant $g_0\in G$ is a variable $x$ with corresponding constraint $\calr_x=\{g_0\}$.
If $G\setminus\{1\}$ is a rational subset, then inequations can be encoded with
rational constraints: 
replace each inequation $w\neq 1$ by an equation $w=y$ where $y$ is a new variable 
with rational constraint $\calr_y=G\setminus\{1\}$.

\subsection{Reduction to triangular equations}
\label{sec_triangular}

We fix $V$ a finitely generated virtually free group.
Classically, a system of equations with rational constraints is equivalent
to a triangular system (\ie in which every $w\in \cale$ has length at most $3$):
for each equation $x_1\dots x_n=1$ with $n\geq 4$, we add some new variables $y_2,\dots,y_{n-2}$ 
with no rational constraint on $y_i$ ($\calr_{y_i}=V$),
and we replace the equation $x_1\dots x_n=1$ by the equations $y_2=x_1x_2$, $y_3=y_2x_3$, \dots
and $y_{n-2}x_{n-1}x_n=1$.
One can also get rid of equations of length one by forgetting about the corresponding variable.

It is convenient to get rid of equations of length 2.
This can be done as follows, using the fact that the set of rational subsets 
of a virtually free group is a Boolean algebra (Lemma \ref{lem;Boolean}).
If we have an equation $x_1x_2=1$
where 
$x_1$ is distinct from $x_2,\ol x_2$ as a formal variable, 
one can substitute every occurrence of $x_2$ (resp. $\ol x_2$) with $\ol x_1$ (resp. $x_1$),
and change the rational constraint   
$x_1\in\calr_{x_1}$ to $x_1\in \calr_{x_1}\cap \calr_{x_2}\m$.
Of course, one can forget about equations of the form $x\ol x=1$.
Each equation of the form $x^2=1$ can be replaced by two equations $x^4=1$ and $x^6=1$,
which can be triangulated in the usual way.

 The transformation into a triangular system of equations
is algorithmic.

\subsection{Twisted equations}\label{sec_twisted}

Consider the general situation of a group $V$ containing a normal subgroup $G$ of finite index.
Let $\pi:V\ra Q=V/G$ the quotient map.
For each $q\in Q$, choose a lift $\Tilde q\in V$.
Given $x_1,x_2,x_3\in V$, let $q_i=\pi(x_i)$ and write $x_i=\Tilde q_i g_i$ for some $g_i\in G$.
Then $x_1,x_2,x_3$ satisfies the equation $x_1x_2x_3=1$ if and only if $q_1q_2q_3=1$, 
(so $\Tilde q_1\Tilde q_2\Tilde q_3\in G$), and
 $(\Tilde q_1\Tilde q_2\Tilde q_3)\, g_1^{\Tilde q_2\Tilde q_3}\, g_2^{\Tilde q_3}\, g_3=1$.
This last equation can be viewed as an equation with unknowns $g_1,g_2,g_3$, 
twisted by automorphisms of the form $g\mapsto g^{v}$ for some $v\in V$.

The group generated by these automorphisms has finite image in $\Out(G)$, but it usually fails to lift
to a finite subgroup of $\Aut(G)$.

In Proposition \ref{prop;trick_Out_Aut} below, we prove that when $G$ is a free group, we can embed $G$ in a larger free group $\grp{S}$
with free basis $S$,
so that the twisting automorphisms give rise to automorphisms of $\grp{S}$ that preserve $S\cup S\m$.
In particular, these automorphisms of $\grp{S}$ preserve word length, which will be of importance to us.
The price to pay is that we have to add rational constraints to our system of equations 
(even if there were no such constraints originally)
to guarantee that the solutions belong to the original free group.

Given a finite set $S$, and $\grp{S}$ the corresponding free group,
we denote by $\AutS$ the finite group of automorphisms of $\grp{S}$
preserving the set $\Spm=S\cup S\m$  (it has order $2^n n!$).

\begin{dfn}[Twisted equations] \label{def;twisted}
  A (triangular) \emph{system of twisted equations with rational constraints} $(\cale,\ul\calr)$ 
in a free group $\grp{S}$
consists of an alphabet of variables $X$, of a finite set $\cale$ of equations of the form 
$((\phi_1,x_1),(\phi_2,x_2),(\phi_3,x_3))\in(\AutS\times (X\cup\ol X))^3$
(representing the equation $\phi_1(x_1)\phi_2(x_2)\phi_3(x_3)=1$),
and a tuple $\ul\calr=(\calr_x)_{x\in X}$ of rational subsets of $\grp{S}$.

A solution of $\cale$ is a tuple $\ul g=(g_x)_{x\in X}\in \grp{S}^X$ 
with $g_x\in \calr_x$ for each $x\in X$, and such that
for each $((\phi_1,x_1),(\phi_2,x_2),(\phi_3,x_3))\in\cale$,
 $\phi_1(g_{x_1})\phi_2(g_{x_2})\phi_3(g_{x_3})=1$
where we define $g_{\ol x}=g_x\m$ for each $x\in X$.
\end{dfn}

\begin{rem} We opted for a restricted meaning of twisting. 
In general (and as presented in the introduction) a twisted equation is 
 an equation where some automorphisms are involved, but here, the only allowed automorphisms are those that preserve a basis of the given free group. 
This will be enough for our needs.
\end{rem}

The \emph{twisting subgroup} $\Phi\subset\AutS$ is the subgroup generated by the twisting morphisms involved in all twisted equations.
We say that $\Phi$ has no \emph{inversion} if for all $s\in S$ and $\phi\in\Phi$, $\phi(s)\neq s\m$.
Our construction from a virtually free group will produce only twisting subgroups without inversions.
However, if we are given a system of twisted equations where $\Phi$ has inversions, one can easily reduce to the case
without inversion using barycentric subdivision as follows. 
Replace $S=\{s_1,\dots,s_n\}$ by $S'=\{u_1,v_1,\dots,u_n,v_n\}$, and
embed $\grp{S}\ra \grp{S'}$ using $s_i\mapsto u_iv_i\m$.
If $\Phi(s_i)=s_j$, define $\Phi'(u_i)=u_j$ and $\Phi'(v_i)=v_j$, 
and if $\Phi(s_i)=s_j\m$, define $\Phi'(u_i)=v_j\m$ and $\Phi'(v_i)=u_j\m$.

\begin{prop}\label{prop;trick_Out_Aut}
The problem of solving equations with regular constraints in a virtually free group can be reduced
to the problem of solving twisted equations in a free group. 

More precisely,
let $(\cale,\ul \calr)$ be a system of equations with rational constraints 
in a virtually free group $V$.

Then there is a free group $\grp{S}$, and a finite set $(\cale'_1,\ul\calr'_1)\dots (\cale'_k,\ul\calr'_k)$ 
of systems of 
twisted equations (in the sense of Definition \ref{def;twisted}) with rational constraints on $\grp{S}$, with the same sets of unknowns $X$, such that
there is a natural bijection between the set of solutions of $\cale$ and the (disjoint) union of the sets of solutions
of $(\cale'_1,\ul\calr'_1),\dots,(\cale'_k,\ul\calr'_k)$.
The twisting subgroup has no inversion and is finite. 

Moreover, one can algorithmically compute $\grp{S}$ and $(\cale'_1,\ul\calr'_1)\dots (\cale'_k,\ul\calr'_k)$ from a presentation of $V$,
and from $(\cale,\ul\calr)$, where all rational subsets are represented by automata.
\end{prop}

\begin{rem}
The free group $\grp{S}$ does  not occur naturally as a subgroup of $V$. 
Only a subgroup of  $\grp{S}$ can be naturally identified as a subgroup of finite index of $V$. 
The natural bijection defined in the proof is clearly computable, but we won't need this fact.
\end{rem}

\begin{proof}
  Write $V$ as an extension $F\ra V\xra{\pi} Q$
where $Q$ is finite and $F=\ker\pi$ is a normal free subgroup of finite index of $V$.
Write $V$ as the fundamental group of a finite graph of finite groups, and
let $V\actson T$ be the dual Bass-Serre tree.
By adding a new edge to the graph of groups, we may assume that $T$ contains a point $*$ with trivial stabiliser ($T$ may fail to be minimal).

Consider the graph $\Gamma=T/F$, and still denote by $*$ the image of $*$ in $\Gamma$. 
Since $F$ acts freely, the quotient map $T\ra \Gamma$
is a covering map, and $F=\pi_1(\Gamma,*)$.
Note that $Q$ acts on $\Gamma$ by graph automorphisms (which may fail to fix $*$).

Consider a set $S$ of oriented edges of $\Gamma$  obtained by choosing an orientation for each edge,
and consider the free group $\grp{S}$. 
 Each edge path in $\Gamma$ defines an element of $\grp{S}$, and in particular, $F\subset \grp{S}$.
 Clearly, $Q$ acts on $\grp{S}$ via automorphisms of $\AutS$.
We denote by $\phi_q(x)$ the image of $x\in\grp{S}$  under the action of  $q\in Q$.

For $g\in V$, let $c_g$ be the edge path of $\Gamma$ obtained by projection of the segment $[*,g.*]\subset T$.
This defines a map $c:V\ra \grp{S}$.
Since the stabiliser of $*$ is trivial, $c$ is one-to-one.
Note that $c_g$ joins $*$ to $\phi_{\pi(g)}(*)$.
The map $c$ is not a morphism but
satisfies  $c(gh)=c(g).\phi_{\pi(g)}(c(h))$ and $c(g\m)=\phi_{\pi(g)\m}(c(g))$.
In particular, the restriction of $c$ to $F$ is a morphism. 

Each solution of $\cale$ in $V$ maps to a solution of $\cale$ in $Q$.
For each solution $\ul q\in Q^X$ of $\cale$ in $Q$, the set $\calr_x\cap \pi\m(q_x)$ is a rational subset by Lemma \ref{lem;Boolean}.
Let $\cals_{\ul q}\subset V^X$ be the set of solutions of $\cale$
with rational constraints $x\in\calr_x\cap \pi\m(q_x)$.
The set of solutions of $(\cale,\ul\calr)$ is the disjoint union of the sets $\cals_{\ul q}$.

Each equation in $\cale$ is equivalent to an equation of the form $x_1x_2x_3=1$ or $x_1x_2\ol x_3=1$
where $x_1,x_2,x_3$ lie in $X$ (not $\ol X$); we assume that each equation is of this form.
Given $\ul g\in \cals_q$, the equation  $g_{x_1}g_{x_2}g_{x_3}=1$
implies $c(g_{x_1}).\phi_{q_{x_1}}(c(g_{x_2})).\phi_{q_{x_1}q_{x_2}}(c(g_{x_3}))=1$.
Similarly, the equation $g_{x_1}g_{x_2}g_{x_3}\m=1$ implies 
 $c(g_{x_1}).\phi_{q_{x_1}}(c(g_{x_2})).c(g_{x_3})\m=1$.

Let $\cale'_{\ul q}$ be the set of equations in $\grp{S}$  obtained by replacing in $\cale$ each equation 
$x_1x_2x_3=1$ (resp. $x_1x_2\ol x_3$)
by the twisted  equation $x_1.\phi_{q_{x_1}}(x_2).\phi_{q_{x_1}q_{x_2}}(x_3)=1$
(resp. $x_1.\phi_{q_{x_1}}(x_2).\ol x_3=1$).
Thus $c$ maps $\cals_{\ul q}$ into the set $\cals'_{\ul q}$ of solutions of $(\cale'_{\ul q},\ul\calr')$
where $\calr'_{\ul q,x}= c(\calr_x\cap \pi\m(q_x))\subset\grp{S}$.
Let's check that  $\calr'_{\ul q,x}$ is a rational subset of $\grp{S}$.
If $q_x=1$, this follows from the fact that $c_{|F}$ is a morphism.
Otherwise, consider $\Tilde q_x\in\pi\m(\{q_x\})$, and consider the rational subset 
$\calr''=\Tilde q_x\m(\calr_x\cap \pi\m(q_x))\subset F$. Since $c_{|F}$ is a morphism, $c(\calr'')$ is rational and so is
$\calr'_{q,x}=c(\Tilde q_x).\phi_{q_x}(c(\calr''))$.

We claim that $c$ maps $\cals_{\ul q}$ onto  $\cals'_{\ul q}$.
Indeed, consider a solution $\ul a\in \cals'_{\ul q}$. For each $x\in X$,
since $a_x\in\calr'_{q,x}$, $a_x$ corresponds to a path joining $*$ to $q_x.*$.
This path lifts to a path in the tree $T$ joining $*$ to some $g_x.*$, where $g_x$ is uniquely determined because $*$ has trivial stabiliser.
By definition, $c(g_x)=a_x$.
Since $c$ realises a bijection between $\calr_x\cap \pi\m(q_x)$ and its image,
the constraint $g_x\in \calr_x\cap \pi\m(q_x)$ is satisfied.
Since for each equation $x_1x_2x_3\in\cale$, $\ul a$ 
satisfies the corresponding equation in $\cale'_{\ul q}$, namely
$a_{x_1}.\phi_{q_{x_1}}(a_{x_2}).\phi_{q_{x_1}q_{x_2}}(a_{x_3})=1$,
we get that $c(g_{x_1}g_{x_2}g_{x_3})=1$.
By injectivity of $c$, $g_{x_1}g_{x_2}g_{x_3}=1$. A similar argument applies to equations of the form $x_1x_2\ol x_3\in\cale$.

This proves that $c$ induces a bijection between the set of solutions of $(\cale,\ul\calr)$ and the (disjoint) union of the sets of
solutions of $(\cale'_{\ul q},\ul\calr'_{\ul q})$ as $\ul q$ ranges among solutions of $\cale$ in $Q^X$.

Let's prove the computability of the new system of equations.
Starting with a presentation of a virtually free group $V$, one can effectively find
a presentation of $V$ as a graph of finite groups $\Lambda$.  
Indeed, enumerating all presentations of $V$, one will find one which is visibly a
presentation of the correct form. More precisely, if $A$ is a finite group, 
\newcommand{\tab}{\mathrm{table}}
its finite presentation $\grp{A|\tab(A)}$ consisting of its multiplication table is a presentation from which finiteness of $A$ is obvious.
Consider an amalgam of two finite groups $G=A*_C B$, given by two monomorphisms  $j_A:C\ra A$, $j_B:C\ra B$.
Then $G$ has a presentation of the form $\grp{A\cup B\cup C | \tab(A),\tab(B),\tab(C),\{c=j_A(c)=j_B(c)\}_{c\in C}}$
from which one can obviously read that $j_A,j_B$ are injective morphisms, and that $G$ is an amalgam of finite groups.
Similarly, given a graph of finite groups $\Gamma$ and $\tau\subset\Gamma$ a maximal tree,
one produces a presentation of $\pi_1(\Gamma,\tau)$ from multiplication tables of vertex and edge groups
from which one can read that the corresponding group is the fundamental group of a graph of finite groups.

One can also find a normal free subgroup $F$ of finite index:
enumerate all morphisms to finite groups $\rho:V\ra Q$,
and check whether for each vertex group $\Lambda_v$, $\rho_{|\Lambda_v}$ is injective; one will eventually find such a $\rho$,
and one can take $F=\ker\rho$.
The graph $\Gamma=T/F$ is constructed from $\Lambda$ as the covering with deck group $Q$: 
for each edge or vertex $x$ of $\Lambda$, we put a copy of this edge or vertex for each element of $Q/\rho(\Lambda_x)$,
and the incidence relation is the natural one. 
Thus $\grp{S}$ is computable together with its natural basis and the action of $Q$.
The new systems of equations $\cale'_{\ul q}$ are now explicit.

There remains to compute the rational constraints $\calr'_{q,x}=c(\calr_{x}\cap\pi\m(q_x))$.
We saw that 
$\calr'_{q,x}=c(\Tilde q_x).\phi_{q_x}(c(\calr''))$ with $\calr''=\Tilde q_x\m(\calr_x\cap \pi\m(q_x))$.
An automaton representing $\calr''$ can be computed using Lemma \ref{lem;Boolean}.
By Lemma \ref{lem_morphism}, since $c_{|F}$ is a morphism, we get an automaton representing $c(\calr'')$,
and thus $\calr'_{q,x}$.
\end{proof}

\section{Band complexes}\label{sec_band_cplex}

We fix a finite set $S$, the free group $G=\grp{S}$,
and the corresponding free monoid $\fmi$ with involution $x\ra \ol x$.
Let $\AutS$ be the corresponding  finite group 
of automorphisms of $G$ preserving $S\cup S\m$,
which we identify with the group of automorphisms of $\fmi$ commuting with the involution.

\subsection{Band complexes}

A \emph{band complex} $\Sigma$ consists of a \emph{domain} $D$ together with a finite set of bands.
The domain $D$ is a disjoint union of non-degenerate segments.
We say that a segment is non-degenerate when it is  non-empty and not reduced to a point.
A \emph{band} $B$ consists of a rectangle $[a,b]\times [0,1]$, together with two injective continuous \emph{attaching} maps
$f_{B,0}:[a,b]\times\{0\}\ra D$ and $f_{B,1}:[a,b]\times\{1\}\ra D$ 
and  a \emph{twisting} automorphism $\phi_B\in \AutS$.
The segments $J_{B,\eps}=f_{B,\eps}([a,b]\times \{\eps\})$ for $\eps=0,1$ are called  \emph{the bases} of the band $B$.

\begin{rem*}
  Since bands carry twisting automorphisms, a  band complex might rather be called a \emph{twisted} band complex.
We won't use this terminology for shortness's sake.
\end{rem*}

We denote by $\Phi\subset\AutS$ the group generated by all twisting morphisms $\phi_B$, and we will assume that
$\Phi$ has no inversion, \ie that $\phi(s)\neq \ol s$ for all $\phi\in\Phi$ and $s\in \Spm$.

The set of \emph{vertices} of $\Sigma$ is the subset of $D$ consisting 
of the endpoints of $D$ together with the endpoints of the bases of the bands.
We identify $\Sigma$ with the CW-complex obtained by gluing the bands on $D$ using the maps $f_{B,\eps}$.

The precise value of the attaching maps are not important:
a band complex is really determined by the ordering of vertices in each component of $D$,
and for each band, the 4-tuple $(f_{B,0}(a),f_{B,0}(b),f_{B,1}(a),f_{B,1}(b))$ together with the twisting automorphism $\phi_B$.

\subsection{Solution of a band complex}\label{sec_solution_BC}

An \emph{elementary segment} of $\Sigma$ is a segment of $D$ joining two adjacent vertices.
We denote by $E(\Sigma)$ the set of oriented elementary segments of $\Sigma$.
We denote by $e\mapsto \ol e$ the orientation reversing involution on $E(\Sigma)$.

A \emph{labelling} $\sigma$ of $\Sigma$ is a labelling of each $e\in E(\Sigma)$
by a 
 non-empty word $\sigma_e\in \fmip$, so that $\sigma_{\ol e}=\ol{\sigma_e}$
(we use the standard notation $\fmip=\Spm\cdot\fmi=\fmi\setminus\{1\}$).
Say that a segment $I\subset D$ is \emph{adapted} to $\Sigma$ if it is a union of elementary segments.
For each oriented adapted segment $I=e_1\dots e_n\subset D$ written as a concatenation of elementary segments, 
we define $\sigma_I=\sigma_{e_1}\dots \sigma_{e_n}$.
Alternatively, we often view a labelling of $\Sigma$ as a subdivision of $D$, together with
a labelling of the subdivided edges by letters in $S\cup \ol S$.

A \emph{solution} $\sigma$ of $\Sigma$ is a labelling such that
for each band $B$ with twisting automorphism $\phi_B$ and with bases $J_0$, $J_1$ 
oriented so that the attaching maps preserve the orientation,
one has $\sigma_{J_1}=\phi_B(\sigma_{J_0})$.
We say that the solution  $\sigma$ is \emph{reduced} if 
for every non-degenerate adapted segment $I\subset D$, $\sigma_I$ is a non-trivial reduced word.

The \emph{length} $|\sigma|$ of a solution $\sigma$ is the sum of lengths of words $|\sigma_{D_1}|+\dots+|\sigma_{D_k}|$
where $D_1,\dots,D_k$ are the connected components of $D$ (with a choice of orientation).

\subsection{Rational constraints}

A set of \emph{rational constraints} on $\Sigma$ is a family of
regular languages 
 $\ul\calr=(\calr_e)_{e\in E(\Sigma)}$ of $\fmip$, indexed by oriented elementary segments of $\Sigma$, and 
such that $\calr_{\ol e}=\ol{\calr_{e}}$. 
A \emph{solution} $\sigma$ of $\Sigma$, satisfying the rational constraint $\ul\calr$ is a solution of $\Sigma$ such
 that for each elementary segment $e\in E(\Sigma)$,  $\sigma_e\in\calr_e$.

All the rational constraints we will use later will be in some \emph{standard form}
as defined in Section \ref{sec_standard}.

\subsection{From twisted equations to band complexes}\label{sec_eq2band}

\begin{prop}\label{prop;red_band_cmplex}
  Let $(\cale,\ul\calr)$ be a system of 
twisted equations with rational constraints on a free group $\grp{S}$.
Then one can effectively compute a finite set of band complexes with rational constraints $\Sigma_1,\dots,\Sigma_p$
and a bijection between the set of solutions of $(\cale,\ul\calr)$ and the disjoint union of the
(reduced) solutions of $\Sigma_1,\dots,\Sigma_p$.

Moreover, every solution of $\Sigma_i$ is reduced.
\end{prop}

\begin{proof}
Say that a tuple $\ul g\in G^X$ is \emph{singular} if $g_x=1$ for some $x\in X$.
Denote by $\Sol(\cale,\ul\calr)$ (resp. $\Sol^*(\cale,\ul\calr)$) 
the set of solutions (resp. of non-singular solutions) of $(\cale,\calr)$.
Clearly, $\Sol(\cale,\ul\calr)$ is 
 in bijection with the disjoint union
of $2^{\abs{X}}$ sets of the form $\Sol^*(\cale_{X_0},\ul\calr)$ 
where for $X_0\subset X$, $\cale_{X_0}$ is the system of 
twisted equations 
whose set of  unknowns is $X\setminus X_0$, and obtained from $\cale$ by replacing each occurrence of the variable $x\in X_0$
by $1$.

Some of the obtained equations may involve less than three variables.
However, if some equation involves just one variable, it is of the form $\phi(x)=1$, so $\cale$ has no non-singular solution 
and we may forget about this $X_0$.
We view the equations involving two variables as twisted equalities of the form $\phi_1(x_1)=\phi_2(x_2)$.

\begin{rem} 
 We could get rid of twisted equalities 
 involving two distinct formal variables by substitution,
and of twisted equalities of the form $\phi(x)=x$ by adding the rational constraint saying that $x\in\Fix\phi$,
but this 
does not allow to get rid of equations of the form $\phi(\ol x)=x$ since the set of fixed points of $x\mapsto \phi(\ol x)$ is
not a rational language in general, 
 and we don't want to add new singular variables.\end{rem}

We now need to build some band complexes encoding
the set of non-singular solutions of a given system of 
twisted equations $(\cale,\ul\calr)$ (including twisted equalities).

For each unknown $x\in X$, we consider an oriented segment $D_x$ (whose labelling will correspond to the value $g_x$ of the unknown $x$).
For $x\in \ol X$, we will denote by $D_{\ol x}$ the same segment as $D_x$, but with the opposite orientation.
Then for each 
twisted equation $\eps$ of the form $\phi_1(x_1)\phi_2(x_2)\phi_3(x_3)=1$
(resp. $\phi_1(x_1)=\phi_2(x_2)$), with $x_i\in X\cup \ol X$, we add three (resp.\ two) oriented segments $D_{\eps,i}$
corresponding to $\phi_i(x_i)$.
We define the domain $D$ of our band complex as $D=(\Dunion_{x\in X} D_x) \dunion (\Dunion_{(\eps,i)} D_{\eps,i})$.

For each  segment $D_{\eps,i}$ corresponding to $\phi_i(x_i)$,
we add a band whose bases are $J_0=D_{x_i}$ and $J_1=D_{\eps,i}$, preserving the orientation,
and whose twisting morphism is $\phi_i$.

For each twisted equality $\eps$ of the form $\phi_1(x_1)=\phi_2(x_2)$, we add a band 
whose bases are $J_0=D_{\eps,1}$ and $J_1=D_{\eps,2}$,
preserving the orientation, and with trivial twisting morphism.

For any non-singular solution $\ul g$ of the 
twisted equation $\eps\in\cale$ corresponding to 
$\phi_1(x_1)\phi_2(x_2)\phi_3(x_3)=1$,
one can consider its  cancellation tripod $\tau$:
this is the convex hull of $a_1=1,a_2=\phi_1(g_{x_1}),a_3=\phi_1(g_{x_1})\phi_2(g_{x_2})$ in the Cayley graph of $\grp{S}$,
so that $[a_i,a_{i+1}]$ is labelled by $\phi_i(g_{x_i})$ (where we view $i$ as an integer mod $3$).
Let $c$ be the centre of $\tau$, \ie the intersection of the three segments $[a_i,a_i+1]$.
There are four possible types of combinatorics of cancellation for each triangular equation $\eps$:
$c\notin\{a_1,a_2,a_3\}$ (in which cases $\tau$ is not a segment), or $c=a_1$, $c=a_2$, $c=a_3$. 
The four types are mutually exclusive because $\ul g$ is non-singular.

\begin{figure}[htb]
  \centering
\includegraphics{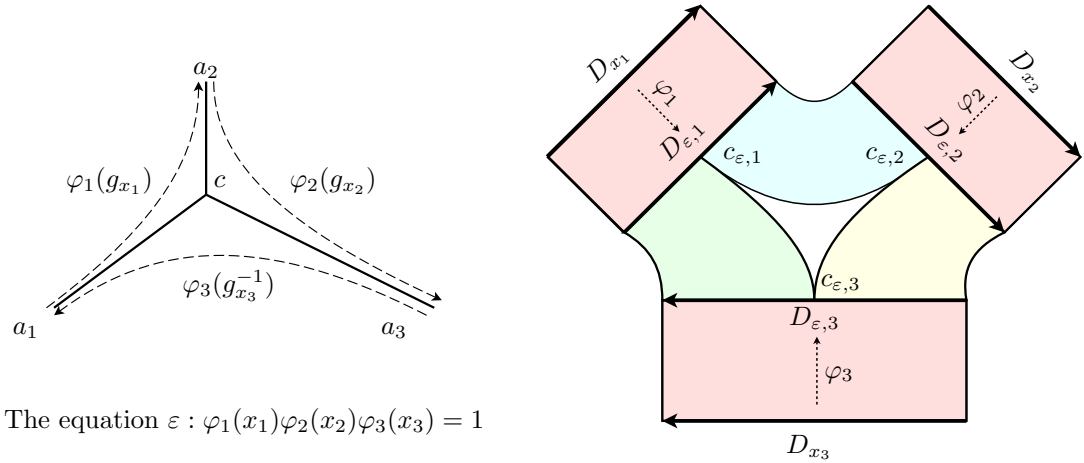}  
  \caption{The bands when the cancellation tripod is non-degenerate}
  \label{fig_bandcplex}
\end{figure}

To each choice of combinatorics of cancellation for this 
twisted equation, we associate a set of bands to be added to our band complex.
If the cancellation tripod is not a segment,
for each $i=1,\dots,3$, add a vertex $c_{\eps,i}$ at the midpoint of $D_{\eps,i}$ (Figure \ref{fig_bandcplex}).
Then add a band reversing  the orientation, whose bases are the right half-segment of $D_{\eps,i}$ 
and  the left half-segment of $D_{\eps,i+1}$ (right and left having a meaning relative to the orientation of those segments),
and whose twisting automorphism is the identity.

\begin{figure}[htbp]
  \centering
\includegraphics{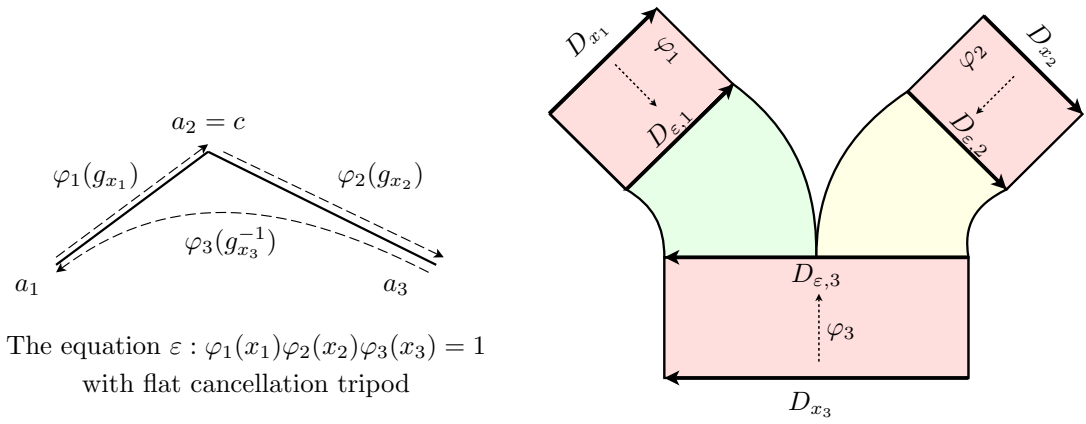}  
  \caption{The bands when the cancellation tripod is flat}
  \label{fig_bigon}
\end{figure}

If the cancellation tripod is such that $c=a_i$ for some $i\in\{1,2,3\}$, 
the product $\phi_{i-1}(g_{x_{i-1}})\phi_{i}(g_{x_i})$ is reduced, $i$ being thought modulo 3 
(Figure \ref{fig_bigon}). 
We add one vertex at the midpoint of $D_{\eps,i+1}$, and
two bands reversing   the orientation, whose 
 bases are $J_0=D_{\eps,i}$ (resp. $J_0=D_{\eps,i-1}$) 
and $J_1$ the 
 initial (resp. final) half-segment of $D_{\eps,i+1}$, 
and whose twisting automorphism is the identity.

Let $t$ be the number of triangular equations of $\cale$.
For each choice $\kappa$ out of the $4^t$ possible combinatorics of cancellation,
we obtain a band complex $\Sigma_\kappa$.

Finally, for each rational subset $\calr_x\subset \grp{S}$,
we consider $\Tilde\calr_x\subset\fmi$ the set of reduced words representing elements of $\calr_x$.
This is a rational language by Lemma \ref{lem_reduit}.
We add to $\Sigma_\kappa$ the rational constraint
 $\Tilde\calr_x$ on $D_x$ for each $x\in X$, and we don't put any constraint 
 on the other elementary segments of $\Sigma$ (\ie we set $\calr_e=\fmi$).
Since $\Tilde\calr_x$ is a language of reduced words, any solution of $\Sigma_\kappa$ is reduced.
\\

We claim that the set of reduced solutions of $\Sigma_\kappa$ is in one-to-one correspondence
with the subset $\cals_\kappa$ of $\Sol^*(\cale,\ul\calr)$ whose combinatorics of cancellation correspond to $\kappa$.

Indeed, a non-singular solution $\ul g\in \cals_\kappa$ defines a labelling $\sigma$ of $\Sigma_\kappa$ as follows:
$D_x$ is labelled by the reduced word $\sigma_x\in\fmi$ representing $g_x$,
each $D_{\eps,i}$ representing $\phi_i(x_i)$ is labelled by the reduced word representing $\phi_i(g_{x_i})$,
and if the midpoint of $D_{\eps,i}$ is a vertex of $\Sigma_\kappa$,
the labelling of the two half segments of $D_{\eps,i}$
should be such that the position of this vertex corresponds
to the centre of the cancellation tripod.
The labelling $\sigma$ thus constructed is a clearly reduced solution of $\Sigma_\kappa$.

Conversely, if $\sigma$ is a reduced solution of $\Sigma_\kappa$, the image $g_x$ in $\grp{S}$ of the word $\sigma_{D_x}$ 
clearly defines a non-singular solution of $(\cale,\ul\calr)$.
Moreover,  for each $\eps\in\cale_T$, the label of $D_{\eps,i}$ defines a geodesic segment, 
and the midpoint of the three segments $D_{\eps,i}$ corresponds to the centre of the cancellation tripod for $\eps$.
Thus, the combinatorics of cancellation agree with $\kappa$, and $g_x\in\cals_\kappa$.

 The construction of the band complexes $\Sigma_\kappa$ is clearly effective, so Proposition \ref{prop;red_band_cmplex} follows.
\end{proof}

\subsection{Standard forms of rational constraints}\label{sec_standard}

It will be convenient to represent all the needed rational languages in a fixed finite monoid.
This will allow to treat uniformly all rational languages appearing under various transformations
of the band complexes.

\begin{lem}\label{lem_monoid}
Let $\calr_1,\dots,\calr_k$ be a finite set of rational languages in $\fmip$.

There exists a finite monoid $\calm$ with an involution $m\mapsto \ol m$ 
and with an action of $\AutS$, and an $\AutS$-equivariant morphism $\rho:\fmip\ra\calm$ commuting
with the involutions, such that
each $\calr_i$ is the preimage of a finite subset of $\calm$.

Moreover, all this data is algorithmically computable from automata representing $\calr_i$.
\end{lem}

When $\calr$ is the preimage under $\rho$ of some subset of $\calm$ as above,
we say that $\calr$ is \emph{represented} by $\rho$.

\begin{proof}
We first consider the case of a single regular language.
Let $\calr\subset\fmip$ be a rational language.
Let $A$ be an automaton representing $L$ with vertex set $v_1,\dots v_n$.
For each $s\in S\cup \ol S$, 
consider the $n\times n$ Boolean incidence matrix $M_s$
of the subgraph of $A$ whose edges are those labelled by $s$.

Let $\calm_0$ be the semigroup of $n\times n$-Boolean matrices,
and $\rho_0:\fmip\ra\calm_0$ the morphism sending $s$ to $M_s$.
Let $v_S$ (resp. $v_A$) be the Boolean vector representing $Start$ (resp. $Accept$).
Then $\calr$ is the preimage under $\rho$ of the set $A_0\subset \calm_0$ of matrices $M$ such that
${}^tv_S M v_A=1$.

Now let $\ol\rho_0:\fmip\ra\calm_0$ be the morphism sending $s$ to ${}^tM_{\ol s}$.
Note that $\ol\rho_0(w)={}^{t}\rho_0(\ol{w})$ for all $w\in \fmip$.
Consider the finite monoid $\calm=\calm_0\times\calm_0$ endowed with the involution
$(M_1,M_2)\mapsto \ol{(M_1,M_2)}\overset{\mathrm{def}}{=}({}^{t}M_2,{}^{t}M_1)$, and consider
$\rho:\fmip\ra\calm$ sending $s$ to $(\rho_0(s),\ol\rho_0(s))$.
By the remark above, $\rho$ commutes with the involutions of $\fmip$ and $\calm$,
and $\calr=\rho\m(A_0\times\calm_0)$. 

Now given a finite set of languages $\calr_1,\dots,\calr_k$, consider finite monoids with involutions
$\calm_1,\dots,\calm_k$ and $\rho_i:\fmip\ra\calm_i$ commuting with the involutions
such that $\calr_i=\rho_i\m(A_i)$ for some $A_i\subset \calm_i$.
Consider the finite monoid $\calm=\calm_1\times\dots\times\calm_k$  with the diagonal involution, 
and $\rho=\rho_1\times\dots\times \rho_k$. Then for $A'_i=\calm_1\times\cdots\times \calm_{i-1}\times A_i\times\calm_{i+1}\times \cdot\times\calm_k$,
 $\calr_i=\rho\m(A'_i)$.

Finally, we need to put the $\AutS$-action in the picture.
For each $\alpha\in\AutS$, let $\calm_\alpha$ be a copy of $\calm$, and consider $\rho_\alpha=\rho\circ\alpha:\fmip\ra \calm_\alpha$.
Consider the product $\Tilde\calm=\prod_{\alpha\in\AutS}\calm_\alpha$ endowed with the action of $\AutS$ by permutation of factors,
 and $\Tilde\rho=\prod_{\alpha\in\AutS}\rho_\alpha$.
Then $\Tilde\rho:\fmip\ra \Tilde\calm$ satisfies the lemma.
\end{proof}

By Lemma \ref{lem_monoid}, one can represent all the rational constraints $\calr_e$ of a given band complex 
by a single morphism $\rho:\fmip\ra\calm$.
It will be useful to assume that $\rho$ also represents
$\Fix\phi\subset\fmip$ for all $\phi\in\Phi$ (this is a regular language as it is the submonoid generated by
$\Phi$-invariant letters in $\Spm$).
By a finite disjunction of cases, we  
need only to consider rational constraints of the form $\calr_e=\rho\m(\{m_e\})$.

\begin{dfn}[Standard form for rational constraints]
A set of rational constraints in \emph{standard form} on a band complex $\Sigma$ consists of
\begin{enumerate}
\item a finite monoid $\calm$ with an involution and an action of $\AutS$,
\item an epimorphism $\rho:\fmip\ra\calm$ commuting with the involutions and $\AutS$-equivariant,
such that for all $\phi\in\AutS$, the rational language $\Fix\phi\subset\fmip$ is represented by $\rho$,
\item a collection $\ul m=(m_e)_{e\in E(\Sigma)}\in\calm^{E(\sigma)}$ such that $m_{\ol e}=\ol{m_e}$.
\end{enumerate}
\end{dfn}

The tuple $\ul m=(m_e)$ defines a tuple of rational constraints by $\calr_e=\rho\m(\{m_e\})$.
By definition, a solution of $(\Sigma,\ul m)$ is a solution of $\Sigma$ with the corresponding rational constraints.
Using  Lemma \ref{lem_monoid}, we immediately get the following lemma.
The fact that $\rho$ is onto can be obtained by changing $\calm$ to $\rho(\fmip)$.

\begin{lem}
  Given a band complex $\Sigma$ with rational constraints $\ul \calr$, one can compute
 a finite monoid $\calm$, a morphism $\rho:\fmip\ra\calm$, and some tuples
$\ul m_1,\dots,\ul m_p\in\calm^{E(\Sigma)}$ defining rational constraints in standard form on $\Sigma$
 such that the set of solutions of $(\Sigma,\ul\calr)$ is the 
disjoint union of the sets of solutions of $(\Sigma,\ul m_1),\dots(\Sigma,\ul m_p)$.
\end{lem}

From now on, all band complexes are endowed with rational constraints in standard form.

\section{Bulitko's Lemma: bounding the exponent of periodicity}\label{sec_bulitko}

The goal of this section is a version of Bulitko's lemma \cite{Bulitko_equations} 
adapted to  twisted equations, which we state in terms of band complexes.
The \emph{exponent of periodicity} of a solution $\sigma$ of a band complex $\Sigma$ is the largest integer $s$ such that
some subsegment of the domain $D$ of $\Sigma$ is labelled by some word of the form $p^s$, with $p$ a non-trivial word.
Recall that the length of a solution $\sigma$ of a band complex is the number of letters in $\Spm$ involved in the labelling.

\begin{prop}[\cite{Bulitko_equations}]\label{prop_Bulitko} 
  Let $\Sigma$ be a band complex with rational constraints $\ul m$ in standard form.

There exists a computable number $B(\Sigma,\ul m)$ such that 
any shortest solution of $\Sigma$ has an exponent of periodicity at most $B(\Sigma,\ul m)$.
\end{prop}

We will use and prove the proposition only for a band complex arising from a system of equations
as in Section \ref{sec_eq2band}. More precisely, we assume that every component $D_0$ of the domain of the band complex
contains at most one vertex in its interior, and at most two bases of bands distinct from $D_0$ (see Figure \ref{fig_bandcplex} and \ref{fig_bigon}).
The general statement can be easily deduced from this case.

\subsection{Normal decomposition with respect to powers}

Consider a solution $\sigma$ of the band complex $\Sigma$.
We view $\sigma$ as a subdivision of $D$ together with a labelling of the edges by elements of $\Spm{}$.
Each interval $A\subset D$ we consider is a union of such edges.
Its length $\abs{A}$ is the number of edges it contains, which coincides with
the length of the word $\sigma_A$ labelled by $A$.

We fix some $q\in\bbN\setminus \{0\}$. Given some interval $A\subset D$,
we are going to define a natural set of \emph{special segments}
encoding its large enough subwords which are $q$-periodic.
Say that an interval $A\subset D$ of length $n>q$ is \emph{$q$-periodic}
if $\sigma_A=a_1\dots a_n$ with $a_i\in\Spm{}$ and for all $i\leq n-q$,
$a_i=a_{i+q}$.

For each maximal $q$-periodic segment $P\subset A$ of length at least $2q$,
we say that its central subsegment $Q\subset P$ of length $\abs{P}-2q$ is a \emph{special segment} of $A$.
A special segment is reduced to a vertex when $\abs{P}=2q$.
Although maximal $q$-periodic segments may overlap, the following claim holds.

\begin{claim}
  Two distinct $q$-special segments of $A$ are disjoint, and separated by a segment of length at least $q+1$.
\end{claim}

\begin{proof}
  Let $Q,Q'\subset A$ be two special segments, and $P\supset Q$ and $P'\supset Q'$ the corresponding maximal $q$-periodic segments.
If the claim fails, then $P\cap P'$ has length at least $q$. 
Thus, two edges of $Q\cup Q'$ at distance $q$ of each other either both lie in $P$ or in $P'$,
so  $P\cup P'$ is $q$-periodic.
\end{proof}

For any word $p$ of length $q$, and $r=n/q\in \frac1q\bbN$, it will be convenient to denote by $p^r$
the prefix of length $n$ of 
a sufficiently large power of $p$. If $r\in\bbN$, then $p^r$ coincides with the usual $r$-th power of $p$.
For each special segment $Q$ we denote by $r(Q)=\abs{Q}/q\in \frac1q\bbN$
 the exponent and by $p(Q)$ the period so that $Q$ is labelled by $p(Q)^{r(Q)}$ (if $r(Q)<1$,
$p(Q)$ should be defined using the maximal $q$-periodic subsegment $P$ containing $Q$).

The  
\emph{normal decomposition} of $A\subset D$ (with respect to $q$ and to the solution $\sigma$) is 
the decomposition $A=U_0.Q_1.U_1\dots Q_{n}.U_{n}$  defined by its special segments $Q_1,\dots,Q_{n}$.
Note that if the decomposition is non-trivial (\ie if $n\geq 1$), then 
$\abs{U_0}\geq q$ since the suffix of length $q$ of $U_0$ is $p(Q_1)$.
If $\phi\in\AutS{}$, and if $\sigma_{A'}=\phi(\sigma_A)$ (resp. if $\sigma_{A'}=\sigma_A\m$), 
then the normal decompositions match (up to changing the orientation).

\subsection{Special segments and concatenation}

Assume that $A,A',A''$ are oriented intervals of $D$ such that $A=A'.A''$.
Any special segment $Q'$ of $A'$ 
 (or $A''$) is clearly contained in a special segment $Q$ of $A$.
If $Q=Q'$, we say that $Q'$ is \emph{regular}, and \emph{exceptional} otherwise. 
We say that a special segment $Q$ of $A$ is \emph{regular} if it 
is contained in $A'$ or $A''$ and is a special segment in it.

Let $A'=U'_0.Q'_1.U'_1\dots Q'_{n'}.U'_{n'}$ and $A''=U''_0.Q''_1.U''_1\dots Q''_{n''}.U''_{n''}$
be the normal decomposition of $A'$ and $A''$.
If $Q'$ (resp. $Q''$) is an exceptional segment of $A'$ (resp. $A''$), then $Q'=Q'_{n'}$ (resp $Q''=Q''_1$)
and by maximality of special segments, 
 $\abs{U'_{n'}}=q$ (resp. $\abs{U''_0}=q$).

Let $Q$ be an exceptional special segment of $A$.
We distinguish several cases.
\begin{enumerate}
\item If $Q$ contains two special segments of $A'$ or $A''$, then $Q=Q'_{n'}. U'_{n'}.U''_0. Q''_0$ so
$r(Q)=r(Q'_{n'})+r(Q''_0)+2$.
\item 
 Otherwise, if $Q$ contains exactly one special segment of $A$ or $A'$, and if this special segment lies in $A'$,
then $Q\supset Q'_{n'}$ and $\abs{Q\cap A''}< q$ since otherwise, $Q\cap A''$ would contain a special segment of $A''$.
Thus, $r(Q)< r(Q'_{n'})+2$.
\item Symmetrically, if $Q$ contains exactly one special segment, lying in $A''$, 
then $Q\supset Q''_{0}$ and $r(Q)< r(Q''_{0})+2$.
\item If $Q$ contains no special segment of $A'$ or $A''$, then $\abs{Q\cap A'}<q$ and $\abs{Q\cap A''}<q$ 
so $r(Q)< 2$.
\end{enumerate}
There can be at most $2$ exceptional segments in $A$ because two special segments are at least
at distance $q+1$ apart. It may happen that there are two exceptional segments:
for $q=5$, consider the word $p=ababa$, $A'$ labelled by $cpaba$ and $A''$ by $babapc$ so that $A'$ and $A''$ contain
no special segment; then $A'.A''$ is labelled by $cp^2bapc=cpabp^2c$ and has two special segments of length $0$.

\subsection{Proof of Bulitko's Lemma}

We are now ready for the proof of Proposition \ref{prop_Bulitko}.
Once normal forms have been adapted to twisted equations, the proof is identical to the one given in \cite{DGH_pspace}.
We give it here for the reader's convenience.

\begin{proof}

Start with a shortest solution $\sigma$ of $\Sigma$ with presumably high exponent of periodicity,
\ie containing high powers of a word of length $q$.
For each special segment $Q$, recall that $p(Q)$ denotes the word labelled by $Q$,
and $r(Q)=\abs{Q}/q\in \frac 1q \bbN$.
We also define $t(Q)\in \bbN$ and $u(Q)\in\{0,\frac 1q,\dots,\frac{q-1}q\}$
as the integer and fractional part of $r(Q)$ so that $r(Q)=t(Q)+u(Q)$.

The structure of the proof is as follows.
We view the $t(Q)$'s as variables satisfying a system of integer equations $\Lambda$.
Each solution of $\Lambda$ provides a new solution of $\Sigma$, and the fact that $\sigma$ is shortest
says that this solution of $\Lambda$ is minimal.
The system of equations happens to be equivalent to a system of equations whose 
 number of equations, variables, and coefficients are bounded independently of $\sigma$, 
thus giving a bound on the size of a minimal solution of $\Lambda$,
and thus on the exponents of special segments.

Since we assume that our band complex $\Sigma$ comes from the construction in section \ref{sec_eq2band},
each component $D_0$ of its domain $D$ contains at most one vertex in its interior,
and at most two bases of bands distinct from $D_0$.
We first forget about rational constraints.

Let $\calb$ be the set of bases of bands of $\Sigma$,
with a chosen orientation.
For each $A\in\calb$, let $A=U_{A,0}.Q_{A,1}.U_{A,1}\dots Q_{A,n_A}.U_{A,n_A}$ be its normal decomposition.
Given a set of values $(\tau_{A,i})\in \bbN^{n_A}$,
we define a new word by changing the integer part the exponents of the special segments of $A$:

$$ w_A(\tau_{A,1},\dots,\tau_{A,n}) = \sigma_{U_{A,0}}. \left(   p(Q_{A,1})^{\tau_{A,1}+u(Q_{A,1})} \right)  .\sigma_{U_{A,1}}\dots 
  \left(  p(Q_{A,1})^{\tau_{A,n_A}+u(Q_{A,n_A})} \right) .\sigma_{U_{A,n_A}}$$

By definition, $w_A(t(Q_{A,1}),\dots,t(Q_{A,n_A}))=\sigma_A$.

We now define a system of equations $\Lambda$ with unknowns $(\tau_{A,k})\in \oplus_{A\in \calb} \bbN^{n_A}$.
By construction, $\tau_{A,k}=t(Q_{A,k})$ will be a solution of $\Lambda$.
For each band with bases $A_0,A_1$ with twisting morphism $\phi$, we have $\sigma_{A_1}=\phi(\sigma_{A_0})$,
$n_{A_1}=n_{A_0}$, and we add the equations 
$\tau_{A_1,k}=\tau_{A_2,k}$ for $k\in\{1,\dots, n_{A_1}\}$.
For each connected component $A$ of $D$ containing a vertex distinct from its endpoints,
decomposing $A$ into $A'.A''$, 
we add to $\Lambda$ an equation $\tau_{A',k}=\tau_{A,k}$ for each regular special segment of $A'$,
and an equation $\tau_{A'',n_{A''}-k}=\tau_{A,n_{A}-k}$ for each regular special segment of $A''$.

If $Q$ is a exceptional segment of $A$ containing two exceptional segments $Q'\subset A'$
and $Q''\subset A''$, then $r(Q)=r(Q')+r(Q'')+2$,
so either $t(Q)=t(Q')+t(Q'')+2$
or $t(Q)=t(Q')+t(Q'')+3$; we add to $\Lambda$ the corresponding equation $\tau_{A,i}=\tau_{A',n_{A'}}+\tau_{A'',1}+2$
or $\tau_{A,n_x}=\tau_{A',n_{A'}}+\tau_{A'',1}+3$.

If $Q$ contains exactly one special segment $Q'$, lying in $A'$, 
then   $r(Q)<r(Q')+2$,
so $t(Q)=t(Q')+\eps$ for some $\eps\in\{0,1,2\}$; we add to $\Lambda$ the equation $\tau_{A,k}=\tau_{A',n_{A'}}+\eps$
where $k$ is the index corresponding to $Q$ in $A$.
We proceed symmetrically if $Q$ contains exactly one special segment in $A''$.

If $Q$ contains no special segment of $A'$ or $A''$,
then $r(Q)<2$ so $t(Q)=\eps$ for some $\eps\in\{0,1\}$, and we add to $\Lambda$ the equation
$\tau_{A,k}=\eps$ where $k$ is the index corresponding to $Q$ in $A$.

This defines $\Lambda$. This system has a solution, namely $\tau_{A,k}=t(Q_{A,k})$.
Moreover, if $(\tau_{A,k})$ is any solution of $\Lambda$, then 
 $w_A(\tau_{A,1},\dots,\tau_{A,n_A})$ defines a solution of $\Sigma$.
In particular, if $t(Q_{A,k})$ is not a minimal solution of $\Lambda$,
meaning that there exists a solution of $\Lambda$ with $\tau_{A,k}\leq t(Q_{A,k})$
for all $A,k$, and some inequality is strict,
then $\sigma$ is not the shortest solution.

The number of equations and of unknowns of $\Lambda$ is not bounded a priori.
However, there are at most $2\#\pi_0(D)$ equations which are not equalities of the type $\tau_{A,k}=\tau_{A',k'}$.
By substitution, one can get rid of these equalities, and we are left with a system $\Lambda'$ of
at most $2\#\pi_0(D)$ equations, with coefficients in $\{0,\dots,3\}$, and with $l\leq 6\#\pi_0(D)$ unknowns.
Note that the bounds are independent of $\sigma$ and of $q$.

Knowing only the band complex $\Sigma$ (and not the solution $\sigma$), 
one can list all possible systems $\Lambda'$,
decide which of them have a solution in $\bbN^l$, 
list all minimal solutions of each of them (see for instance \cite[Th 17.1]{Schrijver}) 
and compute the maximal value $M$ of any $\tau_{A,k}$ in any such minimal solution.
The value $M$ is then an upper bound on the exponent of periodicity of
a shortest solution $\sigma$ of $\Sigma$.

In presence of rational constraints, represented by $\rho:\fmip\ra \calm$, 
we first compute $\alpha,\beta\in\bbN\setminus\{0\}$ such that for all $m_1,m_2,m_3\in\calm$ and all $\tau\in\bbN$
$m_1m_2^\beta m_3=m_1m_2^{\alpha\tau+\beta}m_3$ (it is an easy standard fact that such $\alpha,\beta$ exist and are computable).
In particular, $u_1u_2^\beta u_3\in\calr_e$ if and only if $u_1u_2^{\alpha\tau+\beta}u_3\in \calr_e$.
Then we modify the system $\Lambda$ as follows: for each $A\in\calb$ and $k\in\{1,\dots,n_A\}$ such that 
$t(Q_{A,k})\leq \beta$, replace the variable $\tau_{A,k}$ by the constant $t(Q_{A,k})$;
for each $t(Q_{A,k})> \beta$,  write by Euclidean division $t(Q_{A,k})=\beta+\alpha t'_{A,k}+ \eps_{A,k}$ for some $\eps_{A,k}<\alpha$ and $t'_{A,k}\in\bbN$;
then replace the unknown $\tau_{A,k}$ by $\tau'_{A,k}$ using the identity  $\tau_{A,k}=\alpha\tau'_{A,k}+\beta+\eps_{A,k}$.
We get a linear system of integer equations with unknowns $\tau'_{A,k}$, with at most 
$2\#\pi_0(D)$ equations, with coefficients bounded in terms of $\alpha$ and $\beta$, and with $l\leq 6\#\pi_0(D)$ unknowns.
As above, one can compute a bound $M'$ on all minimal solutions of all such systems which have a solution.
The same argument as above shows that  $\alpha M'+\beta$ is an upper bound the exponent of periodicity of
the shortest solution of $\Sigma$, including rational constraints.
\end{proof}

\begin{rem*}
Using estimates on the size of minimal solutions of linear systems of Diophantine equations,
and being more careful in the analysis, one can get an explicit bound of $M=2^{O(\#\pi_0(D))+n\log n}$ where $n$ is the size of
an automaton accepting the languages of $\ul\calr$, see \cite{DGH_pspace}.
\end{rem*}

\section{Prelaminations}\label{sec_prelamin}

\subsection{Definition} \label{subsec_prelamination}

\begin{dfn}
A \emph{prelamination} $\call$ on a band complex $\Sigma$ consists in
 \begin{itemize}
 \item  
 A finite set $V(\call)\subset D$ containing all  vertices of $\Sigma$.
 Points of $V(\call)$ are called \emph{$\call$-subdivision points}. 
 \item 
 For each band $B=[a,b]\times [0,1]$, a set $L_B(\call)\subset [a,b]\times [0,1]$ 
of the form $\{c_1,\dots, c_r\}\times [0,1]$. Each segment $\{c_i\}\times[0,1]$ is called a \emph{leaf segment}.
The endpoints of each leaf segment are asked to lie in $V(\call)$: $f_{B,\eps}(c_i,\eps)\in V(\call)$ for $\eps=0,1$.
We also require that both segments $\{a\}\times [0,1]$ and $\{b\}\times [0,1]$ lie in $L_B(\call)$.
 \end{itemize}
\end{dfn}

See Figure \ref{fig_prelamination} for an illustration.

 \begin{rem}
In this definition, the precise value of the attaching maps $f_{B,\eps}$ are not important, except at endpoints of leaf segments.
Moreover, the precise position of the $\call$-subdivision points in $D$ is unimportant, only the induced ordering matters.
Thus, a prelamination of a band complex is really determined by the ordering of the elements of $V(\call)$ in each component of $D$,
and for each band, by a bijection preserving or reversing the ordering, 
between a subset of $V(\call)\cap J_0$ and a subset of $V(\call)\cap J_1$.
\end{rem}

We also view leaf segments of $\call$ as subsegments of the topological realisation $\abs{\Sigma}$ of $\Sigma$.
A \emph{leaf} of $\call$ is a connected component of the union of the leaf segments
 in $\abs{\Sigma}$.
A \emph{leaf path} is a concatenation of
oriented leaf segments (which defines a path contained in a leaf).

A leaf segment $\{c\}\times [0,1]\subset [a,b]\times[0,1]$ is \emph{singular} if $c\in\{a,b\}$.
A leaf path is \emph{regular} if it consists of non-singular leaf segments.
A leaf is \emph{singular} if it contains a singular leaf segment, \ie if it contains a vertex of $\Sigma$.

An \emph{elementary segment} of $\call$ is a segment of $D$ joining two adjacent $\call$-subdivision points.
The set of oriented elementary segments of $\call$ is denoted by $E(\call)$.
We say that $J\subset D$ is \emph{adapted to $\call$} if it is a union of elementary segments of $\call$.

A point $x\in V(\call)$ is \emph{not an orphan} if 
for every base $J_\eps$ of any band $B$, with $x\in J_\eps$, there is a leaf segment of $B$ having $x$ as an endpoint:
$x\in f_{B,\eps}(L_B(\call))$. On Figure \ref{fig_prelamination}, orphans are represented as thick dots.
The prelamination is \emph{complete} if no $\call$-subdivision point is an \emph{orphan}
(in this case, the prelamination could be called a genuine \emph{lamination}).

If $B=[a,b]\times [0,1]$ is a band with attaching maps $f_{B,0}$, $f_{B,1}$, 
we define the \emph{opposite} $B\m$ of $B$ as the band $[a,b]\times[0,1]$ with attaching maps $f_{B\m,i}=f_{B,1-i}$ for $i=0,1$,
with twisting morphism $\phi_{B\m}=\phi_{B}\m$.
An \emph{oriented band} is a band of $\Sigma$ or its opposite.
The leaf segments of $B$ define leaf segments of $B\m$.
We orient each leaf segment $\{x\}\times [0,1]$ of the oriented band $B$ from $(x,0)$ to $(x,1)$
so that leaf segments of $B$ and $B\m$ have the opposite orientation.
If $B$ is an oriented band with bases $J_0,J_1$, we denote by $\dom B=J_0$ its initial base,  and we call it its \emph{domain}.

A \emph{$\Sigma$-word} is a word on the alphabet of oriented bands of $\Sigma$.
The twisting morphism of the $\Sigma$-word $w=B_1\dots B_n$ is $\phi_w=\phi_{B_1}\circ\dots\circ\phi_{B_n}$.
To each leaf path naturally corresponds a $\Sigma$-word corresponding to the oriented bands crossed by this path.
We write $[x,w)$ for the leaf path starting at the point $x$  and whose $\Sigma$-word is $w$ (this leaf path is unique, 
but may fail to exist for some $x$ and $w$).
The \emph{holonomy} of $w$ is the map $h_w$ defined on a subset of $V(\call)$
and mapping $x$ to the terminal endpoint of $[x,w)$ when $[x,w)$ exists.
We write paths from right to left  so that $h_{ww'}=h_{w}\circ h_{w'}$ (thus, the path $ww'$ starts by $w'$ and ends by $w$).
By extension, if $[a,b]\subset D$ is such that $h_w$ maps $a$ to $a'$ and $b$ to $b'$, we say that $h_w$ maps $[a,b]$ to $[a',b']$.
When such $w$ exists, we say that $[a,b]$ and $[a',b']$ are \emph{equivalent under the holonomy} of $\call$.

\subsection{M\"obius strips}\label{sec_mobius}

A \emph{M\"obius strip} $S$ in a prelaminated band complex $(\Sigma,\call)$ consists of
a reduced band word $w$ and an interval $I$ adapted to $\call$ such that $h_w$ maps $I$ to itself reversing the orientation,
and $h_{w_1}(I)\neq h_{w_2}(I)$ for any two distinct non-empty  suffixes $w_1,w_2$ of $w$. 
Note that $(\Sigma,\call)$ has only finitely many M\"obius strips.
A \emph{core leaf} of $S$ is a leaf path $[x,w)$ with $h_w(x)=x$.

\begin{rem}\label{rem_strip}
  Note that if the twisting morphism $\phi_w$ of $w$ is trivial, or more specifically, if we
are solving untwisted equations in a free group, the presence of a M\"obius strip prevents the existence of a reduced
solution of $\Sigma$. Thus, in the context of equations in free groups, prelaminations containing  M\"obius strips can be discarded.
\end{rem}

If $\call$ is complete and all its M\"obius strips have a core leaf, we say that $\call$ is \emph{M\"obius-complete}.
A leaf is \emph{pseudo-singular} if it is singular or if it contains the core leaf of a M\"obius strip.
Most prelaminations we will use will consist of pseudo-singular leaves.

If $S$ has no core leaf, \emph{adding a core leaf to $S$} means extending $\call$ by adding at most $|w|$ new leaf segments $l_1,\dots,l_n$
 so that $S$ has a core leaf $l$ in the obtained prelamination, with $l\supset l_1\cup \dots\cup l_n$.

\subsection{Rational constraints on the prelamination}\label{sec_rational_prelamin}

Consider a band complex $\Sigma$ with rational constraints in standard form given
by $\rho:\fmip\ra \calm$ and $\ul m\in \calm^{E(\Sigma)}$.
Let $\call$ be a prelamination on $\Sigma$.
Recall that $E(\call)$ is the set of oriented elementary segments of $\call$,  
\ie of subsegments of $D$ joining adjacent $\call$-subdivision points. 

A set of rational constraints on the prelamination $\call$
is a tuple $m'_e\in\calm^{E(\call)}$ such that $m'_{\ol e}=\ol{m'_e}$
and which refines $\ul m$ in the following sense.
For a segment $J$ written as a concatenation of elementary segments of $\call$,  $J=e_1\dots e_n$,
define $m'_J=m'_{e_1}.\dots m'_{e_n}\in \calm$.
Then $\ul m'$ refines $\ul m$ if for all $J$ oriented elementary segment of $\Sigma$, $m_J=m'_J$.
This compatibility allows to blur the distinction between $\ul m$ and $\ul m'$, and we will use the same
notation $\ul m$ for both.

\subsection{Prelaminations and rational constraints induced by a solution}\label{sec_induced}

Consider a band complex  $\Sigma$ with $\rho:\fmip\ra\calm$ 
and $\ul m=(m_e)_{e\in E(\Sigma)}$ a set of rational constraints standard form.
Let $\sigma$ be a solution of $(\Sigma,\ul m)$.
The labelling $\sigma$ allows to subdivide $D$ into subsegments, each of which is labelled
by a letter in $\Spm$. 
This defines a complete prelamination $\call'$ whose subdivision points are the set of
vertices of the subdivided $D$:
since $\sigma$ is a solution, each band $B$ defines a pairing between all the vertices contained in its two bases.
This pairing can be realised geometrically by a finite set of leaf segments of the form $\{c\}\times [0,1]\subset[a,b]\times[0,1]$
by changing the two attaching maps of $B$ relatively to the endpoints of the bases.

Let $\call_\infty(\sigma)\subset \call'$ be the complete prelamination consisting of the pseudo-singular leaves of $\call'$.
We call $\call_\infty(\sigma)$ the \emph{complete prelamination induced by $\sigma$}.
This is a M\"obius-complete prelamination because the group $\Phi$ of twisting morphisms has no inversion.

For any segment $J$ adapted to some prelamination $\call\subset\call_\infty(\sigma)$,
one can define $m_J=\rho(\sigma_J)\in\calm$, which defines a rational constraint on $\call$.
This leads to the following definition:

\begin{dfn}
  A prelamination $\call$ with rational constraints $\calm$ is \emph{induced by $\sigma$}
if $\call\subset \call_\infty(\sigma)$, all leaves of $\call$ are pseudo-singular,
 and if  $m_e=\rho(\sigma_e)$ for all $e\in E(\call)$.
\end{dfn}

\begin{rem}\label{rem_contrainte_holonomie}
  The rational constraints induced by a solution are \emph{invariant under the holonomy} in the following sense:
consider a $\Sigma$-word $w$ 
with twisting morphism $\phi_w$ whose holonomy maps some $e\in E(\call)$ to $e'\in E(\call)$,
then $m_{e'}=\phi_w(m_e)$.
\end{rem}

Given a prelamination $\call$ of $\Sigma$, we define the solutions of $(\Sigma,\call)$.
A \emph{labelling} $\sigma$ of $(\Sigma,\call)$ is a labelling of each $e\in E(\call)$ by a word
$\sigma_e\in\fmip$ so that $\sigma_{\ol e}=\ol{\sigma_e}$.
As in section \ref{sec_solution_BC}, this allows to define $\sigma_J\in \fmip$ for all oriented segment $J$ adapted to $\call$.

\begin{dfn}
A solution of $(\Sigma,\call)$ is a labelling $\sigma$ of $(\Sigma,\call)$ such that
for all $\Sigma$-word $w$ and all oriented segments $J,J'$
adapted to $\call$ such that the holonomy of $w$ maps $J$ to $J'$, then $\sigma_{J'}=\phi_w(\sigma_J)$.

If there are rational constraints on $\call$, then we additionally require that 
$\rho(\sigma_e)=m_e$ for all $e\in E(\call)$.
\end{dfn}

The following observation is obvious.
\begin{lem}
  Let $\sigma$ be a solution of $\Sigma$ with rational constraints $\ul m$.
Let $(\call,\ul m')$ be a prelamination with rational constraints induced by $\sigma$.

Then $\sigma$ is a solution of $(\Sigma,\call)$ with rational constraints $\ul m'$.\qed
\end{lem}

\begin{dfn}\label{dfn_extends}
  Consider $\call\subset \call'$ some prelaminations on $\Sigma$ with rational constraints $\ul m,\ul m'$.
We say that $\call'$ \emph{extends} $\call$ and we denote $\call\prec\call'$ if for all $e\in E(\call)$,
$m_e=m'_e$.

\end{dfn}

When $\call'$ extends $\call$, the set of solutions of $(\Sigma,\call')$ with constraints $\ul m'$ 
is  a subset of the set of solutions of $(\Sigma,\call)$ with constraints $\ul m$.

\section{Algorithms on prelaminations}\label{sec_algos}

We now present two algorithms that will be part of the main algorithm solving equations. 
One of them  enumerates prelaminations of a given a band complex. 
The other is given a prelamination, and checks several properties,  for instance to prove that it cannot be induced by a shortest solution.

\subsection{A first algorithm: the prelamination generator}\label{sec_generator}

We first consider a \emph{prelamination generator}:
this straightforward algorithm takes as input a band complex with rational constraints $\Sigma$ 
and enumerates a set of prelaminations, 
such that any prelamination with rational constraints induced by a solution of $\Sigma$
has an extension which is enumerated.

The set of produced prelaminations is organised in a locally finite rooted tree.
The prelamination at depth 0 is the \emph{root prelamination} whose leaf segments are the boundary segments of the bands.
By definition, the root prelamination is contained in any prelamination.
We define the rational constraints on the root prelamination as the rational constraints on $\Sigma$.
The children of a given prelamination $\call$ are the extensions of $\call$ produced by the following
extension process, consisting of the three following steps:

\paragraph{Step 1: extend leaves of the prelamination.} 
Let $\calo$ be the set of orphans of $\call$ (as defined in Section \ref{subsec_prelamination}).
Look at all the possible ways to add leaf segments to $\call$, so that 
every leaf segment added has an endpoint in $\calo$, and points in $\calo$ are no more orphans in the new prelamination.

\paragraph{Step 2: add cores of M\"obius strips.} 
Consider all the (finitely many) M\"obius strips of $(\Sigma,\call)$ having no core leaf.
Then look at all possible ways to add a core leaf to each of them (see section \ref{sec_mobius} for definitions).

\paragraph{Step 3: extend rational constraints.}
If $\call'$ is a prelamination constructed above, and if $\call'=\call$, discard $\call'$.
Given $\call'\supsetneq\call$ constructed above,
and a set of rational constraints $\ul m\in\calm^{E(\call)}$ on $\call$, 
consider all possible rational constraints $\ul m'\in \calm^{E(\call')}$
such that $(\call',\ul m')$ extends $(\call,\ul m)$ as in Definition \ref{dfn_extends}, meaning that
for all oriented elementary segment $e\in E(\call)$ written as a concatenation
$e=e'_1\dots e'_n$ of elements of $E(\call')$, $m_e=m'_{e_1}\dots m'_{e_n}$.
Note that it might happen that for some $\call'$, there exists no set of rational constraints $\ul m'$.
\\

\begin{rem*}
The second step would actually be unnecessary if there was no twisting as no reduced solution of $\Sigma$
can induce a prelamination containing a M\"obius strip whose twisting morphism is trivial. 

Step 2 should add a core leaf to every M\"obius strip of $\call$, but it is allowed to create
new  M\"obius strips having no core leaf.

If all rational constraints are given by one-letter constants, then the third step consists in ensuring that
one does not subdivide the elementary segments corresponding to constants.
\end{rem*}

If $\call$ is M\"obius-complete 
(\ie has no orphan and all its M\"obius strips have a core leaf),
the only extension of $\call$ produced by the extension process is $\call$, and is discarded.
In particular, $\call$ is a leaf of the rooted tree of prelaminations.
If $\call$ is not M\"obius-complete, all the prelaminations produced properly contain $\call$.
Still, it may happen that no extension of $\call$ is produced
if for all $\call'\supsetneq\call$ obtained after step 2,
no extension of the rational constraints to $\call'$ exist.
Of course, in this case, $\call$ cannot be induced by a solution of $\Sigma$.

Recall that all band complexes and prelaminations come with a set of rational constraints which we do not mention explicitly.

\begin{lem}\label{lem_extension1}
Consider $\sigma$ a solution of $\Sigma$.
Assume that $\sigma$ induces some prelamination $\call$ produced by the prelamination generator.

Then either $\call=\call_\infty(\sigma)$,
or there exists some extension $\call'$ of $\call$ produced by the prelamination generator which is induced by $\sigma$.
\end{lem}

\begin{proof}
Assume that $\call\subsetneq\call_\infty(\sigma)$.
If $\call$ is not complete, then step 1 
will clearly produce an extension of $\call_1\supsetneq\call$ contained in $\call_\infty(\sigma)$.
Otherwise, step 1 will produce only $\call_1=\call$.
Since 
all M\"obius strips of  $\call_\infty(\sigma)$ contain a core leaf,
if $\call_1$ contains a M\"obius strip without core leaf,
then step 2 produces some $\call_2\supsetneq\call_1$ contained in $\call_\infty(\sigma)$.
The rational constraints induced by $\sigma$ on $\call_2$ will be produced in step 3.
The lemma follows.
\end{proof}

\begin{cor}\label{cor_extension}
For any  solution $\sigma$ of $\Sigma$, $\call_\infty(\sigma)$ will be produced by the prelamination generator.

Consider some prelamination $\call$ which is not M\"obius-complete, and  such that 
none of its extensions constructed by the prelamination generator
can be induced by a shortest solution of $\Sigma$. Then $\call$ cannot be induced by a shortest solution of $\Sigma$.
\end{cor}

\begin{proof}
The first statement clearly follows from the lemma since $\call_\infty(\sigma)$ cannot contain an infinite chain
of prelaminations strictly contained in each other.
The second statement is clear.  
\end{proof}

\subsection{A second algorithm: the prelamination analyser}\label{subsec_analyser}

We are going to describe a machine which analyses a given prelamination.
We call it the \emph{analyser}.

This algorithm takes as input a band complex $\Sigma$ together with a
prelamination $\call$ (with rational constraints). 
 It tries to reject prelaminations for which it can prove
that it  cannot be induced by a shortest solution.

More precisely, if the input prelamination $\call$ is M\"obius-complete, 
the analyser determines easily whether it is 
induced by some solution or not. 
The analyser the stops and outputs ``Solution found'' together with a solution, or ``Reject'' accordingly.

If the input prelamination $\call$ is not M\"obius-complete, it looks for some certificate 
 ensuring that no shortest solution can induce $\call$. If it can find one, it says ``Reject'' and stops.
Otherwise, it may fail to stop or say ``I don't know''.

Actually, 
we could make sure that it actually always terminates because given $\call$, there are only finitely 
certificates to try. But we won't need this fact.

The overall structure of the prelamination analyser is as follows.\\

\begin{algorithm}[H]
\dontprintsemicolon
\refstepcounter{thm}\label{algo_analyser}
\caption{\textbf{Algorithm \thethm.} Prelamination Analyser}
\KwIn{a prelamination with rational constraints $\call$ on a band complex $\Sigma$.}
\KwOut{``Solution found'', ``Reject'' or ``I don't know''.}
\BlankLine

  \begin{itemize}
  \item If $\call$ is a M\"obius-complete prelamination, the algorithm stops and its output is either
    \begin{itemize*}
    \item ``\emph{Solution Found}'': the analyser can produce a solution inducing $\call$.
    \item ``\emph{Reject}'': there exists no solution inducing $\call$, and the analyser can produce a certificate proving it.
    \end{itemize*}
\item If $\call$ is not M\"obius-complete, then
the algorithm looks for certificates proving that $\call$ cannot be induced by a shortest solution of $\Sigma$. 
The algorithm tries four kinds of reasons to reject $\call$: 
\begin{itemize*}
\item rejection for incompatibility of rational constraints (see subsection \ref{sec_incompatibility}),

\item rejection for non-existence of invariant measure  (see subsection \ref{sec_transverse}),

\item rejection for too large exponent of periodicity (see subsection \ref{sec_reject_bulitko}),

\item and rejection by a shortening sequence of moves (see subsection \ref{sec_shortening}).

\end{itemize*}
This procedure may stop or not, and its output can be
  \begin{itemize*}
  \item ``\emph{Reject}'': the algorithm produces a certificate proving that 
$\call$ cannot be induced by a shortest solution of $\Sigma$. 
  \item ``\emph{I don't know}'': failure to find such a certificate.
  \end{itemize*}
\end{itemize}
\end{algorithm}

We now describe more precisely the work of the prelamination analyser, describing the four types of rejection.

\subsubsection{Incompatibility of rational constraints}\label{sec_incompatibility}

The set of rational constraints $\ul m$ on a prelamination induced by a solution of $\Sigma$
is invariant under the holonomy (see Remark \ref{rem_contrainte_holonomie}):
for all band $B$ whose holonomy maps some $e\in E(\call)$ to $e'\in E(\call)$, $m_{e'}=\phi_B(m_e)$.
Thus, one can obviously reject rational constraints where this does not hold.  
 For example, in  this way one rejects   prelaminations such that the constants do not match under the holonomy.

There is another source of incompatibility coming from the twisting morphisms.
Assume that $w$ is a $\Sigma$-word whose holonomy fixes some segment $J$ adapted to $\call$,
and whose twisting morphism $\phi_w$ is non-trivial.
Then for any solution $\sigma$ of $\Sigma$ inducing $\call$, $\sigma_J$ lies in $\Fix \phi_w$.
By definition of standard form of rational constraints (section \ref{sec_standard}),
$\Fix \phi_w$ is represented by $\rho:\fmi\ra \calm$, so $\sigma_J\in\Fix\phi_w$
if and only $\rho(\sigma_J)\in F$ for the subset $\rho(\Fix\phi_w)\subset\calm$ which has already been computed.
Thus, one can reject $\call$ if $m_J\notin \rho(\Fix\phi_w)$ for some $w$  whose holonomy fixes $J$.
Even though the set of $\Sigma$-words whose holonomy fixes $J$ may be infinite,
this fact can be algorithmically checked.
Indeed, the fact that $m_J\in \rho(\Fix\phi_w)$ needs to be checked only for a generating set 
of the group of $\Sigma$-words $w$ fixing $J$, and such a generating set is easily computed.

We say that $\call$ shows an \emph{incompatibility of rational constraints} when one can reject $\call$ for one of the two possible reasons above.
In the prelamination analyser, rejection for incompatibility of rational constraints simply
consists in checking this algorithmically.

When $\call$ is M\"obius-complete, this is the only obstruction to the existence a solution of $\Sigma$ inducing $\call$ and $\ul m$:

\begin{lem}\label{lem_check_cplete}
  Let $\call$ be a M\"obius-complete prelamination of $\Sigma$.
Then $\Sigma$ has a solution inducing $\call$ with set of rational constraints $\ul m$ if and only if
it shows no incompatibility of rational constraints.
\end{lem}

This explains how the prelamination analyser determines the existence of a solution when its input is a M\"obius-complete prelamination:
it just checks whether $\call$ shows an incompatibility of rational constraints.
The argument below also shows how to produce a solution in the absence incompatibility of rational constraints.

\begin{proof}   
Since $\call$ is M\"obius-complete, the set of oriented elementary segments $E(\call)$ is partitioned
into orbits under the holonomy of $\call$, and no $e\in E (\call)$ is the orbit of $\ol e$ 
(\ie $e$ with the orientation reversed).
For each such orbit, choose one representative $e\in E(\call)$, and choose some word $\sigma_e\in\rho\m(m_e)$
($\rho$ is onto). For each $e'$ in the orbit of $e$, choose a $\Sigma$-word $w$ such that $h_w(e)=e'$
and define $\sigma_{e'}=\phi_w(\sigma_e)$.
Since $\call$ shows no incompatibility of rational constraints of the first kind,
one has $\rho(\sigma_{e'})=\phi_w(m_e)=m'_e$.
If $w'$ is another word with $h_{w'}(e)=e'$, then $\sigma_e\in\Fix \phi_{w' w\m}$
because $\call$ shows no incompatibility of rational constraints of the second kind,
so $\sigma_{e'}$ does not depend on the choice of $w'$.
Since  $e$ and $\ol e$ are not in the same orbit, we can extend these choices by defining
$\sigma_{\ol e'}=\ol\sigma_{e'}$ for all $e'$ in the orbit of $e$.
Doing this for every orbit of elementary segments, we get a solution of $(\Sigma,\call)$
respecting the rational constraints.
\end{proof}

\subsubsection{Rejection for non-existence of invariant measure}\label{sec_transverse}

If a prelamination $\call$ is induced by a solution $\sigma$, then the labelling defines a positive length $l_e$ on 
each elementary segment $e$ of $\call$, and more generally, to each segment adapted to $\call$.
This can be viewed as a combinatorial measure invariant under the holonomy
along the leaves of $\call$ in the following sense: if $I$ and $J$ are segments adapted to $\call$ and such that
the holonomy $h_B$ maps $I$ to $J$ for some band $B$, then $l_I=l_J$.

A prelamination may fail to have such an invariant transverse measure, for example if the holonomy along the leaves maps
a segment to a proper subset of itself. 
 The existence a combinatorial measure $(l_e)$ is equivalent to the existence of a positive solution 
to some system of linear equations with integer coefficients, 
which can be checked algorithmically. 
Thus, the prelamination analyser can check whether $\call$ admits an invariant combinatorial measure,
and if there is no such measure, it rejects $\call$.

\subsubsection{Rejection for too large exponent of periodicity}\label{sec_reject_bulitko}

In some cases, one can read in a prelamination $\call$ the fact that any solution inducing $\call$
must have a very large  exponent of periodicity.

\begin{dfn}\label{dfn_exponent}
We say that $exponent(\call)\geq N$ if there exists 
a segment $J$ which is 
a concatenation of $N$ intervals $I_1\cdots I_N\subset D$ adapted to $\call$,
and some $\Sigma$-words $w_i$ whose holonomy maps $I_i$ to $I_{i+1}$ preserving the orientation,
and whose twisting 
automorphisms $\phi_{w_i}$ are trivial.
\end{dfn}

Clearly, if $exponent(\call)\geq N$, then the exponent of periodicity of any solution inducing $\call$ 
is at least $N$.
By Bulitko's Lemma, one can compute a bound $B$ such that any shortest solution of $\Sigma$
has an exponent of periodicity bounded by $B$.

In the prelamination analyser, rejection for too large exponent of periodicity
consists in computing $B$, and checking whether $exponent(\call)> B$. This is easily done algorithmically.

\subsubsection{Rejection by a shortening sequence of moves}
\label{sec_shortening}

This last type of rejection is more complicated.
In Section \ref{sec_moves}, we will introduce a finite set of moves which can be performed on a band complex with a prelamination.
There will be several types of moves.

\begin{dfn}[Inert move]
We say that a prelaminated band complex $(\Sigma',\call')$ is obtained by an \emph{inert move}
from $(\Sigma,\call)$ if 
\begin{itemize}
\item for any solution $\sigma$ of $\Sigma$ inducing $\call$,
there is a solution $\sigma'$ of $\Sigma'$ inducing $\call'$ such that $|\sigma'|=|\sigma|$
\item for any solution $\sigma'$ of $\Sigma'$ (which need not induce $\call'$),
there is a solution $\sigma$ of $\Sigma$ (which need not induce $\call$) and such that $|\sigma|=|\sigma'|$
\end{itemize}
\end{dfn}

An obvious property of an inert move is that if $\call$ is induced by a shortest solution of $\Sigma$,
then $\call'$ is induced by a shortest solution of $\Sigma'$.
Examples of inert moves include band subdivision, domain cut, and band removal moves (see section \ref{sec_moves}).

To illustrate the two other types of moves, we start with an example: \emph{the pruning move}.
The input is a band complex $\Sigma$ with domain $D$ together
with a prelamination $\call$, and the output is a new prelaminated band complex
$(\Sigma',\call')$.
We assume that there is a segment $[a,b]\subset D$ where $a,b$ are vertices of $\Sigma$,
such that $[a,b]$ is contained in a base $J_0$ of some band $B$, and such that
no other base of $\Sigma$ intersects the interior of $[a,b]$ (see Figure \ref{fig_pruning}). 
Assume that $a$ and $b$ are not orphans in $\call$.
The rational constraints on $(\Sigma,\call)$ are assumed to be invariant under the holonomy.
We view the band $B$ as $J_0\times [0,1]$, the attaching map $f_{B,0}$ being the identity.
Let $D'=D\setminus (a,b)$, and let $\Sigma'$ be the band complex with domain $D'$ obtained by replacing
the band $B=J_0\times [0,1]$ by the two bands defined by the connected components of $(J_0\setminus (a,b))\times [0,1]$.
If $a$ or $b$ is an endpoint of $D$, then $(a,b)$ has to be interpreted as the interior of $[a,b]$
in $D$, 
in which case $(J_0\setminus (a,b))\times [0,1]$ consists of only one band or no band at all.
The attaching map of the new bands are well defined because $a$ and $b$ are not orphans.
The new prelamination is the one naturally obtained by restriction.
The new set of rational constraints on $\call'$ is obtained by restriction.
The set rational constraints $\ul m'$ on $\Sigma'$ is the one induced by $\call'$:
any elementary segment $e'$ of $\Sigma'$ is a concatenation of elementary segments $e'_1,\dots,e'_n$ of $\call'$,
which allows to define $m'_{e'}=m'_{e'_1}\dots m'_{e'_n}$.

\begin{figure}[htbp]
  \centering
  \includegraphics[width=10cm]{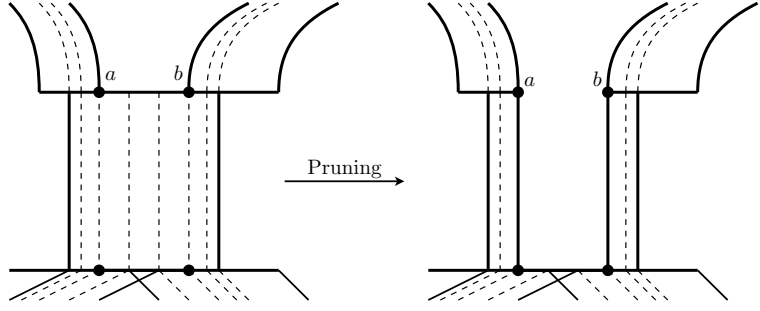}
  \caption{Pruning move}
  \label{fig_pruning}
\end{figure}

It is clear that any solution $\sigma$ of  $\Sigma$ inducing $\call$ 
(together with its set of rational constraints, which we don't mention explicitly)
defines a solution $\sigma'$  of  $\Sigma'$ inducing $\call'$ (with rational constraints).
Moreover, the length of $\sigma'$ is the length of the restriction of $\sigma$ to $D'$.

Conversely, if $\sigma'$ is any solution of $\Sigma'$, without assuming that 
$\sigma'$ induces $\call'$, it naturally extends to a solution $\sigma$ of $\Sigma$,
and $|\sigma|\leq 2|\sigma'|$.
We say that the move from $(\Sigma,\call)$ to $(\Sigma',\call')$ is a \emph{restriction move}
and that the move from $\Sigma'$ to $\Sigma$ is an \emph{extension move}.

\begin{rem} 
  Note that restriction moves are about prelaminated band complexes while
extension moves are about band complexes without prelamination.
\end{rem}

We will consider long sequences of  restriction and extension moves which may leave a large part of the band complex 
essentially untouched. To keep track of this, we partition the domain $D$ of the involved band complexes
into $D=D_I\dunion D_A$
where $D_I$ and $D_A$ are a union of connected components of $D$
where $D_I$ is the \emph{inert} part (which is left untouched by the move), and  $D_A$ the \emph{active} part.
If $D'$ is the domain of $\Sigma'$, we denote by $D'_I\dunion D'_A$ the corresponding decomposition of $D$.

\begin{dfn}[Restriction move]
We say that $(\Sigma',\call')$ is obtained by a \emph{restriction move}
from $(\Sigma,\call)$ if there is an injective map $\iota:D'\hookrightarrow D$
with  $D_I=\iota(D'_I)$ and $\iota(D'_A)\subset D_A$ adapted to $\call$, such that 
for any solution $\sigma$ of $\Sigma$ inducing $\call$,
there is a solution $\sigma'$ of $\Sigma'$ inducing $\call'$ such that
\begin{itemize}
\item $\abs{\sigma'_{|D'_I}}=\abs{\sigma_{|D_I}}$
\item $\abs{\sigma'_{|D'_A}}=\abs{\sigma_{|\iota(D'_A)}}$
\end{itemize}
\end{dfn}
 
\begin{rem}
We call $\iota$ the \emph{restriction map}.
For the pruning move, for instance, $\iota$ is just the inclusion.
\end{rem}

\begin{dfn}[Extension move]
We say that $\Sigma$ is obtained from $\Sigma'$ by an \emph{extension move}
with Lipschitz factor $\lambda\geq 1$
if  for any solution $\sigma'$ of $\Sigma'$,
there is a solution $\sigma$ of $\Sigma$ such that
\begin{itemize}
\item $\abs{\sigma_{|D_I}}=\abs{\sigma'_{|D'_I}}$
\item $\abs{\sigma_{|D_A}}\leq \lambda\abs{\sigma_{|D'_A}}$.
\end{itemize}
\end{dfn}

For instance, one can take $\lambda=2$ for the pruning move, under the requirement that
$D_A$ contain both bases of the band pruned.
\\

To be able to ensure that a prelamination $\call$ cannot be induced by a shortest solution, 
we need a way of certifying that some subset $\iota(D'_A)\subset D_A$ is \emph{short relative to $D_A$},
\ie  that for any solution $\sigma$ inducing $\call$,
$\abs{\sigma_{|D'_A}}$ is at most $\eps\abs{\sigma_{|D_A}}$ for some small $\eps>0$.
Relative shortness will be guaranteed by some repetition as in the lemma below.

\begin{dfn}[Certificate of shortness of $J$ relative to $D_A$]\label{dfn_certif_shortness}
Let $(\Sigma,\call)$ be a prelaminated band complex, and $J \subset D_A$ be a subset adapted to $\call$. 
A \emph{certificate of $\eps$-shortness of $J$ relative to $D_A$} is
a family of adapted subsets $J_1,\dots,J_N\subset D_A$ with disjoint interiors for some $N\geq 1/\eps$,
and which are all equivalent to $J$ under the holonomy of $\call$.
\end{dfn}

\begin{lem}\label{lem_certif_shortness}
If there is a certificate of $\eps$-shortness of $J$ relative to $D_A$,
then  for any solution $\sigma$ of $\Sigma$ inducing $\call$,
then $\abs{\sigma_{|J}}\leq \eps \abs{\sigma_{|D_A}}$.
\end{lem}

\begin{proof}
  This is essentially obvious: for any solution $\sigma$ inducing $\call$,
the labellings of $J_1,\dots,J_N$ will have the same length as $J$.
Since the $J_i$'s don't overlap and are contained in $D_A$,
we get $N.\abs{\sigma_{|J}}\leq \abs{\sigma_{|D_A}}$ and the lemma follows.
\end{proof}

The prelamination analyser will try to apply a sequence of elementary moves (as defined in Section \ref{sec_moves}) 
to the prelaminated band complex.
It is clear that a 
succession of inert moves defines an inert move.
Moreover,  a 
succession of restriction moves having the same inert part also defines a restriction move,
whose restriction map $\iota$ is the composition of the restriction maps of the elementary moves.
Similarly, a 
succession of elementary extension moves having the same inert part
defines an extension move whose Lipschitz factor is 
the product of the Lipschitz factors of the individual moves.\\

\begin{dfn}[Shortening sequence of moves]\label{dfn_shortening_moves}
  A \emph{shortening sequence of moves} for the prelamination $\call_0$ of $\Sigma_0$ consists of the following data:
  \begin{enumerate*}
\item a sequence of inert moves transforming $(\Sigma_0,\call_0)$ to $(\Sigma,\call)$;
  \item a sequence of elementary restriction moves, whose concatenation defines a restriction move
$(\Sigma,\call)\ra(\Sigma',\call')$, with corresponding map $\iota:D'_A\ra D_A$;
\item  a sequence of elementary extension moves transforming $\Sigma'$ back to $\Sigma$;
\item a certificate of $\eps$-shortness of $\iota(D'_A)$ relative to $D_A$ with $\eps<1/\lambda$,
where $\lambda$ is the product of the Lipschitz factors of the extension moves used above.
  \end{enumerate*}
\end{dfn}

\begin{lem}\label{lem_correct_shortening}
  If there exists a shortening sequence of moves for $\call_0$,
then $\call_0$ cannot be induced by a shortest solution of $\Sigma_0$.
\end{lem}

\begin{proof}
Assume on the contrary that $\call_0$ is induced by a shortest solution of $\Sigma_0$.
Using the property of inert moves,  $\call$ is itself induced by a shortest solution $\sigma$ of $\Sigma$.
Using the restriction move, consider $\sigma'$ a solution of $\Sigma'$ such that
$\abs{\sigma'_{|D'_I}}=\abs{\sigma_{|D_I}}$
and $\abs{\sigma'_{|D'_A}}=\abs{\sigma_{|\iota(D'_A)}}$.
Using the extension move, consider a solution $\Tilde \sigma$ of $\Sigma$
such that 
$\abs{\Tilde \sigma_{|D_I}}=\abs{\sigma'_{|D'_I}}$ and
 $\abs{\Tilde \sigma_{|D_A}}\leq \lambda\abs{\sigma'_{|D'_A}}$.
Since $\iota(D'_A)$ has a certificate of $\eps$-shortness relative to $D_A$,
 $\abs{\Tilde \sigma}\leq \abs{\sigma_{D_I}}+\eps\lambda\abs{\sigma_{|D_A}}<\abs{\sigma}$, a contradiction.
\end{proof}

\subsubsection{More about the analyser}

We are now ready to describe the rejection by shortening moves by the prelamination analyser.
The analyser tries all possible sequences of moves looking like a shortening sequence of moves constructed as follows:
a sequence of inert moves transforming $(\Sigma_0,\call_0)$ into some $(\Sigma,\call)$, 
followed by a choice of decomposition  $D=D_I\dunion D_A$,
followed by a sequences of shortening moves with inert part $D_I$ transforming $(\Sigma,\call)$ into some $(\Sigma',\call')$,
followed by a sequence of extension moves with inert part $D_I$ and transforming $\Sigma'$ back into $\Sigma$.
All such sequences can be enumerated. 
Then look for $1/\lambda$-shortness certificates for $\iota(D'_A)$ in $D_A$.
Being a little more careful, one could easily ensure that this enumeration terminates, but we won't need this fact.
If such a certificate exists, it will be found. In this case, the prelamination analyser rejects $\call$.

\begin{lem}
\label{lem_analyser_correct}
  The prelamination analyser is correct: 
if it says ``Solution found'', then $\Sigma$ has a solution inducing $\call$;
if it says ``Reject'',
then $\call$ cannot be induced by a shortest solution.
\end{lem}

\begin{proof}
The correctness when a solution has been found follows from Lemma \ref{lem_check_cplete}.
The correctness for each reason of rejection follows from Subsections \ref{sec_incompatibility}, \ref{sec_transverse},
\ref{sec_reject_bulitko}, and from Lemma \ref{lem_correct_shortening} above.
\end{proof}

\section{Asymptotic of prelaminations: Rips band complexes}\label{sec_asym}

 In this section, we assume that in the locally finite tree of prelaminations produced by the prelaminations generator, 
there is an infinite chain, \ie an infinite sequence of prelaminations with rational constraints extending each other. 

By passing to a limit, 
we will 
construct a topological foliation  on $\Sigma$ together with an invariant transverse measure $\mu$. 
With an assumption on the exponent of periodicity, 
we will prove that 
$\mu$ has no atom. However, $\mu$ may fail to have full support.
Collapsing segments of measure zero, we will get a new foliated band complex $\Sigma_\mu$ where the measure is the Lebesgue measure: we say
that $\Sigma_\mu$ is a \emph{Rips band complex}.

In   \cite{GLP1,BF_stable} is studied   the decomposition of such an object into  so-called  minimal  components, 
 and the classification of such    minimal components. In particular, there are so-called homogeneous components, 
in which, as we will see, one can find arbitrarily high exponent of periodicity. If a minimal component is not homogeneous, 
one can perform a sequence of moves on the Rips band complex 
 (to obtain so-called independent bands) so that it becomes either a surface component 
(also called interval exchange component) or an exotic component (also called  thin, or Levitt). 
In both cases, we will prove the existence of a shortening sequence of moves (in the sense of Definition \ref{dfn_shortening_moves}) 
for all sufficiently large prelaminations in our infinite chain.

These properties are used to prove that the main algorithm \ref{algo;main} always terminates, 
see Section \ref{sec_endgame}.

\subsection{Limit measured foliation}\label{sec_limit}

Consider an infinite sequence of prelaminations with rational constraints extending each other:
$\call_1 \prec \call_2 \prec\dots$ (Definition \ref{dfn_extends}).

The goal of this subsection is to associate to such a chain of prelaminations a topological foliation on $\Sigma$,
and a transverse invariant measure under the assumption that each $\call_i$ has an invariant transverse measure (Subsection \ref{sec_transverse}). 
Of course, those objects are not accessible to the algorithm, but they are used
to prove the existence of certificates  rejecting $\call_i$ for $i$ large enough.

\subsubsection{The limit topological foliation}

We first define the notion of a topological foliation on a band complex $\Sigma$.
Recall that bands of a band complex are defined by attaching maps
$f_{B,\eps}:[a,b]\times\{\eps\}\ra D$ for $\eps=0,1$, 
but the only relevant data is the combinatorial arrangement of the endpoints of bases of bands.
If $\Sigma$ is endowed with a prelamination, more data has to be extracted from $f_B$, 
\ie the combinatorial arrangement of the images of endpoints of leaf segments.
A \emph{topological foliation} $\calf$ consists in choosing for each band, an injective continuous map 
$\Tilde f_{B,\eps}:[a,b]\times\{\eps\}$ 
coinciding with $f_{B,\eps}$ on endpoints of bases of bands.
Here, the precise values of $\Tilde f_{B,\eps}$ matters, not only its values on a finite subset.

\emph{Leaf segments} of $\calf$ in a band $B$ consist of all segments $\{x\}\times [0,1]\subset B$.
The \emph{endpoints} of the leaf segments are the points $f_{B,0}(x,0),f_{B,1}(x,1)\in D$.
Consider the equivalence relation on the set of leaf segments, generated by
the property of having a common endpoint. 
A \emph{leaf} of $\calf$ is the union of leaf segments
in a given equivalence class.
As usual, we identify a leaf with its image in the topological realisation of the band complex.

As for prelaminations, we can talk about singular or regular leaf segments, leaf paths, and leaves.
Given a $\Sigma$-word $w$, and $x\in D$, we still use the notation $[x,w)$ which has a meaning for a topological foliation.
The holonomy map $h_w$ is now a homeomorphism between a compact interval $\dom w\subset D$, called the \emph{domain} of $w$,
and its image.

\begin{dfn}\label{dfn_represent}
Consider a (maybe infinite) set $l_1,\dots,l_p,\dots$ of leaves of a topological foliation $\calf$
containing all singular leaves.

We say that a sequence of prelaminations $\call_i$ 
is \emph{represented} in $\calf$
if one can write $\call_i$ as a union of leaf segments of $\calf$ so that
\begin{itemize}
\item $\call_1\subset\call_2\subset\dots$ is an exhaustion of $l_1\cup\dots\cup l_p \cup \dots $ by finite subgraphs,
\item every M\"obius strip of $\call_i$ has a core in $\call_j$ for some $j\geq i$.
\end{itemize}
The leaves $l_i$ are called  the \emph{special leaves} of $\calf$.
\end{dfn}

\begin{rem}
\label{rem_pseudo}
By definition, every singular leaf $l$ is special.
If $l$ is pseudo-singular but not singular, and if $\call_i$ has an invariant combinatorial measure,
then $l$ is special.
Indeed, consider $x\in \call$ and a $\Sigma$-word $w$ such that $x$ lies in the interior of $\dom w$, $h_w(x)=x$, and $w$ reverses the orientation.
Let $[a,b]$ be the domain of $h_{w^2}$. If $h_{w^2}$ fixes $a$, then there will a M\"obius strip in $\call_i$ for $i$ large enough
since the leaf through $a$ is singular hence special. The second condition ensures that $l$ is special.
Otherwise, up to changing $w$ to $w\m$, we can assume $h_{w^2}(a)>a$.
This implies $h_{w^2}(b')<b'$, where $b'=h_w(a)$.
It follows that the holonomy of $w^2$ maps $[a,b']$ to a proper subsegment of itself, which prevents
the existence of an invariant combinatorial measure.
\end{rem}

\begin{prop}\label{prop_topological}
Consider an infinite sequence of prelaminations $\call_1\prec\call_2\prec\dots$ extending each other,
and consisting of pseudo-singular leaves.
Assume that for each subdivision point $x$ of $\call_i$, there exists $j>i$ such that $x$ is not orphan in $\call_j$,
and that any M\"obius strip of $\call_i$ has a core in some $\call_j$ for some 
$j\geq i$.

Then there exists a topological foliation $\calf$ on $\Sigma$ 
representing $\call_i$.
\end{prop}

\begin{proof}

Embed $D$ in $\bbR$ and endow it with the induced metric. Denote by $|C|$ the diameter
of a subset $C\subset D$.
Let $L_i=V(\call_i)$ be the set of subdivision points of $\call_i$.
Recall that we have some freedom in the choice of the points of $L_i$ 
as long as we don't change the induced ordering of the points.
Thus, we can choose inductively $L_i$ so that $L_{i}\supset L_{i-1}$, 
and so that the following holds.
For each component $C$ of $D\setminus L_{i}$, let  $C'$ 
be the component of $D\setminus L_{i-1}$ containing $C$.
Then either $C=C'$, or $|C|\leq |C'|/2$.

We also have some freedom in the choice of attaching maps
of each band $f_{B,\eps}^{(i)}:[a,b]\times\{0,1\}\ra D$ 
used to define the prelamination $\call_i$.
By modifying $f_{B,\eps}^{(i)}$  with respect to  
 $L_i$, we can assume that 
$f_{B,\eps}^{(i)}$ coincides with $f_{B,\eps}^{(i-1)}$ on each $C$ such that $C=C'$
(formally, we should rather say that they coincide on $f_{B,\eps}\m(C)$, but to
keep notations lighter, we will make this abuse of notations).

Let $L_\infty=\cup_{i\geq 0} L_i$.
Our inductive choice of $L_i$ ensures that for any connected component $C$ of
$D\setminus \ol{L_\infty}$, its endpoints lie in some $L_{i_0}$ so $C$ is actually a connected component
of $D\setminus L_{i_0}$. 
Say that a connected component of $D\setminus L_{i_0}$ is \emph{stabilised} if it is also a connected component of $D\setminus \ol{L_\infty}$.
In this case, any oriented band $B$ whose initial base intersects (and therefore contains) $C$,
the restriction to $C$ of the attaching map $f_{B,0}^{(i)}$ is independent of $i\geq i_0$.
Up to enlarging $i_0$, we can assume that the endpoints of $C$ are not orphan in $\call_{i_0}$. 
It follows that the image $C'$ of $C$ under the holonomy of $B$
is also a stabilised component, so the restriction to $C$ of the holonomy of $B$ does not depend on $i\geq i_0$. 

We fix a band $B$ with bases $J_0,J_1\subset D$.
We claim that $h_i=f_{B,1}^{(i)}(f_{B,0}^{(i)}){}\m:J_0\ra J_1$ converges uniformly as $i$ goes to infinity.
Fix some $\eps>0$. 
Let $U$ be the (finite) union of all connected components of $D\setminus \ol{L_\infty}$
of diameter at least $\eps/2$. Denote by $N_i\subset L_i$ the set of non-orphan subdivision points of $\call_i$,
and note that $\bigcup_{i>0} N_i=\bigcup_{i>0} L_i$.
Let $i_0$ be such that $\partial U\subset N_{i_0}$, 
and that all connected components of 
$(D\setminus U)\setminus N_{i_0}$ have diameter at most $\eps$.

Consider $x\in J_0$. If  $x\in U$, then $h_j(x)=h_{i}(x)$ for all $j,i>i_0$. 
If $x\notin U$, let $[s,s']$ be the smallest interval containing $x$ such that 
$s$ and $s'$ are the endpoints of leaf segments of $B$.
Then $h_i(s)$ and $h_i(s')$ are independent of $i$, and are at distance at most $\eps$
by choice of $i_0$.
Since for all $i$, $h_i(x)\in [h_i(s),h_i(s')]$, we get $|h_i(x)-h_j(x)|\leq \eps$ for all $i,j\geq i_0$, 
which proves uniform convergence to some $h_\infty$. Since $h_i\m$ converges by the symmetric argument,
$h_\infty$ is a homeomorphism.

Finally, we can define the foliation using the attaching maps $\Tilde f_{B,\eps}$ defined by
$\Tilde f_{B,0}=f_{B,0}^{(0)}$
and $\Tilde f_{B,1}=h_\infty\circ f_{B,0}^{(0)}$.

It is now clear that $\calf$ represents $\call_i$: indeed 
$\call_i$ is an exhaustion of the union of the leaves it intersects 
because no subdivision point of $\call_i$ remains orphan forever.
\end{proof}

\begin{rem}\label{rem_complement_speciales}
Although the foliation is meaningless in the complement $\Sigma^*$ of
closure of special leaves, it can be convenient to ensure 
that the foliation on $\Sigma^*$ is a twisted product, in other words, 
that the holonomy of any $\Sigma$-word preserving a connected component $J$ of $\Sigma^*\cap D$
and preserving the orientation has to be the identity on $J$.
This is easily done by choosing the attaching maps in the proof above to be affine
on the complement of the subdivision points.
Also note that the measure to be defined in next section is zero on $\Sigma^*$.
\end{rem}

\subsubsection{The limit invariant measure}

\begin{prop}\label{prop_measure}
Consider an infinite sequence of prelaminations $\call_1\prec\call_2\prec\dots$ represented
in some topological foliation $\calf$.
Assume that for each $i$, $\call_i$ has an invariant combinatorial measure.

Then either the exponent of periodicity of $\call_i$ goes to infinity (see Definition \ref{dfn_exponent}), or 
there exists a measure $\mu$ on $D$, which has no atom and which is invariant under the holonomy of $\calf$.
\end{prop}

In particular, if no $\call_i$ is rejected by the prelamination analyser, then $\calf$ has a non-atomic
transverse invariant measure.

\newcommand{\Orph}{\mathrm{Orph}}

\begin{proof}
The construction follows \cite{Plante_measure}.
Denote by $\calf^s$ the (finite) union of singular leaves of $\calf$.
We first note that if $\calf^s$ is finite (\ie contains only finitely many leaf segments), 
then $\call_i$ is eventually constant.
Indeed, consider $J\subset D$ a segment meeting the singular leaves of $\calf$ only at its endpoints.
Since $\call_i$ preserves a combinatorial measure, the union of all non-singular pseudo-singular leaves of $\call_i$ 
intersects $J$ in at most one point (its midpoint).
Since all leaves of $\call_i$ are pseudo-singular, this bounds the number of leaf segments in $\call_i$.

Consider $l_0$ an infinite singular leaf of $\calf$, and choose a base point $x_0\in l_0$.
Let $l_0(n)$ be the ball of radius $n$ centred at $x_0$ in $l_0$ viewed as a graph, and consider $L_n=l_0(n)\cap D$.

We claim that $\#L_n$ grows polynomially with $n$:
\begin{lem}
$\#L_n\leq 2(2n+1)^b$ where $b$ is the number of bands of $\Sigma$.
\end{lem}

\begin{proof}
Fix some $n$, and consider $i$ large enough so that the finite graphs $l_0(n)$ is contained in the prelamination $\call_i$.
Since $\call_i$ has an invariant combinatorial measure, 
one can embed $D$ in $\bbR$ isometrically with respect to this combinatorial measure.
For each band $B$ of $\Sigma$, the holonomy of $B$ extends to an isometry $\gamma_{B}$ of $\bbR$.
Let $G<\Isom(\bbR)$ be the subgroup generated by the isometries $\gamma_{B}$.
One easily checks that the ball of radius $n$ of the Cayley graph of $G$ 
has cardinal $2(2n+1)^b$ where $b$ is the number of bands of $\Sigma$. 
If follows that the set of images of $x_0$ under elements of length at most $n$ of $G$ has cardinal at most
$2(2n+1)^b$. Since $L_n$ is contained in this set, the lemma follows.
\end{proof}

Consider the normalised counting measure $\mu_n=\frac{1}{\#L_n} \sum_{s\in L_n}\delta_s$
where $\delta_s$ is the Dirac mass at $s$.
Consider the set $O_n=L_{n}\setminus L_{n-1}$. This is the locus where $\mu_n$ might fail to be preserved:
if $B$ is a band with bases $J_0,J_1$, and if $x\in J_0\cap L_n\setminus O_n$,
then $h_B(x)\in L_n$ so $\mu_n(x)=\mu_n(h_B(x))>0$.
Using the same argument for $B\m$, this means that $h_B$ preserves the measure $\mu_n$ except maybe
at points of $O_n\cup h_B(O_n)$.
Because of polynomial growth,
there exists a subsequence $n_k$ such that the proportion $\frac{\#O_{n_k}}{\#L_{n_k}}$ converges to $0$
as $k$ goes to $\infty$.

Let $\mu$ be any accumulation point of $\mu_{n_k}$ for the weak-* topology (existence is ensured by Banach-Alaoglu's theorem).
If we know that the endpoints of the bases of bands have zero measure,
the following lemma proves that $\mu$ is invariant under the holonomy.

\begin{lem}\label{lem_holonomie}
Consider a band $B$ with bases $J$, $J'$ and with holonomy $h:J\ra J'$.
Then $\mu_{|\rond{J}'}=h_*\mu_{|\rond J}$.
\end{lem}

\begin{proof}
\newcommand{\intd}[1]{\int #1\, \mathrm{d}}
Consider $\eps>0$ and a continuous function $f:D\ra\bbR $ with support in $\rond J$, and $f'=f\circ h\m$ with support in $\rond J'$.
Then for all $\eps>0$, for $k$ large enough,
$$\abs{\intd{f'}\mu - \intd{f}\mu }  \leq  \frac{1}{\#L_{n_k}}\left\lvert \sum_{s'\in L_{n_k}} f'(s') - \sum_{s\in L_{n_k}} f(s) \right\rvert +\eps.$$
For each $s\in L_{n_k}\setminus O_{n_k}\cap \rond{J}$, the point $s'=h(s)$ lies in $L_{n_k}\cap \rond{J'}$,
and the two corresponding  terms $f(s)$ and $f'(s')$ cancel.
Thus, 
$$\abs{\intd{f'}\mu - \intd{f}\mu }  \leq  \frac{1}{\#L_{n_k}} \sum_{s'\in O_{n_k}} \abs{f'(s')}+\frac{1}{\#L_{n_k}}\sum_{s'\in O_{n_k}} \abs{ f(s)} +\eps$$
$$\leq   \frac{\#O_{n_k}}{\#L_{n_k}}||f||_\infty+\eps$$
Since the proportion  $\frac{\#O_{n_k}}{\#L_{n_k}}\ra 0$ converges to zero, 
the lemma follows.
\end{proof}

To prove that $\mu$ is invariant under the holonomy, we need to prove that $\mu$ has no atom.
Note that the lemma implies that if $\mu(\{x\})>0$ for some $x\in \rond{J}$, then $\mu(\{h(x)\})=\mu(\{x\})$

Consider $x\in D$  and a segment $J=[x,a]$. Consider the limit $\mu_J$ of the restriction of $\mu_{n_k}$ to $J$ (up to passing to a subsequence).
Clearly, $\mu_J$ coincides with $\mu$ on $\rond{J}$, but $\mu_J(\{x\})$ may differ from $\mu(\{x\})$  because
leaves of $\call_i$ can go in the neighbourhood of $x$ without entering $J$.
If $J'=[x,b']\subset J$, then $\mu_{J'}(\{x\})=\mu_J(\{x\})$. We thus denote by $\mu_+(x)=\mu_{[x,x+t]}(\{x\})$ and 
$\mu_{-}(x)=\mu_{[x-t,x]}(\{x\})$ for some (any) $t>0$.
Since $\mu_n(\{x\})\ra 0$ as $n\ra \infty$, one gets $\mu_+(x)+\mu_-(x)=\mu(\{x\})$.
The proof of Lemma \ref{lem_holonomie} directly extends to the following:

\begin{lem}\label{lem_holonomie2}
Let $J$ be a base of some band $B$ with holonomy $h:J\ra J'$, and assume that $[x,x+\eps t]\subset J$ for some $\eps=\pm1$, $t>0$.
Let $\eps'=\pm 1$ be the sign of $h(x+\eps t)-h(x)$ for $t>0$ small enough.

Then $\mu_{\eps'}(h(x))=\mu_{\eps}(x)$.  \qed
\end{lem}

The following lemma ends the proof of Proposition \ref{prop_measure}.
\end{proof}

\begin{lem}
If $\mu$ has an atom,
then $exponent(\call_i)\ra\infty$.
\end{lem}

\begin{proof}
Consider $a\in D$ with $\mu(\{a\})>0$.
 Assume for instance that $\mu_+(a)>0$.

Let $l$ be the $\calf$-leaf through $a$.
Define the graph $\calc$ as follows: its vertex set
$V(\calc) \subset (l\cap D)\times \{+1,-1\}$ is defined by $(y,\eps)\in V(\calc)$ if 
there exists a $\Sigma$-word $w$ whose holonomy $h_w$ is defined on $[a,a+t]$ for some $t>0$, 
with $h_w(a)=y$ and $h_w(a+t)-h_w(a)$ of the sign of $\eps$.
Put an edge labelled $B$ between $(y,\eps)$ and $(y',\eps')$ if 
$h_B(y)=y'$ and $h_B(y+\eps t)-h_B(y)$ is defined and of the sign of $\eps'$ for $t>0$ small enough.
The natural map $V(\calc)\ra l$ sending $(y,\eps)$ to $y$ extends to 
a map $\calc\ra l$ sending edge to edge, and which is at most $2$ to $1$.
Let us choose the point   $(a,+1)$ (corresponding to the empty word) as a base point for  $\calc$.

By Lemma \ref{lem_holonomie2}, for all $(y,\eps)\in \calc$, $\mu_\eps(y)=\mu_+(a)$.
Since $\mu(D)<\infty$, $\calc$ is a finite graph.
Since $\mu_+(a)>0$, the singular leaf $l_0$ of $\calf$ 
accumulates on $a$ on the right.

Let $T$ be a maximal tree in $\calc$. For any  vertex $(y,\epsilon)\in \calc$, 
let us denote $[(a,+1),(y,\epsilon)]_T$ the unique path in $T$ joining the two points. Each such segment 
defines a  $\Sigma$-word $w_{(y,\epsilon)}$ whose domain  contains an interval  $[a,a+\eta_y]$ for some $\eta_y>0$.    For each oriented edge $e\in\calc$ not in $T$, labelled by $B$, joining $v_1$ to $v_2$, the $\Sigma$-word $w_e= w_{v_2}^{-1}B w_{v_1}$ has its domain containing some interval  $[a,a+\eta_e]$ for some $\eta_e>0$. Note that by definition of $\calc$, its holonomy  fixes $a$ and preserves the orientation.    

Let $\eta>0$ be such that the holonomy of each  $\Sigma$-word $w_e, e\in \calT\setminus T$ is defined on  $[a,a+\eta]$.

We claim that for any $x\in l_0\cap D$ close enough to $a$ on the right of $a$,
there exists an oriented edge $e\in \calc\setminus T$ such that $h_{w_e} ([a,x])\subsetneq [a,x]$.
Otherwise, take $x$  close enough to $a$ so that the set
$\bigcup_{v\in V(\calc)} h_{w_v}([a,x])$ does not contain an endpoint of a base of band.
Then any band defined on some $h_{w_v}(x)$ is defined on the whole interval $h_{w_v}([a,x])$.
If for all oriented edge $e\in   \calc\setminus T   $, $h_{w_e}(x)=x$, then the leaf through $x$ is finite since its 
 intersection with $D$ 
is the finite set  $\bigcup_{v\in   V(\calc) } h_{w_v}(x)$. This contradicts the fact that $x\in l_0$ since $l_0$ is infinite,
and proves our claim.

Consider an oriented edge $e\in   \calc\setminus T $ and $x\in l_0$ such that $h_{w_e}([a,x])\subsetneq [a,x]$.
Since $l_0$ is singular, $x$ is a subdivision point of $\call_i$ for $i$ large enough.
This does  not necessarily prevent the existence of a combinatorial measure for $\call_i$
because $a$ is not necessarily a subdivision point of $\call_i$. 
Up to replacing $w_e$ by some power $w$ of it, one can assume that the twisting morphism $\phi_{w}$ is trivial.
Then for any integer $m$, there exists $i$ large enough such that 
all segments $[x,h_w(x)],[h_w(x),h_w^2(x)],\dots,[h_w^{m-1}(x),h_w^m(x)]$
are mapped to each other under the holonomy of a power of $w$ in $\call_i$.
It follows that $exponent(\call_i)$ goes to infinity with $i$.

\end{proof}

\subsubsection{Collapsing segments of measure 0, Rips band complex.}

Now we assume that $\Sigma$ is endowed with its topological foliation $\calf$, and its non-atomic invariant transverse measure $\mu$.
Still, $\mu$ may fail to have full support.
We now define a \emph{Rips} band complex by collapsing connected components of the complement of the support of $\mu$ as follows.

Let $D_\mu$ be the quotient of $D$ by the equivalence relation $\mu([x,y])=0$  (recall that $\mu$ has no atom), 
and $\pi_\mu:D\ra D_\mu$ the quotient map.
For each band $B$ of $\Sigma$ with bases $J_0,J_1$, we consider the corresponding band $B_\mu$ with 
bases $\pi_\mu(J_0),\pi_\mu(J_1)$ and whose holonomy is induced by the holonomy of $B$. Let $\Sigma_\mu$ be the complex of bands on $D_\mu$ whose bands are the $B_\mu$.
The holonomy of the bands naturally defines a foliation $\calf_\mu$ on $\Sigma_\mu$.
We still denote by $\mu$  the measure induced by $\mu$ on $D_\mu$.
This measure is invariant under the holonomy of $\calf_\mu$.

Thus, one obtains a new foliated band complex $\Sigma_\mu$,
where one needs to generalise the definition of a band complex to allow 
some connected components of $D_\mu$ to be reduced to a point, and some bands to have
bases reduced to a point.
In this setting, we still can talk of leaves, singular leaves etc.

Since $\mu$ has no atom and has full support in $D_\mu$, 
it can be transported to the Lebesgue measure of a finite union of intervals of $\bbR$ by a homeomorphism.
Such a band complex $\Sigma_\mu$ with its measured foliation $(\calf_\mu,\mu)$ is thus
the object studied in  \cite{BF_stable}
or in \cite{GLP1} under the name system of isometries.
We call such a band complex a \emph{Rips} band complex.

We will use the same notations for the bands of $\Sigma$ and the corresponding bands of $\Sigma_\mu$.
In particular, we will view any $\Sigma$-word as a $\Sigma_\mu$-word, and conversely.

We now gather some simple facts about the projection $\Sigma\ra\Sigma_\mu$.

\begin{lem}\label{lem_singular}
  Consider $x\in D_\mu$ whose $\calf_\mu$-leaf is singular (resp. pseudo-singular). 

Then there exists $\Tilde x\in D$ with $\pi_\mu(\Tilde x)=x$ such that $\Tilde x$ lies in a singular (resp. pseudo-singular) leaf of $\calf$.
\end{lem}

\begin{proof}
If $x$ is a vertex of $\Sigma_\mu$, then $\pi_\mu\m(\{x\})$ contains a vertex of $\Sigma$ and we are done.
Otherwise, there is a $\Sigma_\mu$-word $w$ such that $[x,w)$ is a regular leaf path joining $x$ to a vertex $y$ of $\Sigma_\mu$.
Let $\Tilde y\in\pi_\mu\m(\{y\})$ be a vertex of $\Sigma$.
Viewing $w$ as a $\Sigma$-word, the point $h_{w\m}(\Tilde y)\in \pi_\mu\m(\{x\})$ satisfies the lemma.
If the leaf through $x$ is pseudo-singular but not singular, let $w$ be a $\Sigma_\mu$ word such that
$h_w(x)=x$ and $h_w$ reverses the orientation.
Then the corresponding $\Sigma$-word $w$ has a fix point $\Tilde x$, 
so the leaf through $\Tilde x$ is pseudo singular.
Since $x$ is the unique fix point of $h_w$, $\pi_\mu(\Tilde x)=x$.
\end{proof}

\begin{lem}\label{lem_finite_leaves}
  No regular leaf of $\Sigma_\mu$ is finite.
\end{lem}

\begin{proof}
Let $x\in D_\mu$ whose $\Sigma_\mu$-leaf is regular and finite.
Consider a small open interval $I\subset D_\mu$ containing $x$ in its interior,
small enough so that the every leaf meeting $I$ is regular and finite.
Let $K$ be the union of leaves meeting $I$.
Then $\pi_\mu\m(K)$ is  a union of regular leaves of $\Sigma$.
Since $\mu$ is a limit of measures supported on singular leaves of $\Sigma$,
$\mu(\pi_\mu\m(I))=0$, a contradiction.
\end{proof}

\subsection{Moves}\label{sec_moves}

We describe moves that can be performed on a band complex together with a prelamination.
These combinatorial moves can be performed by the algorithm. 
They can be performed in the forward direction on a prelaminated band complex,
or in the backward direction on a naked band complex (without prelamination).
Each of these forward moves has a counterpart which can be applied to a measured foliation.
Performing a move on the foliation will give a way to uniformly perform a move on a sequence of prelaminations represented by this foliation.

Each band complex and prelamination comes with a set of rational constraints in standard form, described by
$\rho:\fmip\ra\calm$ and a tuple $\ul m\in\calm^{E(\call)}$.
We always assume that $\ul m$ is invariant under the holonomy and shows no incompatibility of rational constraints (see section \ref{sec_incompatibility}).
All the moves will keep $\rho$ and $\calm$ unchanged.

We introduce inert moves before restriction and extension moves (see Section \ref{sec_shortening} for definitions).

\begin{figure}[htbp]
  \centering
  \includegraphics{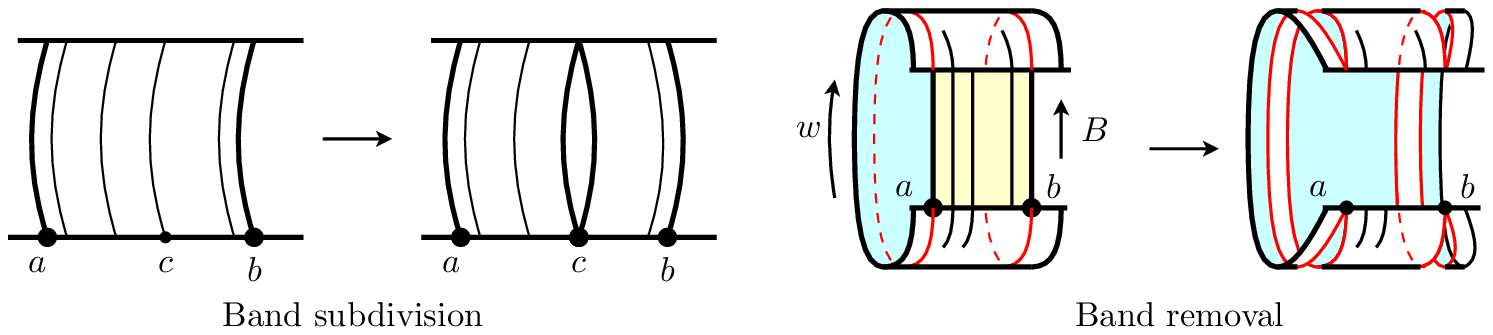}
  \caption{Moves\label{fig_move1}}
\end{figure}

\subsubsection{Band subdivision}

Let $\Sigma$ be a  prelaminated band complex with domain $D$, 
$\call$ a prelamination on $\Sigma$,
$B=[a,b]\times[0,1]$ a band of $\Sigma$,
and $\{c\}\times [0,1]\subset B$ a leaf segment of $\call$ with $c\notin\{a,b\}$ (see Figure \ref{fig_move1})

Let $\Sigma'$ be the band complex with domain $D'=D$, 
obtained by replacing the band $B$
by two bands $B'_a=[a,c]\times[0,1]$ and $B'_b=[c,b]\times [0,1]$, the attaching maps being the
restrictions of the attaching maps of $B$.
The twisting morphisms of $B'_a$ and $B'_b$ are the same as the twisting morphism of $B$.
The prelamination $\call$ naturally induces a prelamination $\call'$ on $\Sigma'$.
The set $\ul m\in\calm^{E(\call)}$ of rational constraints associated to $\call$
can be seen as a set of rational constraint on $\call'$ since $E(\call')=E(\call)$.
This induces a set of rational constraints on $\Sigma'$ by writing each elementary segment of $\Sigma'$
as a concatenation of elementary segments of $\call$.

It is clear that any solution of $\Sigma$ inducing $\call$ defines a solution of $\Sigma'$
inducing $\call'$, and that any solution of $\Sigma'$ induces a solution of $\Sigma$.
Thus,  $(\Sigma,\call)\ra (\Sigma',\call')$ is an inert move.\\

If we are given a measured foliation $(\calf,\mu)$ on $\Sigma$ instead of a prelamination,
together with a leaf segment $\{c\}\times [0,1]\subset B$,
we can similarly subdivide the band to obtain a new measured foliation $(\calf',\mu')$ on a new band complex $\Sigma'$.
Moreover, if $\calf$ represents an increasing sequence of prelaminations $\call_i$,
and if $\{c\}\times [0,1]$ is in a special leaf of $\calf$ (as in Definition \ref{dfn_represent}), 
then for all $i$ large enough, $\{c\}\times [0,1]$ is contained in a leaf segment of $\call_i$,
so one can perform the band subdivision move on $(\Sigma,\call_i)$, and
get an increasing sequence of prelaminations $\call'_i$ on $\Sigma'$ represented in $\calf'$.

\begin{dfn}\label{dfn_compatible_with_Li}
Assume that the topological foliation $\calf$ represents $\call_i$.

We say that a move $(\Sigma,\calf)\ra(\Sigma',\calf')$ is \emph{compatible} with $\call_i$ if 
for all $i$ large enough, one can perform the corresponding combinatorial move $(\Sigma,\call_i)\ra(\Sigma',\call'_i)$
(including rational constraints) and $\call'_i$ is represented by $\calf'$.
 \end{dfn}

We proved:
\begin{lem}\label{lem_band_subdiv}
  Assume that $\calf$ represents a sequence of prelaminations $\call_i$.
If $e$ is a leaf segment of $\calf$ contained in a special leaf,
then band subdivision along $e$ is compatible with $\call_i$.
\end{lem}

\subsubsection{Band removal}

Let $(\Sigma,\call)$ be a prelaminated band complex with domain $D$.
Assume that there is an oriented band $B=[a,b]\times[0,1]$ 
with bases $J_0=[a_0,b_0],J_1=[a_1,b_1]$,
and a $\Sigma$-word $w$ not involving the band $B$ or its inverse, 
whose holonomy map $h_{w}$ is defined and coincides with the holonomy of $B$ on $\{a_0,b_0\}$
(see Figure \ref{fig_move1}).

Let $(\Sigma',\call')$ be the prelaminated band complex obtained by
performing on $(\Sigma,\call)$ a band subdivision for each leaf segment of 
$[a,w)$ and $[b,w)$, then by removing the band $B$ and all the leaf segments contained in $B$.
Clearly, any solution of $\Sigma$ inducing $\call$ is a solution of $\Sigma'$ inducing $\call'$.

If the twisting morphisms of $w$ and $B$ agree, then conversely, 
any solution of $\Sigma'$ (without prelamination) is a solution of $\Sigma$.
Consider the case where the twisting morphism $\phi=\phi_w\m\phi_B$ is non-trivial in $\AutS$.
Let $m_{J_0}\in \calm$ be the element induced by the rational constraint $\ul m$ on $J_0$.
Since there is no incompatibility of rational constraints, 
$m_{J_0}\in\rho(\Fix\phi)$.
By definition of standard form for rational constraints, $\Fix\phi$ is represented by $\rho$ so
$\Fix\phi=\rho\m(\rho(\Fix\phi))$.
If follows that $\rho\m(m_{J_0})\subset \Fix\phi$
so the rational constraint $m_{J_0}$ alone imposes the fact that any solution
$\sigma'$ of $\Sigma'$ satisfies $\sigma_{J_0}\in\Fix\phi$.
Thus, in this case too, any solution of $\Sigma'$ is a solution of $\Sigma$.

If we are given a measured foliation $(\calf,\mu)$ on $\Sigma$ instead of a prelamination,
and if $h_{w}$ is defined and coincides with the holonomy of $B$ on $\{a_0,b_0\}$ as above,
we can similarly perform this move to obtain a new measured foliation $(\calf',\mu')$ 
on a new band complex $\Sigma'$.
If $\calf$ represents an increasing sequence of prelaminations $\call_i$ without incompatibility of rational constraints,
then $a_0$ and $b_0$ lie in singular, hence special leaves.
Thus, for $i$ large enough, the $\call_i$-holonomy of $w$ will be defined on $a_0$ and $b_0$,
and we will be able to perform the band removal move on $\call_i$. 
This defines an increasing sequence of prelaminations $\call'_i$ on $\Sigma'$ represented in $\calf'$.
Thus, the move on the foliated band complex uniformly induces moves on the prelaminations.
We thus proved:

\begin{lem}\label{lem_band_removal}
  Assume that $\calf$ represents a sequence of prelaminations $\call_i$ without incompatibility of rational constraints.
Assume that $B$ is an oriented band and $w$ a $\Sigma$-word not involving $B^{\pm 1}$
whose holonomy maps one base of $B$ exactly to the other base of $B$.

Then the band removal move on $\calf$ is compatible with $\call_i$.
\end{lem}

\subsubsection{Domain cut}

Let $(\Sigma,\call)$ be a prelaminated band complex with domain $D$, and consider $x\in D\setminus \partial D$ a $\call$-subdivision point, 
and assume that $x$ is non-orphan.
Let $B_1,\dots,B_k$ the oriented bands containing $x$ in the interior of their domain (one may have $B_i=B_j\m$ for some $i\neq j$).
By performing a band subdivision of the band $B_i$ along the leaf segment $[x,B_i)$, we can assume
that $x$ is in the interior of no base of band.
Now cut $D$ along $x$, \ie replace $D$ by $D'$ the disjoint union of the closure of the connected components of $D\setminus \{x\}$.
The attaching maps of the bands are still well defined, and $\call$ induces a prelamination $\call'$ on the obtained band complex $\Sigma'$.

It is clear that any solution of $\Sigma$ inducing $\call$ defines a solution of $\Sigma'$
inducing $\call'$, and that any solution of $\Sigma'$ induces a solution of $\Sigma$.
Thus,  $(\Sigma,\call)\ra (\Sigma',\call')$ is an inert move.
It can also be viewed as a restriction move for the map $\iota:D'\ra D$ induced by inclusion
as long as $x$ lies in the active part of $D$, the backwards move $\Sigma'\ra \Sigma$ being an extension move
of Lipschitz factor $1$.

This move generalises naturally to a foliation $\calf$ instead of a prelamination, 
and if $x$ lies in a special leaf, then it is non-orphan in some $\call_i$, so
it uniformly induces domain cut moves on a sequence of prelaminations represented by $\calf$:

\begin{lem}\label{lem_domain_cut}
  Assume that $\calf$ represents a sequence of prelaminations $\call_i$.

If $x\in D$ lies in a special leaf of $\calf$,
then domain cut at $x$ is compatible with $\call_i$.
\end{lem}

\subsubsection{Pruning move}

The pruning move for a prelaminated band complex has been described in section \ref{sec_shortening}.
This move generalises naturally to a foliation $\calf$ instead of a prelamination, 
and this move is always compatible with a sequence of prelaminations represented by $\calf$.

\subsubsection{Forgetful move}

Let $(\Sigma,\call)$ be a prelaminated band complex with domain $D$. 
Assume that there is a connected component $D_0$ of $D$ 
which contains exactly two bases of bands $B_1,B_2$ with $B_1\neq B_2^{\pm1}$, and such that 
$\dom B_1\cup \dom B_2=D_0$ and $I=\dom B_1\cap \dom B_2$ is a non-degenerate segment.

By performing a pruning move on the components of $D_0\setminus I$, we can assume that
$\dom B_1=\dom B_2=D_0$.
Let $\Sigma'$ the band complex with domain $D\setminus D_0$, and where the bands $B_1,B_2$ are removed and replaced by
a single band $B'$ with holonomy $h_{B_2}\circ h_{B_1}\m$ and 
twisting morphism $\phi_{B_2}\circ \phi_{B_1}\m$. 
The prelamination $\call$ induces a prelamination $\call'$ on $B'$,
whose set of subdivision points is $D_\call\cap D'$, and such that leaf segments of $B'$ are exactly the segments $[x,B')$
such that $[x,B_2B_1\m)$ is a leaf path of $\call$.

It is clear that any solution of $\Sigma$ inducing $\call$ defines a solution of $\Sigma'$
inducing $\call'$, and that any solution of $\Sigma'$ induces a solution of $\Sigma$.
Thus, for any partition $D=D_I\cap D_A$ where $D_A$ contains both bases of $B_1$ and of $B_2$,
 the move $(\Sigma,\call)\ra (\Sigma',\call')$ is a restriction move with respect to the inclusion $\iota:D'\hookrightarrow D$
and $\Sigma'\ra \Sigma$ is an extension move of Lipschitz factor $3$.

This move generalises naturally to a foliation $\calf$ instead of a prelamination, 
and it uniformly induces forgetful moves on a sequence of prelaminations represented by $\calf$.

\begin{lem}\label{lem_forgetful}
  Assume that $\calf$ represents a sequence of prelaminations $\call_i$.
Consider $D_0$ a connected component of $D$ which contains exactly two bases of bands $B_1,B_2$ with $B_1\neq B_2^{\pm1}$, and such that 
$\dom B_1\cup \dom B_2=D_0$, and $\dom B_1\cap \dom B_2$ is a non-degenerate segment.

Then the forgetful move on $\calf$ (consisting in removing $D_0$) is compatible with $\call_i$.
\end{lem}

\subsection{Minimal components}\label{sec_minimal}

A Rips band complex  has a decomposition into a finite union of \emph{minimal components} (\cite{BF_stable,GLP1}).
In this subsection, we introduce $\mu$-minimal components of the topological foliation $\Sigma$ which mimic
the minimal components of the corresponding Rips band complex $\Sigma_\mu$.
We prove in this subsection an important result showing that the measure of an interval $I$
controls the relative length of any solution inducing $\call_i$ for $i$ sufficiently large.

Consider a foliation $\calf$ on a band complex $\Sigma$ with domain $D$.
Recall  that the singular leaf segments of $\calf$ are the leaf segments in the boundary of the bands of $\Sigma$,
and that a leaf path is regular if it does not involve any singular leaf segment.
We say that two points are in the same $\calf$-leaf if they can be connected by a regular leaf path.

Thus, two points of $\Sigma_\mu$ are in the same $\rond\calf_\mu$-leaf if they can be joined by a regular leaf path of $\calf_\mu$
(but the projection in $\Sigma_\mu$ of a regular leaf path of $\calf$ may fail to be regular in $\Sigma_\mu$).
It is convenient to call $\mu([x,y])$ the $\mu$-distance between $x$ and $y$.

\begin{dfn}
Let $(\Sigma_\mu,\calf_\mu)$ be a Rips band complex 
such that no component of $D$ and no base of band is reduced to one point.

We say that $(\Sigma_\mu,\calf_\mu)$ is \emph{minimal} if
for each $x\in D_\mu\setminus \partial D_\mu$, every $\rond{\calf_\mu}$-leaf is dense in $\Sigma_\mu$.
\end{dfn}

\begin{dfn}
A  measured foliation $(\calf,\mu)$ on a band complex $\Sigma$ with domain $D$
is called \emph{$\mu$-minimal} if all connected components of $D$ and all bands of $\Sigma$ have positive measure, and
\begin{itemize}
\item $\Sigma_\mu$ is minimal (as a Rips band complex)
\item for each $x\in\partial D$, there exists a $\Sigma$-word whose holonomy 
is defined on an interval of positive measure, and maps $x$
at positive $\mu$-distance of $\partial D$.
\end{itemize}
\end{dfn}

\begin{rem}
  One easily checks that all leaves of a $\mu$-minimal band complex are infinite, but this may be false if
one removes the 
 last condition of the definition.
Leaves of a $\mu$-minimal band complex may in general fail to be dense, as there may be wandering intervals.
\end{rem}

Assume that $\Sigma$ is a $\mu$-minimal foliated band complex representing a sequence of prelaminations $\call_i$.
Given a solution $\sigma$ of $\Sigma$ 
inducing $\call_i$ for $i$ large enough, one can control the ratio 
${\abs{\sigma_I}}/{\abs{\sigma}}$ in terms of the measure $\mu(I)$:

\begin{lem} \label{lem_certif_mesure}
Let $(\Sigma,\calf,\mu)$ be a $\mu$-minimal foliated band complex with domain $D$.
Let $\call_i$ be a sequence of prelaminations represented by $\calf$.

For all $\eps>0$, there exists $\eta>0$ such that the following holds:
for each segment $I\subset D$ with $\mu(I)<\eta$,
and for $i$ large enough, there exists $I'\supset I$ adapted to $\call_i$,
and a certificate of $\eps$-shortness of $I'$ relative to $D$ as in Definition \ref{dfn_certif_shortness}
(ensuring that $\frac{\abs{\sigma_I}}{\abs{\sigma}}\leq \eps$ by Lemma \ref{lem_certif_shortness}).
\end{lem}

\begin{proof}
Fix $\eps>0$, and $N\geq\frac{1}{\eps}$. 
For $x\in D_\mu\setminus \partial D_\mu$, the $\rond\calf_\mu$-leaf of $x$ is infinite,
so we can find $N$ distinct $\Sigma_\mu$-words $w_1,\dots,w_N$ 
whose holonomies are defined on a neighbourhood of $x$, and map $x$ to $N$ distinct points of $D_{\mu}$.
Take $\eps_1,\eps_2>0$ small enough so that
the holonomies $h_{w_i}$ are defined on $[x-\eps_1,x+\eps_2]$ and
map this segment to $N$ disjoint intervals.
Since every leaf is dense in $\Sigma_\mu$, one can assume moreover that $x-\eps_1$ and $x+\eps_2$ lie in singular leaves of $\Sigma_\mu$.
Let $I_x\subset D$ whose endpoints are in singular leaves of $\Sigma$, and such that $\pi_\mu(I_x)=[x-\eps_1,x+\eps_2]$.
Then for $i$ large enough, $I_x$ is adapted to $\call_i$, and has a certificate of $\eps$-shortness
relative to $D$.

Consider $x\in\partial D_{\mu}$ and $\Tilde x$ its preimage in $\partial D$.
Consider $w$ a $\Sigma$-word whose $\calf$-holonomy is defined on an interval of positive measure, and
maps $\Tilde x$ at positive $\mu$-distance from $\partial D$. 
Let $\Tilde y=h_w(\Tilde x)$, and $y=h_w(x)$ the image of $\Tilde y$ in $D_\mu$.
Consider the interval $I_y\subset D$ defined above and define $I_x=h_w\m(I_y)$. The endpoints of $I_x$ 
are in singular leaves of $\Sigma$, so  $I_x$ is adapted to $\call_i$ for $i$ large enough, and has a certificate of $\eps$-shortness
relative to $D$.

Consider the covering of $D_{\mu}$ by the interiors in $D_{\mu}$ of $\pi_\mu(I_x)$,
and by compactness, consider $\eta$  the Lebesgue number of this covering.
Then any segment $J\subset D_k$ of measure at most $\eta$ is contained in the interior of some $I_x$,
and the lemma follows.
\end{proof}

\begin{prop}\label{prop_minimal}
Consider a band complex $\Sigma$ with a measured foliation $(\calf,\mu)$ without atom
representing a sequence of prelamination $\call_i$.
Assume that there is a bound on the exponent of periodicity of $\call_i$.
 
Then one can perform a finite number of inert moves 
on $(\Sigma,\calf)$ compatible with $\call_i$
so that the obtained foliated complex $(\Sigma',\calf')$ 
is  the disjoint union of finitely many $\mu$-minimal band complexes $\Lambda_1,\dots,\Lambda_p$ 
together with finitely many bands and components of $D'$ of measure 0.
\end{prop}
 
The exponent of periodicity of $\call_i$ is defined in Definition \ref{dfn_exponent}.

\begin{proof}
We first recall the decomposition into minimal components for $\Sigma_\mu$.
Let $\calc_\mu\subset\Sigma_\mu$ be the finite graph defined by the union of
singular finite $\rond\calf_\mu$-leaves (a  $\rond\calf_\mu$-leaf is \emph{singular}
if it contains a subdivision point of  $\Sigma_\mu$, or equivalently, if the  $\calf_\mu$-leaf containing it is singular).
This graph is finite because $\Sigma_\mu$ has only finitely many vertices.

Denote by $\Lambda_{\mu,1},\dots,\Lambda_{\mu,p}$ the Rips band complexes obtained as 
the closure of the connected components of $\Sigma_\mu\setminus\calc_\mu$.
Equivalently, $\Lambda_{\mu,i}$ are the connected components of the band complex obtained 
from $\Sigma_\mu$
by  performing
band subdivision and domain cuts at edges and vertices of $\calc_\mu$,
and by removing bands of width $0$.
By \cite{GLP1,BF_stable}, either $\Lambda_{\mu,i}$ is minimal, or
$\Lambda_{\mu,i}$ is a \emph{family of finite leaves} (all leaves of $\Lambda_{\mu,i}$ are finite in particular).
In our setting, families of finite leaves cannot occur by Lemma \ref{lem_finite_leaves}.

We now construct a lift $\calc$ of $\calc_\mu$ in $\Sigma$.
Working independently on each connected component on $\calc_\mu$, we can assume that $\calc_\mu$ is connected.
Consider $x_\mu\in \calc_\mu$, and $\Tilde x\in D$ be a point whose leaf is special 
with $\pi_\mu(\Tilde x)=x_\mu$. Existence of $\Tilde x$ follows from Lemma \ref{lem_singular} 
 since singular leaves are special.
Consider $T$ a maximal tree of $\calc_\mu$.
Given two vertices $v_1,v_2\in T$, we denote by $[x_1,x_2]_T$ the path of $T$ joining them.
First, we can lift $T$ in $\Sigma$. Each segment $[x_\mu,v]_T$ 
defines a $\Sigma$-word $w_v$, and $\Tilde x$ lies in the interior of $\dom w_v$ 
since $\calc_\mu$ contains no singular leaf segment.
Therefore, one can define $\Tilde T\subset\Sigma$ as the union of the leaf paths $[\Tilde x,w_v)$. 
For each oriented edge $e\subset \calc_\mu\setminus T$ joining $v_1$ to $v_2$,
consider $B_e$ the oriented band containing it,
and the $\Sigma$-word $w_e=w_{v_2}\m B_e w_{v_1}$. 
To lift $e$, we need that the endpoint of the leaf segment $[\Tilde v_1,B_e)$ coincides with $\Tilde v_2$, where $\Tilde v_i$
denotes the lift of $v_i$ in $\Tilde T$.
Equivalently, we need to check that for all edge $e$ of $\calc_\mu\setminus T$,
$h_{w_e}(\Tilde x)=\Tilde x$.

Let $I\subset D$ be the preimage of $x_\mu$ under $\pi_\mu$.
Since $\Tilde x$ and $h_{w_e}(\Tilde x)\in I$, we are done if $I$ is reduced to one point.
As the holonomy in $\Sigma_\mu$ of every loop in $\calc_\mu$ based at $x_\mu$ is defined on a neighbourhood of $x_\mu$,
this defines in $\Sigma$ a morphism $\psi:\pi_1(\calc,x_\mu)\ra \Homeo(I)$.
Let $H$ be the image of $\psi$ and $H_+<H$ the orientation preserving subgroup.

We claim that for each point $\Tilde y\in I$ whose leaf is special,
$h.\Tilde y=\Tilde y$ for all $h\in H_+$. Otherwise, one can assume for instance $h.\Tilde y>y$,
so $h^n. \Tilde y> \Tilde y$ for all $n>0$.
Moreover, up to changing $h$ to a power, we may assume that $h=\psi(w)$ for some $\Sigma$-word $w$
with trivial twisting morphism. Since $\Tilde y$ lies in a special leaf of $\calf$,
for any $n>0$ there exists $i$ such that the all the intervals $[\Tilde y,h.\Tilde y],\dots,[h^{n-1}.\Tilde y,h^n.\Tilde y]$
are equivalent under the holonomy of $\call_i$, so $Exponent(\call_i)\ra \infty$, a contradiction.

Let $w\in\pi_1(\calc,x_\mu)$ be a $\Sigma$-word whose holonomy does not fix $\Tilde x$. 
Since $\Tilde x$ lies in a singular hence special leaf of $\calf$,
the homeomorphism $h=\psi(w)$ reverses the orientation. Let $m\in I$ be its unique fixed point.
The leaf through $m$ is special since $w$ and $[x_\mu,h_w(x_\mu)]$ define a M\"obius strip in some $\call_i$,
whose core needs to appear in some $\call_j$. The claim above shows that $H_+.m=m$, 
and since $H=\grp{H_+,h_w}$, $H.m=m$.
This means that choosing $m$ instead of $\Tilde x$ as a lift of $x_\mu$, 
one can lift $\calc_\mu$ to $\Sigma$.

Let $\calc\subset\Sigma$ be the graph obtained by lifting each connected component of $\calc_\mu$ as above.
Let $(\Sigma',\calf')$ be the measured foliated band complex with domain $D'$ obtained from $(\Sigma,\calf)$ by first performing 
a band subdivision at each edge of $\calc$ and then a domain cut at each vertex of $\calc$.
We still denote by $\mu$ the obtained transverse measure on $\Sigma'$.
The band subdivisions along edges of $\calc$ are mirrored by band subdivisions on $\Sigma_\mu$ since edges of $\calc_\mu$ are at positive
$\mu$-distance of the boundary of the band containing them.
This is not true for the domain cuts:
given a vertex $x\in\calc$, and an oriented band $B$ containing $x$ in the interior of domain,
the leaf segment $[x,B)$ may fail to lie in $\calc$, in which case
$x$ is at zero $\mu$-distance from an endpoint of $\dom B$; in this case,
the domain cut at $x$ will split the band $B$ into two bands, one of them
having measure $0$. 
Thus, $\Sigma'_{\mu}$ differs from $\Sigma_\mu$
only by the potential presence of additional bands of measure 0.
Let $D'^0$ be the union of connected components of $D'$ with positive measure, and
$\Sigma'^0$ the band complex on $D'^0$ whose bands are the bands of $\Sigma'$ with positive measure.
Then $\Sigma'^0$ is a disjoint union of foliated band complexes $\Lambda'_1,\dots\Lambda'_p$,
such that $\Lambda'_{1,\mu},\Lambda'_{p,\mu}$ are precisely the minimal components of $\Sigma_\mu$.

We need to perform some more operations on $\Sigma'$ to make sure that $\Lambda'_i$ satisfies the second condition of $\mu$-minimality.
We focus on a component $\Lambda'_i\subset \Sigma'$, and to make notations lighter, we use the notations $\Lambda=\Lambda'_i$,
and by $D\subset D'$ its domain.
Consider the graph $\cald_\mu\subset \Lambda_\mu$ with vertex set $\partial D_\mu$, 
and whose edges are the leaf segments of $\Lambda_\mu$
having both endpoints in $\partial D_\mu$.
Working independently on each component of $\cald_\mu$, we can assume that $\cald_\mu$ is connected.
There exists a vertex $x_\mu\in\cald_\mu$ and an oriented band $B_\mu$ of $\Sigma_\mu$ sending $x_\mu$ into the interior of $D_\mu$:
otherwise leaves through the points close to the the vertex set of $\calc_\mu$ would be compact.

Let $\Tilde x$ be the (unique) endpoint of $D$ projecting to $x_\mu$.
A loop in $\cald_\mu$ based at $x_\mu$  defines a $\Sigma$-word $w$ whose holonomy $h_w$ in $\Sigma$
is defined on an interval of positive measure, and contains a point at $\mu$-distance 0 from $\Tilde x$.
Moreover, $h_w$ preserves the orientation.
We claim that for each $y\in\dom w$ and in a special leaf, $h_w(y)=y$.
Indeed, 
write $\dom w=[a,b]$, and assume for instance $h_w(y)>y$.
If $\mu[y,b]>0$, then $h_{w^k}(y)$ is defined for all $k>0$ since $\mu([y,h_w(y)])=0$;
if $\mu[y,b]=0$, then up to changing $w$ to $w\m$ and $y$ to $h_w(y)$, the same conclusion holds. 
This contradicts
the bound on the exponent of periodicity as above.

We construct a lift $\cald$ of $\cald_\mu$.
Let $T$ be a maximal tree of $\cald_\mu$. For each $v\in\cald_\mu$, let $w_v$ be the $\Sigma$-word
representing the segment $[x_\mu,v]_T$.
For each oriented edge $e\in \cald_\mu\setminus T$ joining $v_1$ to $v_2$, let $B_e$ be the oriented band containing it,
and $w_e=w_{v_2}\m B_e w_{v_1}$.
Let $I\subset D'^0$ be the intersection of all domains $\dom w_v$, $\dom w_e$, and $\dom B$.
Since all these domains have positive measure and a point a $\mu$-distance 0 from $\Tilde x$,
$I$ has an endpoint $\Tilde y\in \pi_\mu\m(x_\mu)$.
The leaf through $\Tilde y$ is singular hence special, so $h_{w_e}(\Tilde y)=\Tilde y$ for all $e$ by the claim above.
Thus, the union of all leaf paths $[\Tilde y,w_e)$, $[\Tilde y,w_v)$ 
defines a lift $\cald$ of $\cald_\mu$.

Perform on $\Sigma$ band subdivisions at all edges of $\cald$ and domain cuts at all vertices of all connected components of $\cald$,
and denote by $\Sigma''$ with domain $D''$ the obtained foliated band complex. 
Let $D''^0$ be the union of connected components of $D''$ with positive measure, and
$\Sigma''^0$ the band complex on $D''^0$ whose bands are the bands of $\Sigma''$ with positive measure.
Each band of $\Sigma'$ has been cut into
one band in $\Sigma''^0$, and at most two bands of measure 0 (at most one on each side).
A similar fact holds for connected components of $D'$.
Thus $\Lambda$ has been transformed to $\Lambda''\subset\Sigma''^0$ together with some intervals and bands of measure $0$.

We claim that $\Lambda''$ is a $\mu$-minimal component. 
Denote by $D_{\Lambda''}\subset D"^0$ the domain of $\Lambda''$.
Because of the cutting operation performed, for every endpoint $x$ of a base of a band of $\Lambda''$,
either $x$ lies in $\partial D_{\Lambda''}$, or it lies at positive $\mu$-distance from it.
Since for every endpoint $x_\mu$ of $D_\mu$ there is a $\Sigma_\mu$-word defined on a set of positive measure
mapping $x_\mu$ to the interior of $D_\mu$, this word lifts to a $\Lambda''$-word mapping the corresponding endpoint of
$D_{\Lambda''}$ at positive $\mu$-distance from $\partial D_{\Lambda''}$.

Since all band subdivision moves and all domain cuts have been performed on special leaves of $\calf$,
these moves are compatible with $\call_i$, and the proposition follows.
\end{proof}

\subsection{Homogeneous components}\label{sec_homogeneous}

 Recall that two points of $\Sigma_\mu$ are in the same $\rond\calf_\mu$-leaf if they can be joined by a regular leaf path of $\calf_\mu$.

\begin{dfn}
  One says that $\Sigma_\mu$ contains a \emph{homogeneous component} (also called \emph{axial} or \emph{toral} component) 
if there exists a non-degenerate open interval
$I\subset D_\mu$ and a subgroup $P\subset\Isom(\bbR)$ with dense orbits such that for all $x,y\in I$,
$x,y$ are in the same $\rond\calf_\mu$-leaf if and only if $x,y$ are in the same $P$-orbit.
\end{dfn}

When a minimal component satisfies this definition, it is called a homogeneous component.
If $\Sigma_\mu$ contains a homogeneous component as in the definition above, one of its minimal components
is a homogeneous component.

\begin{prop}\label{prop_homogene}
  If $\Sigma_\mu$ contains a homogeneous component, then the exponent of periodicity of $\call_i$
goes to infinity as $i$ goes to infinity.
\end{prop}

Given a $\Sigma_\mu$-word $w$ whose holonomy is a translation of positive length $t(w)$, whose
domain of definition has positive measure,
its \emph{translation ratio} is the ratio $tr(w)=\mu(\dom(w))/t(w)$.
The proposition is based on the following fact from Rips' Theory (of which we give a proof below):

\begin{fact*}
  Assume that a system of isometries $\Sigma_\mu$ contains a homogeneous component.

Then there are $\Sigma_\mu$-words with arbitrarily large translation ratio.
\end{fact*}

\begin{proof}[Proof of Proposition \ref{prop_homogene}]
Let $k$ be the cardinal of $\AutS$.
Since $tr(w^k)\geq \frac{1}{k}tr(w)-1$, there are $\Sigma_\mu$-words
with trivial twisting whose translation ratio is arbitrarily large.
Consider such a $\Sigma_\mu$-word $w$ with translation ratio at least $N$, and view it as a $\Sigma$-word.
Let $x\in D$ be an endpoint of $\dom w$.
Since $tr(w)\geq N$, up to changing $w$ to $w\m$, we can assume that
for all $k=0,\dots, N-1$, $x\in\dom w^k$.
Since the leaf through $x$ is singular hence special, 
for all $i$ large enough,
all intervals  $[x,h_w(x)]$,\dots,  $[h_w^{N-2}(x),h_w^{N-1}(x)]$ 
are equivalent under the holonomy of some ${w^j}$. 
The exponent of periodicity of $\call_i$ then goes to infinity with $i$.
\end{proof}

\begin{proof}[Proof of the fact]
Given $\eta>0$, consider $\Sigma'$ the Rips band complex on $D_\mu$ obtained by 
narrowing each band by $\eta$ on each side. More formally, this amounts to restricting
$f_{B,\eps}:[a,b]\times\{\eps\}\ra D$ to $[a+\eta,b-\eta]\times\{\eps\}$.
  If $\Sigma_\mu$ contains a homogeneous component, then one can find $\eta>0$ such that
$\Sigma'$ has an infinite orbit (\cite{GLP1}).
In particular, there exists $\Sigma'$-words whose holonomy is a translation of arbitrarily small length.
Fix some large $M\in \bbR$, and consider $w$ whose holonomy is a translation of length at most $\eta/M$.
The holonomy of corresponding $\Sigma_\mu$-word is a translation of the same length,
and its domain has measure at least $2\eta$, so $tr(w)\geq 2M$.
\end{proof}

\subsection{Independent bands}\label{sec_indep}

A theorem by Gaboriau asserts that if a Rips band complex $\Sigma_\mu$ has no homogeneous component,
one can narrow its bands so that its bands are \emph{independent}, meaning that
no $\rond\calf_\mu$-leaf contains a cycle.
This result was also used in \cite{GLP1} or \cite[Section 6]{Gui_approximation} for the analysis of actions on $\bbR$-trees.
Our goal is to mirror this fact in the topological foliated band complex $(\Sigma,\calf)$, 
and into the prelaminated band complexes $(\Sigma,\call_i)$.\\

Given a foliated band complex $(\Sigma,\calf)$ with domain $D$, an \emph{orbit} of $(\Sigma,\calf)$ is the intersection of a leaf
with $D$.

\begin{thm}[\cite{Gab_indep}]\label{thm_Gab}
Given $\Sigma_\mu$ a Rips band complex with domain $D$ without homogeneous component,
there exists a Rips band complex $\Sigma'_\mu\subset \Sigma_\mu$ with same domain $D$, obtained by replacing each band $B=[a,b]\times [0,1]$ of $\Sigma$
by a narrower band $B'=[a',b']\times[0,1]\subset B$, and such that
\begin{enumerate*}
\item $\Sigma_\mu$ and $\Sigma'_\mu$ have the same orbits
\item $\Sigma'_\mu$ has independent bands: no $\rond\calf'_\mu$-leaf contains a cycle
\item any singular leaf of $\Sigma'_\mu$ is contained in a pseudo-singular of $\Sigma_\mu$.
\end{enumerate*}
\end{thm}

\begin{rem}
  If bands of $\Sigma'_\mu$ are independent, every pseudo-singular leaf is singular.
\end{rem}

The last condition is not stated in Gaboriau's theorem but can be ensured by slightly modifying Gaboriau's construction.

\begin{proof}
We explain how to ensure that the set of pseudo-singular orbits of $\Sigma_\mu$ 
does not increase
during Gaboriau's construction in the case where $\Sigma_\mu$ is a minimal component.
For convenience, we give references to statements in Gaboriau's paper. 
To lighten notations, we drop the 
index $\mu$ so we denote by $\Sigma$ the initial Rips band complex.
Clearly, one can ensure that condition 3 holds if all orbits are finite (\cite[Lemme 5.2(1)]{Gab_indep}).
The proof of Gaboriau's theorem (see \cite[section 7]{Gab_indep}) consists in modifying the band complex
$\Sigma$ by narrowing then re-enlarging bands.
Instead of narrowing  bands by a uniform amount, 
we need to choose narrowing widths so that we don't create new pseudo-singular leaves.

The set of sides of bands (\ie the set of singular leaf segments) is the set of edges $E(\Delta)$ of the finite graph $\Delta$ of \cite{Gab_indep}.
Given a tuple $\ul t=(t_\eps)_{\eps\in E(\Delta)}$ of (small enough) non-negative numbers, one can define $\Sigma_{\ul t}$
by narrowing each band by $t_\eps$ on the side corresponding to $\eps$.

Consider $\calt$ a maximal subforest of $\Delta$, 
a small tuple $\ul t$ so that 
$t_\eps=0$ for $\eps\in \calt$, and for $\eps\notin \calt$, choose $t_\eps>0$ 
so that the boundary of the new band lies in a singular leaf of $\Sigma$.
Moreover, $\ul t$ should be small enough so that $\Sigma$ and $\Sigma_t$ have the same orbits (\cite[Prop. 4.5]{Gab_indep}).
This is possible because the union of singular leaves is dense.

Let $t_1$ be the minimum of $t_\eps$ for $\eps\notin\calt$,
and consider $\ul u=(u_\eps)_{\eps\in E(\Delta)}$ a small tuple with $u_\eps=0$ for $\eps\notin \calt$
and $0<u_\eps\leq t_1$ for $\eps\in\calt$ so that the singular leaves of $\Sigma_2=\Sigma_{\ul u+\ul t}$ 
are contained in singular leaves of $\Sigma$.
Since $\Sigma$ is non-homogeneous, and since all coordinates of $\ul u+\ul t$ are positive,
leaves of $\Sigma_{\ul u+\ul t}$ are compact (\cite[Th 3.3]{Gab_indep}).
Denote by $e(\Sigma)$ the measure of the space of leaves of a Rips band complex $\Sigma$ (this is $0$ if every leaf is dense),
and by $l(\Sigma)$ the sum of the transverse measure of its bands.
Then $e(\Sigma_{\ul t})=0$, and using \cite[Th II.E.4]{Gus_feuilletages},
$e(\Sigma_{\ul u+\ul t})=|\ul u|$ where $|\ul u|=\sum_{\eps\in\Delta} u_\eps$ (this is the generalisation we need of \cite[Prop. 4.3]{Gab_indep}).
Since $\Sigma_2$ has compact leaves, 
one can find $\ul v$ such that $v_\eps=0$ for $\eps\in \calt$,
and $\Sigma_3=\Sigma_{\ul u+\ul t+\ul v}$ has the same orbits as $\Sigma_2$, and has independent bands 
and has no new pseudo-singular leaves \cite[Lemme 5.2]{Gab_indep}.
Since bands are independent, $e(\Sigma_3)+l(\Sigma_3)=\mu(D)$ (\cite[Th. 6.3]{Gab_indep}).

Finally, consider $\Sigma_4=\Sigma_{\ul t+\ul v}$, obtained from $\Sigma_3$ by enlarging the bands by $\ul u$.
Then as in  \cite[section 7]{Gab_indep}, $\Sigma_4$ has the same orbits as $\Sigma$, so $e(\Sigma_4)=0$.
Since $l(\Sigma_4)+e(\Sigma_4)=l(\Sigma_3)+|\ul u|=\mu(D)$, its bands are independent \cite[Th. 6.3]{Gab_indep}.
Finally, one easily checks that any singular leaf of $\Sigma_4$ is contained in a pseudo-singular leaf of $\Sigma$.
\end{proof}

\begin{dfn}
Say that a measured foliation $(\calf,\mu)$ on $\Sigma$
 has $\mu$-independent bands if 
the holonomy of no non-trivial 
 reduced $\Sigma_\mu$-word fixes an interval of positive measure.
\end{dfn}

Our goal is the following proposition.

\begin{prop}\label{prop_indep}
Consider a band complex $\Sigma$ with a measured foliation $(\calf,\mu)$ without atom
representing a sequence of prelaminations $\call_i$.
Assume that there is a bound on the exponent of periodicity of $\call_i$, and that each $\call_i$
has an invariant combinatorial measure.

Then one can perform a finite number of inert moves on $(\Sigma,\calf)$ compatible with $\call_i$
so that the obtained foliated complex $(\Sigma',\calf')$ has $\mu$-independent bands.
\end{prop}

We will use the following useful finiteness property:
\begin{prop}[Segment-closed property, \cite{GLP2}]
Let $(\Sigma_\mu,\calf,\mu)$ be a Rips band complex, and $\phi:I\ra J$ a partial isometry between two segments
of $D_\mu$ such that for all $x\in I$, $x$ and $\phi(x)$ are in the same leaf.

Then there exists finitely many $\Sigma_\mu$-words $w_1,\dots,w_n$ whose domains cover $I$,
and whose holonomies coincide with $\phi$ on $I$.\qed
\end{prop}

\begin{proof}[Proof of Proposition \ref{prop_indep}]
For each band $B$ of $\Sigma$, we denote by  $B_\mu$ the corresponding band of $\Sigma_\mu$.
Consider $\Sigma'_\mu$ the Rips band complex on $D_\mu$ whose bands $B'_\mu\subset B_\mu$ are given by Theorem \ref{thm_Gab}.
Consider a band $B$ and $I,J$ (resp. $I_\mu,J_\mu$ and $I'_\mu,J'_\mu$) the bases of $B$ (resp. $B_\mu$, $B'_\mu$).
Let $K_\mu$ be the closure of a connected component of $I_\mu\setminus I'_\mu$.
Since $\Sigma_\mu$ and $\Sigma'_\mu$ have the same orbits, 
the holonomy $h_{B_\mu}$ of $B_\mu$ is a partial isometry such that 
$x$ and $h_{B_\mu}(x)$ are in the same $\Sigma'_\mu$-leaf.
By segment-closed property, there exist $\Sigma'_\mu$-words $w'_1,\dots,w'_n$ whose domains cover $K_\mu$,
and whose holonomies coincide with $h_{B_\mu}$ on $K_\mu$. 
We can assume that $\dom w'_i$ is not reduced to a point \ie has positive measure.

We now lift this situation to $\Sigma$.
For each band $B$ of $\Sigma$ whose bases have positive measure,
consider a restriction $B'$ of $B$ in $\Sigma$ whose base projects to $B'_{\mu}$ in $\Sigma_\mu$.
By Lemma \ref{lem_singular}, one can choose the endpoints of $B'$ in pseudo-singular leaves of $\calf$. 
By remark \ref{rem_pseudo}, these leaves are special.
Using the band subdivision move, one can subdivide the band $B$ along the boundary of the corresponding sub-band $B'$.
Doing this for each band of $\Sigma$  gives get a new band complex $\Sigma_1$
 whose set of bands is $\calb'\dunion \calb''$
where $\calb'$ is the set chosen lifts of bands of $\Sigma'_\mu$.

For each band $B''\in\calb''$, there are $\Sigma'_\mu$-words $w'_1,\dots,w'_n$ whose domains cover $\dom B''_\mu$,
and whose holonomies coincide with $h_{B''_\mu}$.
We view these words as $\Sigma_1$-words involving only bands of $\calb'$. 
The domains of $w'_i$ cover all $\dom B''$ except for some finite union of intervals of measure zero.
By further subdividing $B''$ along some singular leaves (defined as the boundary of the domain of some $w'_i$),
we can replace each band $B''$ by a finite set of smaller sub-bands $B'''$ so that either $\mu(\dom B''')=0$
or $\dom B'''\subset \dom w'_i$ for some $i$.

Fix some band $B'''$ whose domain $[a,b]$ has positive measure.
By construction, the holonomy of the word $u=w_i\m B'''$ induces the identity on $\pi_\mu([a,b])$.
If $h_u(a)\neq a$, then up to changing $u$ to $u\m$, $h_{u^k}(a)$ is defined for all $k>0$.
Since the leaf through $a$ is special, this clearly implies that $exponent(\call_i)$ goes to infinity.
It follows that $h_u$ fixes both $a$ and $b$,
so one can remove the band $B'''$ using  the band removal move.
The bands of the obtained foliated band complex $\Sigma_2$ are either bands of $\calb'$, or bands whose bases have measure zero.
Thus, $\Sigma_2$ has $\mu$-independent bands.
\end{proof}

\subsection{The pruning process}\label{sec_exotic}

Consider a foliated band complex $(\Sigma,\calf,\mu)$ 
with $\mu$-independent bands, 
decomposed into $\mu$-minimal components as in Proposition \ref{prop_minimal}. Let $\Lambda$ be a minimal component of this decomposition.

We recall the following pruning process (also known as process 1 of the Rips machine).
All the moves are performed on $\Sigma$, modifying only $\Lambda\subset \Sigma$.

We denote by $D_\Lambda=D\cap\Lambda$ the domain of $\Lambda$, which will be the \emph{active} part for the moves.
The pruning process consists in the iteration of the following steps:
\begin{itemize}
\item \emph{Step a.}
Say that an open interval $I\subset D_{\Lambda}$ is \emph{prunable} if every $x\in \rond{I}$ lies in exactly one base of a band 
of $\Sigma$.
If $\Lambda_\mu$ has no prunable interval of positive measure, the process stops.
Otherwise, let $I_1,\dots I_n$ be the complete list of maximal prunable intervals of positive measure, 
and perform a pruning move for each $I_i$.

\item \emph{Step b.} 
If some component of $D_\Lambda$ contains at most one base of band, prune this band, and iterate step b as many times as possible.

\item \emph{Step c.} 
Perform on all possible forgetful moves on $\Lambda$.
\end{itemize}

\begin{rem}
  In step a, note that each band of $\Sigma$ having a base in $D_\Lambda$ either lies in $\Lambda$ or has measure zero.
Step b and step c can be repeated only
  finitely many times in a row since they decrease the number of bands.
\end{rem}

Starting with  $(\Sigma_0,\calf_0,\mu_0)=(\Sigma,\calf,\mu)$,
we define inductively $(\Sigma_{j+1},\calf_{j+1},\mu_{j+1})$ by performing the three steps a,b,c on $(\Sigma_j,\calf_j,\mu_j)$.
The sequence of prelaminations $\call_i\subset \Sigma$ represented by $\calf$ induces
a sequence of prelaminations $\call_{j,i}\subset \Sigma_j$ represented in $\calf_i$.
Moreover, the rational constraints on $\call_i$ (in standard form) 
induce rational constraints on $\call_{j,i}$ which define
rational constraints $\ul m_j\in\calm^{E(\Sigma_j)}$ on $\Sigma_j$ independently of $i$
by compatibility of the rational constraints.
All rational constraints are in standard form, using the same morphism $\rho:\fmip\ra\calm$.

We denote by $\Lambda_j\subset\Sigma_j$ the subset corresponding to the minimal component under study. 
We denote by  $D_{\Sigma_j}\subset D$ and $D_{\Lambda_j}=D_{\Sigma_j}\cap \Lambda$
the domains of $\Sigma_j$ and $\Lambda_j$, viewed as subsets of $D$.

Since bands are independent (which excludes homogeneous components), surface and  exotic components 
can be defined by the fact that the pruning process stops or not (\cite{GLP1}):

\begin{dfn}
  Say that $\Lambda$ is \emph{exotic} (or Levitt, or thin), if the pruning process never stops.
Otherwise it is a \emph{surface} component.
\end{dfn}

\subsubsection{Exotic components}

We focus on an exotic minimal component $\Lambda$ of this decomposition.
The goal of this section is the following proposition:

\begin{prop}\label{prop_exotic}
Consider a foliated band complex $(\Sigma,\calf,\mu)$
with $\mu$-independent bands, 
decomposed into $\mu$-minimal components as in Proposition \ref{prop_minimal}.
Consider $\Lambda$ an exotic $\mu$-minimal component of $\Sigma$.
Let $\call_i$ be a sequence of prelaminations represented by $\calf$.

Then for $i$ large enough, there exists a shortening sequence of moves for $(\Sigma,\call_i)$ as
in definition \ref{dfn_shortening_moves}, certifying that a shortest solution of $\Sigma$ cannot induce $\call_{i}$.
\end{prop}

Proposition \ref{prop_exotic} will follow from the two following facts.

\begin{lem}\label{lem_exotic_cplexity}
The complexity of $\Sigma_j$, defined as the sum of number of bands of $\Sigma_j$ 
and of the number of connected components of $D_j$ remains bounded.

In particular, only finitely many distinct unfoliated band complexes with rational constraints 
appear in the sequence $\Sigma_1,\Sigma_2,\dots$.
\end{lem}

\begin{proof}
The second assertion easily follows from the first one since all rational constraints are in standard form,
defined using the same morphism $\rho:\fmip\ra \calm$.

We need only to prove that the number of bands of $\Lambda_j$ and of connected components of $D_{\Lambda_j}$
is bounded.
The band complex $\Lambda_j$ is homotopy equivalent to a graph $\Gamma_j$ whose set of vertices $V(\Gamma_j)$ 
is the set of connected components of $D_{\Lambda_j}$ and whose set of non-oriented edges $E(\Gamma_j)$ corresponds to non-oriented bands.
The first Betti number of $\Lambda_j$ is $b_1(\Lambda_j)=\# V(\Gamma_j)-\# E(\Gamma_j)$.
The moves of the pruning process preserve the first Betti number: this is clear for moves involved in step b and c;
 for moves in step a, this follows from the fact that if $[a,b]$ is pruned, then either $a\in \partial D$, 
or $a$ lies in the the interior of the pruned band because otherwise, the $\rond\calf$-leaf through $a$ would be finite,
contradicting $\mu$-minimality.

After step b,  every terminal vertex of $\Gamma_j$ corresponds to a connected component of $D_{\Lambda_j}$ 
containing a base of a band of $\Sigma_j\setminus \Lambda_j$.
Similarly, after step c, every vertex of valence 2 of $\Gamma_j$ corresponds to a connected component of $D_{\Lambda_j}$ 
containing a base of a band of $\Sigma_j\setminus \Lambda_j$.
Since step c does not create new terminal vertices, 
the number of vertices of valence 1 and 2 of $\Gamma_j$ is bounded.
Since $b_1(\Gamma_j)$ is bounded, the lemma follows.
\end{proof}

\begin{lem}\label{lem_exotic_mesure}
For all $\eta>0$, there exists $j$ large enough such that $\mu(D_{\Lambda_j})<\eta$.
\end{lem}

\begin{proof}
By \cite[Prop. 7.1]{GLP1}, the $\mu$-diameter of the connected components of $D_{\Lambda_j}$ converges to $0$.
Since the number of connected components of $D_{\Lambda_j}$ is bounded, the proposition follows.  
\end{proof}

\begin{proof}[Proof of Proposition \ref{prop_exotic}]
Since the complexity of band complexes 
appearing in the sequence $\Sigma_1,\Sigma_2,\dots$ is bounded,
there is a bound $C$ on the number of unfoliated band complexes with rational constraints appearing in this sequence.
The number of moves performed at step a, b and c, is bounded in terms of the complexity of $\Sigma_j$,
and is therefore bounded by some number $k$.
It follows that for all $j$, one can go back from the unfoliated band complex $\Sigma_j$ to 
the initial unfoliated band complex $\Sigma$ using at most $Ck$ extension moves.

Consider $\eps<1/3^{Ck}$.
Using Lemma \ref{lem_certif_mesure}, consider $\eta$ so that 
for any $I\subset D_\Lambda$ with $\mu(I)<\eta$, 
there exists $\eps$-shortness certificates for $I$ relative to $D_{\Lambda}$.
By Lemma \ref{lem_exotic_mesure}, let $j_0$ be large enough so that $\mu(D_{\Lambda_{j_0}})<\eta$. 

Let $\call_i$ be a sequence of prelaminations represented in $\calf$, and consider $i_0$ large enough so that
the sequence of moves from $(\Sigma,\calf)$ to $(\Sigma_{j_0},\calf_{j_0})$ 
can be applied to $(\Sigma,\call_{i_0})$, transforming it into $(\Sigma_{j_0},\call')$.
These moves are restriction moves with inert part $D_I=D\setminus \Lambda$ and active part $D_A=D_\Lambda$.
Since $\Sigma$ can be obtained from $\Sigma_{j_0}$ using at most $Ck$ extension moves of Lipschitz factor $3$,
going from $\Sigma_{j_0}$ to $\Sigma$ is an extension move of Lipschitz factor $3^{Ck}$.
Thus, we have found a shortening sequence of moves for $(\Sigma,\call_{i_0})$,
which certifies that a shortest solution of $\Sigma$ cannot induce $\call_{i_0}$ (see Definition \ref{dfn_shortening_moves}).
\end{proof}

\subsubsection{Surface components}\label{sec_surface}

\begin{prop}\label{prop_surface}
Consider a foliated band complex $(\sigma,\calf,\mu)$ with $\mu$-independent bands, 
decomposed into $\mu$-minimal components as in Proposition \ref{prop_minimal}.
Assume that $\Lambda$ is a $\mu$-minimal surface component.
Let $\call_i$ be a sequence of prelaminations represented by $\calf$.

Then for $i$ large enough, there exists a shortening sequence of moves for $(\Sigma,\call_i)$ as
in definition \ref{dfn_shortening_moves},
certifying that a shortest solution of $\Sigma$ cannot induce $\call_{i}$.
\end{prop}

Consider $\Lambda$ a $\mu$-minimal surface component.
By definition, this means that the pruning process defined in section \ref{sec_exotic} stops.
Let $(\Sigma',\calf',\mu')=(\Sigma_{j_0},\calf_{j_0},\mu_{j_0})$ the foliated complex obtained at the end of the pruning process.
The Rips band complex $\Lambda'_{\mu'}$ is a (non-orientable) interval exchange 
in the following sense: every $x\in D_{\Lambda'_{\mu'}}$ outside a finite set belongs to exactly two bases of $\Lambda'_{\mu'}$
(\cite[Cor. 6.3]{GLP1}).
Opening-up a separatrix (or unzipping a train-track carrying the foliation) 
gives a sequence of interval exchanges with arbitrary small domains.
We want to implement this \emph{unzipping process} at the level of $\Sigma$.

Consider $x_\mu\in D_{\Lambda'_{\mu'}}$ a point lying in 3 bands of $\Lambda'_{\mu'}$.
Let $B_1B_2\dots B_j\dots$ be the infinite $\Lambda'_{\mu'}$-word
corresponding to the unique infinite regular leaf path
starting at $x_\mu$, and consider the corresponding $\Lambda'$-word $w_j=B_1\dots B_j$.

\begin{lem}\label{lem_choix}
There exists $x\in\pi_{\mu'}\m(\{x_\mu\})$ lying in a singular leaf of $(\Sigma',\calf')$, such that for all $j\geq 0$,
$w_j(x)$ is not in the interior of a base of a band of $\Sigma'\setminus \Lambda'$.
\end{lem}

\begin{proof}
Let $I_1,\dots,I_p$ be the bases of the bands of $\Sigma'\setminus\Lambda'$ which lie in $D_{\Lambda'}$.
Recall that $\mu'(I_1\cup\dots\cup I_p)=0$.
Let $$E=\left\{x\in\pi_{\mu'}\m(\{x_\mu\})\ |\ \exists k\text{ s.t. } h_{w_j}(x)\in I_1\cup\dots\cup I_p \right\}.$$

Since the points $h_{w_j}(x_\mu)$ are all distinct in $\Lambda'_{\mu'}$, 
for each $i$ and each $x\in \pi_{\mu'}\m(x_\mu)$, there is at most one index $k$ such that $w_j(x)\in I_i$.
It follows that $E$ is a finite union of segments.
If $E$ is empty, just take any point $x\in\pi_\mu\m(x_\mu)$ lying in a singular leaf of $\Sigma$.
Otherwise, the leftmost point of $E$ satisfies the Lemma.
\end{proof}

We are now ready to describe the unzipping process.
Consider $x\in D_{\Lambda'}$ as given by the lemma above and $x_j=h_{w_j}(x)\in D_{\Lambda'}$.
For each $j\geq 1$, $x_j$ lies in the interior of exactly two bases of $\Lambda'$.
For $j=0,1,\dots$  perform the following moves on $\Sigma'$, modifying only $\Lambda'$:
\begin{itemize}
\item \emph{Step $a_j$.} Cut the domain at $x_j$.
\item  \emph{Step $b_j$.} Perform all possible forgetful moves on $\Lambda'$.
\end{itemize}

We denote by $(\Sigma_{i_0+j},\calf_{i_0+j},\mu_{i_0+j})$ the band complex obtained after performing steps $a_j$ and $b_j$.
Proposition \ref{prop_surface} will follow from the two following facts.

\begin{lem}\label{lem_surface_cplexity}
The complexity of $\Sigma_j$, defined as the sum of number of bands of $\Sigma_j$ 
and of the number of connected components of $D_j$ remains bounded.

In particular, only finitely many distinct unfoliated band complexes with rational constraints 
appear in the sequence $\Sigma_1,\Sigma_2,\dots$.
\end{lem}

\begin{proof}
Before step $a_0$, $x$ lies in the interior of at least 1 and at most 3 bases of $\Lambda'$
because $x_\mu$ lies in the boundary of two bases and in the interior of one base of $\Lambda'_{\mu'}$.
By choice of $x$, $x$ does not lies in the interior of a base of band of $\Sigma'\setminus\Lambda'$.
Step $a_0$ may increase $b_1(\Sigma)$ by at most 2.
After step $a_0$, $x_1$ lies in exactly 3 bases of $\Lambda$, and the interior of exactly one base  by Lemma \ref{lem_choix}. 
This does not change after set $b_0$.

By induction, we see that that steps $a_j,b_j$ do not change $b_1(\Lambda)$, 
and after these steps, $x_j$ lies in exactly 3 bases of $\Lambda$, and the interior of exactly one base by Lemma \ref{lem_choix}.
As in the proof of Lemma \ref{lem_exotic_cplexity},
consider the graph $\Gamma_j$ whose vertex set $V(\Gamma_j)$ 
is the set of connected components of $D_{\Lambda_j}$ and whose set of non-oriented edges $E(\Gamma_j)$ corresponds to non-oriented bands.
Since $\Lambda_{j,\mu_j}$ is an interval exchange, $G_j$ has no vertex of valence one.
Because of step b, each vertex of valence 2 of $\Gamma_j$ is bounded by the number of
bases of bands of $\Sigma_j\setminus \Lambda_j$.
Since $b_1(\Gamma_j)$ is bounded, the lemma follows.
\end{proof}

\begin{lem}\label{lem_surface_mesure}
For all $\eta>0$, there exists $j$ large enough such that $\mu(D_{\Lambda_j})<\eta$.
\end{lem}

\begin{proof}
Since $\{w_j(x_\mu)\}_{j\geq 0}$ is dense in $D_{\Lambda'_{\mu'}}$, 
the maximum of the measures of the connected components
of $D_{\Lambda'}\setminus\{x_0,\dots,x_j\}$ goes to $0$ as $j$ tends to infinity.
If follows that the measure of the connected components of $D_{\Lambda_j}$ go to $0$ as $j$ goes to infinity.
Since the number of connected components of $D_{\Lambda_j}$ is bounded, the lemma follows.
\end{proof}

The proof of Proposition \ref{prop_surface}
is now identical to the proof of Proposition \ref{prop_exotic}.

\section{Endgame}\label{sec_endgame}

\subsection{Finding solutions of band complexes}

\begin{thm}\label{thm_band_cplex}
  There is an algorithm which takes as input a band complex with rational constraints,
and decides whether it has a solution or not.
\end{thm}

Let us recall that we have defined two algorithms.
The first one is the \emph{prelamination generator}, described in Section \ref{sec_generator}. It takes  as input a band complex $\Sigma$ with rational constraints,
and explores the space of all prelaminations (with rational constraints). The prelaminations are organised into a rooted tree $\calt$ of finite 
valence, in such a way that children of a prelamination $\call$ are extensions of $\call$.
If $\sigma$ is a solution of $\Sigma$, it induces a M\"obius-complete prelamination $\call_\infty(\sigma)$ 
(see Sections \ref{subsec_prelamination} and \ref{sec_mobius} for definitions).
Lemma \ref{cor_extension} says that $\call_\infty(\sigma)$  will be produced by the prelamination generator.

The second algorithm is the \emph{prelamination analyser}, described in Subsection \ref{subsec_analyser}. 
It takes as input the band complex $\Sigma$ together with a
prelamination $\call$ (with rational constraints). 
If $\call$ is M\"obius-complete, it decides if $\call$ is induced by a solution or not (see Lemma \ref{lem_check_cplete}),
and in this case, the prelamination analyser stops and outputs ``Solution found'' together with a solution, or ``Reject'' accordingly.
If $\call$ is not M\"obius-complete, it looks for some certificate 
 ensuring that no shortest solution can induce $\call$. If it can find one, it says ``Reject'' and stops.
Otherwise, it may fail to stop or say ``I don't know''.
The certificate is either 
\begin{itemize*}
\item the detection of an incompatibility of rational constraints,
\item the detection of the non-existence of invariant measure,
\item the detection of an exponent of periodicity  that is too large for a shortest solution in regards of Bulitko's Lemma (Prop. \ref{prop_Bulitko}),
\item or the  detection of a shortening sequence of moves on $(\Sigma,\call)$,
which proves that a solution inducing $\call$ cannot be shortest.    
\end{itemize*}
See Section \ref{subsec_analyser} for more details.

We can now explain the structure of the main algorithm.

\begin{algorithm}[H]
\dontprintsemicolon
\refstepcounter{thm}\label{algo;main}
\caption{\textbf{Algorithm \thethm.} Main algorithm}
\KwIn{a band complex with rational constraints $\Sigma$}
\KwOut{a solution of $\Sigma$, or ``\emph{There is no solution}''} 
\BlankLine
\SetKw{Exit}{exit}

\SetKwFor{ForeachP}{For each}{do in parallel:}{endfor}
\SetKwFor{Forblank}{}{}{}
Run the prelamination generator to construct step by step the rooted tree $\calt$ of prelaminations\;

For each prelamination $\call$ produced by the generator, do in the background:\Forblank{}
{
\If{$\call$ is a leaf of $\calt$ and is not M\"obius-complete}{mark $\call$ as ``\emph{Rejected}''} 
Run the prelamination analyser on $\call$\;
\If{the analyser says ``\emph{Solution found}''}{\Exit with the solution provided by the analyser}
\If{the analyser says ``\emph{Reject}''}{mark $\call$ as ``\emph{Rejected}''}
}

Simultaneously:\Forblank{}{
\If{some prelamination $\call$ is not a leaf and all children $\call$ get marked as ``\emph{Rejected}''}{mark $\call$ as ``\emph{Rejected}''}

\If{the root lamination get marked as ``\emph{Rejected}''}{
\Exit saying ``\emph{There is no solution}''}
}
\end{algorithm}

\begin{rem*}
 If some prelamination $\call$ is rejected by the prelamination analyser, 
 we could also reject all its descendants, and we could ask the 
 prelamination generator to stop extending this prelamination.  

The algorithm could easily provide a proof that there is no solution when this is the case.
\end{rem*}

The main difficulty is to prove that the algorithm always stops.
Theorem \ref{thm_band_cplex} follows immediately from the two following lemmas.

\begin{lem} 
  The main algorithm is correct: if it stops, its answer is true.
\end{lem} 

\begin{proof}
 By Lemma \ref{lem_analyser_correct}, the analyser is correct. 
Thus, the main algorithm is correct if it stops because the analyser has found a solution.

We claim that any prelamination which is marked as rejected cannot be induced
by a shortest solution of $\Sigma$.

We prove the claim inductively.
If $\call$ is marked as rejected by the prelamination analyser, then the claim
is true by correctness of the prelamination analyser.  
If $\call$ is marked as rejected because $\call$ is a leaf of $\calt$ and is not M\"obius-complete,
then the claim is true because of Lemma \ref{lem_extension1}.
If $\call$ is not a leaf of $\calt$ and if all its sons are correctly marked as rejected, then 
the claim holds for $\call$ by  Corollary \ref{cor_extension}. 
This proves the claim.

If the root prelamination has been marked as rejected,
then the root prelamination cannot be induced by a shortest solution of $\Sigma$,
so $\Sigma$ has no shortest solution, and has no solution at all.
\end{proof}

\begin{lem}\label{lem_algo_stops}
The main algorithm always stops.
\end{lem}

\begin{proof}
Let $\Sigma_0$ be the band complex taken as input.
If $\Sigma_0$ has a solution, then it has a shortest solution $\sigma$, and the prelamination $\call_\infty(\sigma)$
will be produced by the prelamination generator. 
If the algorithm does not stop,
at some point $\call_\infty(\sigma)$ will be analysed by the analyser.
Since $\call_\infty(\sigma)$ is M\"obius-complete, the 
prelamination analyser will find out that there is a solution, and the main algorithm will stop, a contradiction.

Assume that $\Sigma_0$ has no solution. 
We claim that if some prelamination $\call$ is never rejected, 
then it is not a leaf, and at least one of its children is never rejected.
Indeed, every leaf of the rooted tree $\calt$ of prelaminations
will be marked as rejected: either because it is not M\"obius-complete, or because the analyser will reject it.
On the other hand, if all children of $\call$ are marked as rejected at some point, then $\call$ will itself be marked as rejected.
This proves the claim.

Assume by contradiction that the algorithm does not stop.
Then the root lamination never gets rejected.
The claim says that there is an infinite  ray in  $\calt$ of prelaminations which are never rejected. 
This ray is an infinite sequence of prelaminations with rational constraints extending each other:
$\call_1 \prec \call_2 \prec\dots$ (Definition \ref{dfn_extends}). 

First, there  exists a topological foliation $\calf_0$ on $\Sigma_0$ representing the prelaminations $\call_i$ 
 (this is by Proposition  \ref{prop_topological}, see Definition \ref{dfn_represent} for the notion of foliation representing a sequence of prelaminations).
Moreover, since no $\call_i$ is rejected for large exponent of periodicity,   
there exists  a measure $\mu_0$, without atom,  invariant under the holonomy of $\calf_0$.  
This is ensured by  Proposition  \ref{prop_measure}.

Now by  Proposition  \ref{prop_minimal},
if $i$ is large enough, one can uniformly perform a finite sequence of inert moves on $(\Sigma_0,\call_i)$,
and get a sequence of prelaminations represented in  a foliated band complex $(\Sigma_1,\calf_1,\mu_1$) 
which is decomposed into $\mu$-minimal components. 

By  Proposition \ref{prop_homogene}, 
the Rips band complex $\Sigma_{1\mu_1}$ corresponding to $(\Sigma_1,\calf_1,\mu_1)$ 
cannot contain a homogeneous minimal component, 
since otherwise, $\call_i$ would be rejected by the prelamination analyser for $i$ large enough
because of large exponent of periodicity.

Proposition \ref{prop_indep} then says that one can uniformly perform a finite sequence 
of inert moves on $(\Sigma_1,\calf_1,\mu_1)$
to get some band complex $(\Sigma,\calf,\mu)$ whose bands are $\mu$-independent. 
There are two alternatives left by the classification of minimal components: either there is an exotic component in  
$(\Sigma,\calf,\mu)$, or there is a surface component.
In both cases,  there exists a shortening sequence of moves for $\call_i$ for $i$ large enough
(see Propositions \ref{prop_exotic} and \ref{prop_surface}).
This proves the existence of shortening certificates for $\call_i$ for $i$ large enough,
contradicting the fact that no $\call_i$ is rejected.
\end{proof}

\subsection{Reduction of main theorems to band complexes}

\begin{prop}\label{prop_twisted2}
There exists an algorithm that takes as input a basis of a free group $F$, 
a finite set $\Phi$ of automorphisms of $F$ preserving $S\cup S\m$,
and a system of twisted equations with rational constraints in $F$ (with twisting automorphisms in $\Phi$)
and that decides whether there is a solution or not.  
\end{prop}

\begin{proof}
By Proposition \ref{prop;red_band_cmplex}, this reduces to 
deciding whether a band complex has a solution. Theorem \ref{thm_band_cplex} concludes.
\end{proof}

\begin{UtiliseCompteurs}{thmvf}
  \begin{thmbis}
  The problem of equations with rational constraints is solvable in finite extensions of free groups.

  More precisely, there exists an algorithm with takes as input a presentation of a virtually free group $G$, and a system of
  equations and inequations with constants in $G$, together with a set of rational constraints, and which decides if there exists
  a solution or not.
\end{thmbis}
\end{UtiliseCompteurs}

\begin{proof}
By proposition \ref{prop;trick_Out_Aut}, the  
 problem of equations with rational constraints in a virtually free group $V$ 
reduces to the problem of twisted equations with rational constraints 
in the sense of Definition \ref{def;twisted} where the twisting morphisms permute a basis.
Proposition \ref{prop_twisted2} concludes.
\end{proof}

Without assuming that twisting morphisms preserve a basis of $S$, we can prove:
\begin{UtiliseCompteurs}{twistedpb}
\begin{thmbis}
There exists an algorithm that takes as input a basis of a free group $F$, 
a finite set $\Phi$ of automorphisms of $F$ whose image in $\Out(F)$ generate a finite subgroup, 
and a system of twisted equations with rational constraints in $F$ (with twisting automorphisms in $\Phi$)
  and that decides whether there is a solution or not.
\end{thmbis}
\end{UtiliseCompteurs}

\begin{proof}
  Let $Q\subset \Out(F)$ be the finite group generated by the image of $\Phi$,
and let $V$ be the preimage of $Q$ in $\Aut F$.
Embed $F$ into $\Aut F$ via inner automorphisms $x\mapsto i_x$ where $i_x(g)=xgx\m$.
Note that for all $\phi\in \Phi$ and all $x\in F$, $i_{\phi(x)}=\phi\circ i_x \circ \phi\m$.
It follows each equation twisted by $\Phi$ on $F$ corresponds to a non-twisted equation in $V$,
together with the rational constraint saying that the variables should lie in $F$.
Since rational subsets of $F$ are rational subsets of $V$, the theorem follows from Theorem \ref{thm_vf}.
\end{proof}

\section{Equations in hyperbolic groups with torsion}\label{sec_eqn_hyp}

In \cite{RiSe_canonical}, Rips and Sela constructed canonical representatives for a torsion free hyperbolic group $\G$,
in order to \emph{lift} solutions of a system of equations on $\Gamma$ to a free group.
Then using Makanin's algorithm, they deduced an algorithm deciding whether a given system of equations in $\G$ has a solution.

Delzant \cite[Rem.III.1]{Delzant_image} remarked that for certain hyperbolic groups with torsion, such canonical representatives do
not exist, whereas Reinfeldt \cite{Reinfeldt} noticed that, if the abelian subgroups are finite or cyclic, they do exist.
In fact, in the general case, \emph{most} of the construction of \cite{RiSe_canonical} remains valid,
up to the construction of canonical cylinders.
Instead of defining canonical representatives as paths in the Cayley graph of $\G$ (thus living in a free group)
as in \cite{RiSe_canonical},
we need to interpret them as paths in the $1$-skeleton $\calk$ of the barycentric subdivision of a Rips complex of $\G$.
The action of $\G$ on $\calk$ fails to be free in presence of torsion, and paths in $\calk$ correspond to 
elements of the virtually free group $V$ occurring as the fundamental group of the
graph of groups $\calk/\G$ (see below for details).
The natural generalisation of Rips and Sela's construction in presence of torsion 
then leads to a family of systems of equations in the virtually free group $V$.

In \cite{Dah_existential}, rational constraints were used to handle inequations in torsion free hyperbolic groups.
More generally, we would like to solve equations with rational constraints in a hyperbolic group.
However, this is impossible without restriction on the rational subsets involved (even if $\G$ is torsion free).
For instance, any finitely generated subgroup is a rational subset, but
the membership problem is unsolvable in some hyperbolic groups (\cite{Rips_subgroups}).
This is why we introduce a nice class of rational subsets: \emph{\qi embeddable} rational subsets 
(see Definition \ref{dfn_qie} below).
Examples of such  subsets include 
finite subsets, quasiconvex subgroups, and their complements.

A set of \emph{\qi embeddable rational constraints} on a system of equations is the additional requirement that
each variable $x$ should live in a \qi embeddable rational subset $\calr_x\subset\G$. 
As we will see, this class of subsets is a Boolean algebra (Cor. \ref{cor_boolean}), so inequations
are a particular case of \qi embeddable rational constraints.

Here is the main statement of this section (compare with \cite[\textsection 5, 5.3]{Dah_existential}).

\begin{UtiliseCompteurs}{thm_eqn_hyp}
  \begin{thmbis}
    There exists an algorithm which takes as input
    \begin{itemize*}
    \item a presentation $\grp{S|R}$ of a hyperbolic group $\G$ (maybe
      with torsion),
    \item a finite system of equations, inequations, with constants in
      $\G$ and with \qi embeddable rational constraints
    \end{itemize*}
    and which decides whether there exists a solution or not.
  \end{thmbis}
\end{UtiliseCompteurs}

Each rational constraint $\calr_x$ should be given to the algorithm under the form of a finite automaton 
accepting a \qi embedded language $\Tilde \calr_x\subset \fmi$ projecting to $\calr_x\subset \Gamma$ (see Definition \ref{dfn_qie}).

Let us emphasise  that the algorithm is uniform over all hyperbolic groups. This is a usual application of the fact that the hyperbolicity 
 constant of a hyperbolic group can be computed from a presentation  
(see \cite[Prop 8.6.1]{Bow_notes} or \cite{Gromov_hyperbolic,Papasoglu_algorithm}).
 In all the following, we will make sure that the algorithms we use are explicit if the hyperbolicity constant is known 
(in particular that the constants we use are explicit, or computable in terms of the presentation and the hyperbolicity constant). 
 In this way, there is no restriction in considering that the hyperbolic group is given once and for all.

In particular, we get (compare with \cite[Theorem 0.1]{Dah_existential}): 

\begin{cor}\label{cor_exist}
  The existential theory with constants of a
  hyperbolic group  is decidable. 
\end{cor}

\subsection{Quasi-isometrically embeddable rational subsets}

\begin{dfn}\label{dfn_qie}
Let $\Gamma$ be a hyperbolic group generated by a finite set $S$, and $\pi:\fmi\onto \Gamma$ the corresponding morphism.

A regular language $\Tilde \calr\subset \fmi$ is \emph{quasi-isometrically embedded} in $\G$
if there exists $\lambda\geq 1,\mu\geq 0$, such that for any word $w\in \Tilde\calr$,
$\abs{\pi(w)}_\G\geq \frac1\lambda |w|-\mu$.

A rational subset $\calr\subset \G$ is \emph{quasi-isometrically embeddable} in $\G$ if
there exists a quasi-isometrically embedded regular language $\Tilde \calr\subset \fmi$
such that $\pi(\Tilde \calr)=\calr$.
\end{dfn}

Here, $|w|$ is the length of $w$ as a word on $\Spm$, and $\abs{\cdot}_\G$ is any word metric on $\G$.
In particular, finite sets and quasiconvex subgroups are clearly \qi embeddable rational subsets.

The property of being quasi-isometrically embeddable does not depend 
on the chosen word metric on $\G$. Neither does it depend on the choice generating set $S$.
Indeed, consider any other generating set $S'$ 
and express each element of $S$ as a word on $S'_{\pm}$.
This defines a morphism $\rho:\fmi\ra (S'_\pm)^*$ satisfying $|\rho(w)| \leq L|w|$ where $L=\max_{s\in\Spm} \abs{\rho(s)}$.
Consider the natural morphism  $\pi':(S'_\pm)^*\ra \G$.
The regular language $\Tilde\calr'=\rho(\Tilde \calr)$ is quasi-isometrically embedded since
for each $w'=\rho(w)$, $\pi'(w')=\pi(w)$, and $\abs{\pi'(w')}_\G=\abs{\pi(w)}_\G\geq \frac1\lambda |w|-\mu \geq \frac1{\lambda L}|w'|-\mu$.

However, the fact of being quasi-isometrically embedded does depend on the choice of $\Tilde\calr$ representing $\calr$. For instance, 
if $\calr$ is the set of all words, and $\calr'$ is the set of $L$-local geodesics  they represent the same rational subset (namely $\G$), but $\calr$ is not quasi-isometrically embedded.

A quasi-isometrically embeddable rational subset is quasiconvex: there exists $C>0$ such that
for all $x,y\in\calr$, any geodesic of $\G$ joining $x$ to $y$ is contained in the $C$-neighbourhood of $\calr$.
Indeed, by thinness of triangles, one can assume $x=1$.
Let $A$ be an automaton accepting $\Tilde \calr$.
By stability of quasigeodesics, any path represented by a word accepted by $\Tilde R$ lies within a bounded distance from a geodesic joining its endpoints.
Finally, if a path is accepted by $A$ then any of its points lies a bounded distance away from an accepted point (the bound depending only
on the number of states of $A$).
 
Conversely, we don't know if a quasiconvex rational subset is always quasi-isometrically embeddable.\\

\subsection{Full sublanguage of a set of quasigeodesics}

The goal of this section is to prove that given $\calr\subset \G$ a \qi embeddable rational subset,
the set of local quasigeodesics representing elements of $\calr$ is rational (Proposition \ref{prop_full}).
This should be thought as an analogue of 
 Lemma \ref{lem_reduit} saying that if $\grp{S}$ is a free group and $\calr\subset \grp{S}$ is a 
rational subset, then the language of freely reduced words in $\fmi$
representing an element of $\calr$ is a regular language.
As for the free group, we will deduce that 
 \qi embeddable rational subsets form a Boolean algebra.

Let $S$ be a finite generating system of $\G$, 
 $\pi: \fmi \onto \G$ be the corresponding morphism,
$\Cay \G$ be the Cayley graph, and $\abs{\cdot}_S$ be the word metric.

Given $w\in \fmi$, we denote by $p_w$ be the corresponding path in $\Cay \G$, defined by $p_w(i)=\pi(s_1\dots s_i)$.
Given $\lambda\geq 1$, $\mu\geq 0$, let 
$\qg_{\lambda,\mu}(\fmi)\subset \fmi$ be the set of $(\lambda,\mu)$-quasigeodesics, \ie
the set of words $w$ such that $\abs{p_w(i)-p_w(j)}_S\geq \frac1\lambda \abs{i-j}-\mu$.
Similarly, the set of $\nu$-local quasigeodesics
$\lqg_{\nu,\lambda,\mu}(\fmi)\subset \fmi$ is the set of 
words $w$ such that for all $\abs{i-j}\leq \nu$, $\abs{p_w(i)-p_w(j)}_S\geq \lambda \abs{i-j}-\mu$.
Clearly, $\lqg_{\nu,\lambda,\mu}(\fmi)$ is a regular language since this is the complement of the set of words
containing as a subword one of the finitely many words of length at most $\nu$ which are not $(\lambda,\mu)$-quasigeodesic.

We will use the following statement about stability of local quasigeodesics:

\begin{lem}[{\cite[Th\'eor\`eme 1.2, 1.4]{CDP}}] \label{lem;lambdaprime}
Fix $\delta$ a hyperbolicity constant.

Given $\lambda\geq 1$ and $\mu\geq 0$,
there exists computable numbers $\nu$, $\lambda'\geq 1$, $\mu'\geq 0$ such 
that any  $\nu$-local, $(\lambda,\mu)$-quasigeodesic is a 
$(\lambda',\mu')$-quasigeodesic.

Given $\lambda\geq 1,\mu\geq 0$, there exists a computable number $\eta$ such that
any $(\lambda,\mu)$-quasigeodesic is at Hausdorff distance at most $\eta$ from any geodesic 
with the same endpoints.
\end{lem}

We always assume that $\nu$ is large enough so that all $\nu$-local $(\lambda,\mu)$-quasigeodesics
are global quasigeodesics as in the result above.

\begin{prop}\label{prop_full}
Consider a hyperbolic group $\G$, $\pi:\fmi\onto \G$, and $\calr$ a \qi embeddable regular language.
Consider any $\lambda\geq 1, \mu\geq 0$, and any $\nu$ such that
$\nu$-local $(\lambda,\mu)$-quasigeodesics
are global quasigeodesics.

Then $\Tilde\calr=\pi\m(\calr)\cap \lqg_{\nu,\lambda,\mu}(\fmi)$
is a regular language of $\fmi$.

Moreover, given an automaton accepting a \qi embedded regular language representing $\calr$,
one can algorithmically compute an automaton $A$ accepting $\Tilde\calr$.
\end{prop}

We say that $\Tilde \calr$ is \emph{full} in $ \lqg_{\nu,\lambda,\mu}(\fmi)$
since any word in $ \lqg_{\nu,\lambda,\mu}(\fmi)$ representing an element of $\calr$ lies in $\Tilde \calr$.

\begin{rem}
 Recall that the set of geodesic words in a hyperbolic group is itself regular, by finiteness of Cannon's so-called cone
  types  
  \cite{Cannon_combinatorial}.  Using this fact, one could modify the proposition by substituting the set of global geodesic
  words to $\lqg_{\nu,\lambda,\mu}(\fmi)$.  However, we need to use local quasigeodesics for solving equations and inequations in
  hyperbolic groups.
\end{rem}

\begin{cor}\label{cor_boolean}
  The class of all \qi embeddable rational subsets of a hyperbolic group $\G$ is a Boolean algebra.

Moreover, this can be algorithmically computed: given automata accepting \qi embedded
regular languages representing $\calr_1,\calr_2$, one can compute  automata accepting \qi embedded
regular languages representing $\calr_1\cup\calr_2$, $\calr_1\cap\calr_2$ and $\Gamma\setminus \calr_1$.
\end{cor}

\begin{proof}[Proof of the corollary]
Clearly, the union of two \qi embeddable rational subsets is \qi embeddable.
Let $\calr\subset \G$ be a \qi embeddable rational subset.
By Lemma \ref{lem_union},
we need only to prove that $\G\setminus \calr$ is  \qi embeddable.
Let $\call=\lqg_{\nu,1,0}(\fmi)$ be the set of $\nu$-local geodesics (for some $\nu$ large enough so that local geodesics
are quasigeodesics).
By Proposition \ref{prop_full}, 
$\tilde R=\pi\m(\calr)\cap\call$ is a regular language of $\fmi$.
Since regular languages of $\fmi$ form a Boolean algebra, 
$\Tilde R'=\call\setminus \Tilde R$ is a regular language, 
it is quasi-isometrically  embedded,
and since $\pi(\call)=\G$,
$\pi(\Tilde R')=\G\setminus \calr$.
\end{proof}

\begin{proof}[Proof of Proposition \ref{prop_full}]
Let $\Tilde\calr_0\subset \qg_{\lambda_0,\mu_0}(\fmi)$ be a \qi embedded regular language representing $\calr$.
Consider a finite monoid $\calm$, a morphism $\rho:\fmi\onto \calm$, and $Accept\subset \calm$ such that
$\Tilde\calr_0=\rho\m(Accept)$.
Let $\calp$ be the language of prefixes of $\Tilde\calr_0$: 
 $w\in \calp$ if there exists $u\in \fmi$ with $wu\in\Tilde\calr_0$.
Note that $\calp$ consists of $(\lambda_0,\mu_0)$-quasigeodesics, 
and $\calp=\rho\m(P)$ where $P=\{m\in \calm \st \exists n\in\calm,\ mn\in Accept\}$.

Write $\call=\lqg_{\nu,\lambda,\mu}(\fmi)$.
Let $\lambda',\mu'$ and $\eta$ be such that $\nu$-local $(\lambda,\mu)$-quasigeodesics of $\G$ 
are global $(\lambda',\mu')$-quasigeodesic, and is at Hausdorff distance at most $\eta$
from any geodesic with the same endpoints (those values can be algorithmically computed from $\nu,\lambda,\mu$).

On the other hand, from the automaton $A_0$, one can algorithmically find $(\lambda'_0,\mu'_0)$ such that all accepted words
are $(\lambda'_0,\mu'_0)$-quasigeodesics. 
Indeed, given any positive $\nu_1$, and any $\lambda_1,\mu_1$, one can
algorithmically check whether all words accepted by $A_0$ are $\nu_1$-local $(\lambda_1,\mu_1)$-quasigeodesics.
Denote by $\nu(\lambda,\mu)$ $\lambda'(\lambda,\mu)$ and $\mu'(\lambda,\mu)$  computable functions such that $\nu(\lambda,\mu)$-local
$(\lambda,\mu)$-quasigeodesics are global $(\lambda'(\lambda,\mu),\mu'(\lambda,\mu))$-quasigeodesics.
By checking if $\Tilde\calr_0\subset\lqg_{\nu(\lambda_1,\mu_1),\lambda_1,\mu_1}$ for larger and larger values of $\lambda_1,\mu_1$,
one will finally get a positive answer, from which one can conclude that
$\Tilde\calr_0\subset \qg_{\lambda'(\lambda_1,\mu_1),\mu'(\lambda_1,\mu_1)}$.
Compute $\eta_0$ such that each $(\lambda'_0,\mu'_0)$-quasigeodesic
is at Hausdorff distance at most $\eta_0$
from any geodesic with the same endpoints.

For $h\in \G$, consider $\sigma_h=\rho(\calp\cap\pi\m(h))\subset \calm$ 
($\sigma_h$ can be thought as the set of states of the automaton corresponding to words representing $h$,
and which can be extended to accepted words).
Note that $h\in R$ if and only if $\sigma_h\cap Accept\neq \es$.

\newcommand{\undef}{\textrm{undef}}

Given $w\in \call$, Let $p_w$ be the path in $\Cay\Gamma$ corresponding to $w$.
We want to get a local picture of the values of $\sigma_h$ in the neighbourhood of the endpoint $g=\pi(w)$ of $p_w$.
Define $R=\eta_0+\eta+10\delta$ and
$L=\max\{(R+1+10\delta+\mu'+\eta)\lambda',\nu\}$.
Consider $w_L$ the suffix of length $L$ of $w$ ($w=w_L$ if $|w|\leq L$).
Let $q_w$ be the suffix of length $L$ of the path $g\m p_w$ 
($q_w$ is the path ending at $1$ and labelled by $w_L$).
Let $N$ be the $R$-neighbourhood of $q_w$,
and consider $\Sigma_w:N\ra 2^\calm$ defined by 
$\Sigma_w(h)=\sigma_{gh}$ (this records the values of $\sigma_h$ in the neighbourhood of the suffix of $p_w$).
In particular, $\pi(w)\in \calr$ if and only if $\Sigma_w(1)\cap Accept\neq \es$.
The mapping $\Phi:w\mapsto (w_L,\Sigma_w)$ clearly takes finitely many values.

\begin{claim*}
For all word $w\in\fmi$ and all $s\in\Spm$ with $w,ws\in \call$,
$\Phi(ws)$ depends only on $\Phi(w)$ and $s$.
\end{claim*}

\begin{proof}
If $|w_L|<L$, this is clear as $\Phi(w)$ then determines $w=w_L$.
Obviously, the $L$-suffix of $ws$ is determined by $w_L$ and $s$.
We need to prove that for all $|h|\leq R$, $\Sigma_w$ and $w_L$ determine $\sigma_{gsh}$.
Write $w$ as a concatenation $w=w'.w_L$ and let $x_L=\pi(w')\in \G$.
Consider $u\in\calp$ be a word with $\pi(u)=gsh$, and let $p_u$ be the corresponding path in $\Cay \G$.

\begin{fact*}
   $p_u$ intersects $B(x_L,R)$. 
\end{fact*}

\begin{proof}[Proof of the fact]
Consider $c_u$ (resp.\ $c_w$) a geodesic joining $1$ to $gsh$ (resp.\ to $g$).
The Hausdorff distance between $c_u$ and $p_u$ (resp.\ between $c_w$ and $p_w$) is at most $\eta_0$
(resp. $\eta$).
The projection $x_{w}$ of $x_L$ on $c_w$ satisfies $d(x_w,x_L)\leq \eta$.
Thus, $d(g,x_w)\geq -\eta + \frac{1}{\lambda'}L-\mu' \geq R+1+10\delta$ by choice of $L$.
Looking at a comparison tree for the triangle $1,g,gsh$, 
since $d(g,gsh)\leq R+1$, we see that $x_w$ is $10\delta$-close to some point $x_u\in c_u$.
The projection $x'_u$ of $x_u$ on $p_u$ is $\eta_0$ close to $x_u$ and satisfies
$d(x'_u,x_L)\leq \eta_0+10\delta + \eta =R$.
\end{proof}

We now prove that the fact implies our claim. 
Given $u\in \calp$ with $\pi(u)=gsh$, write
$u=u'u''$ with $\pi(u')=x'\in B(x_L,R)$, and let $m'=\rho(u')$. 
Consider $y_L=g\m x_L$ the initial point of $q_w$, and $y'=g\m x'\in N$.
For each $x'\in B(x_L,R)$, the set of possible $m'$ is known as it is encoded in $\sigma_{x'}=\sigma_{gy'}$.
Since $u''$ is a quasigeodesic between two points at distance at most $L+2R+1$,
$|u''|$ is bounded by $L_{u''}=\lambda(L+2R+1+\mu)$.
Since $u\in\calp$, among those quasigeodesic $u''$, 
one should consider only those such that $m'\rho(u'')\in P$ \ie $u''\in \rho\m(m'{}\m P)$.
Then $\sigma_{gsh}$ is precisely the set of all possible values of $m'\rho(u'')$
as $y'$ varies in $B(y_L,R)$, $m'$ varies in $\sigma_{gy'}$ and $u''$ varies in the set of words of length at most
$L_{u''}$ in $\rho\m(m'{}\m P)$.
This proves the claim, and proves moreover that 
$\Phi(ws)$ is algorithmically computable from $\Phi(w)$ and $s$.
\end{proof}

We now construct a deterministic automaton $A$ accepting $\call\cap\pi\m(\calr)$.
Note that for any individual $w\in \fmi$, $\Phi(w)$ is algorithmically computable since one can enumerate the finitely many
$(\lambda'_0,\mu'_0)$-quasigeodesics with endpoint $\pi(w)$.

Let $\call_{\leq L}$ be the set of words of length at most $L$ in $\call$.
Since the $R$-neighbourhood $N$ of $q_w$ satisfies 
$N\subset B(1,R+L)$, $\Sigma_w$ is a partially defined map on $B(1,R+L)$,
\ie an element of the finite set $D=(2^\calm\cup\{\undef\})^{B(1,R+L)}$.
Thus, $\Phi(\call)$ takes values in the finite set $F=\call_L\times D$.

We take $F$ as the set of states of our automaton, and $\Phi(1)$ as its initial state.
If $f=(w_L,\Sigma_w)\in F$ with $|w_L|<L$, we can discard $f$ if $f\neq \Phi(w_L)$.
If $|w_L|<L$ and $f=\Phi(w_L)$, for each $s\in\Spm$, we add an edge joining $f$ to $\Phi(w_L s)$.
If $|w_L|=L$, and if $w_L.s\in\call$, 
we connect $(w_L,\Sigma_w)$ to the state $(w',\Sigma')$ determined 
by $(w_L,\Sigma_w)$ and $s$ as in the claim.
The set of words $w$ read by the automaton starting from the initial state
is exactly $\call$.
Moreover, the corresponding final state is $\Phi(w)$ which determines
$\sigma_{\pi(w)}=\Sigma_w(1)$.
We define the set of accepting states of $A$ as the set of elements $(w_L,\Sigma_w)\in F$
such that
$\sigma_\pi(w)\cap Accept\neq \es$.
Since $\pi(w)\in \calr$ if and only if $\Sigma_{\pi(s)}\cap Accept\neq es$,
the language accepted by $A$ is $\pi\m(\calr)\cap\call$.

Finally, note that the automaton $A$ can be algorithmically computed from $A_0$.
\end{proof}

\subsection{Canonical representatives in hyperbolic groups with torsion}

\newcommand{\pp}{\mathbf{p}}
\newcommand{\qq}{\mathbf{q}}
\newcommand{\cc}{\mathbf{c}}
\renewcommand{\ll}{\mathbf{l}} %\ll signifie ``<<''
\newcommand{\rr}{\mathbf{r}}

\subsubsection{Canonical sliced cylinders and paths.}

Let $\Gamma$ be a hyperbolic group, and $\Cay \Gamma$ be a Cayley graph.  
Let $\delta$ be its hyperbolicity constant.
A \emph{cylinder} for  $(x,y)\in \G\times \G$ is a subset of the 
$5\delta$-neighbourhood of a geodesic segment 
$[x,y]$ in $\Cay \G$, containing every such segment.   
A \emph{slicing} of a cylinder (or a slice decomposition),  
is a partition into subsets (called slices), 
with a total ordering of them, such that 
\begin{itemize}
\item the union of any two consecutive slices has diameter at most $50\delta$
\item for  this ordering, the slice containing $x$ is smaller or equal to the slice containing $y$
\item slices move quasigeodesically: if $x,y$ lie respectively in the $i$-th and $j$-th slice for this ordering,
then $d(x,y)\geq \frac{1}{4m_0} |i-j| -200\delta$ where $m_0=\#B(1,50\delta)$.
\end{itemize}

In \cite{RiSe_canonical}, Rips and Sela proved the following theorem (which applies for hyperbolic groups with torsion).

\begin{thm}[{\cite[Cor. 4.3]{RiSe_canonical}}]  
\label{thm_can_cyl}
Let $\G$ be a hyperbolic group and $\Cay \G$ a Cayley graph. 
Then there exists a computable constant $\kappa$ (controlling the size of the defect $C_i$ below) such that the following holds.

Let $X$ be a set of variables, and $\cale$ be a triangular system of equations over $X$.
Let $\ul g=(g_x)_{x\in X\cup \ol X}\in \G^X$ be a solution of $\cale$.

Then for every $x\in X\cup \ol X$, there exists a cylinder $\Cyl_x$ for $(1,g_x)$ together with a slicing, such that
\begin{itemize}
\item for each variable $x\in X$, $\Cyl_{\ol x}=g_x\m \Cyl_x$, their slices are the same but in reverse order,
\item for each equation $x_1x_2x_3\in \cale$, and for each $i\in\{1,2,3 \mod 3\}$ 
there is a decomposition (depending on the equation) of the cylinder as 
$$\Cyl_{x_i}= L_i \dunion C_i \dunion R_i$$ 
where
\begin{itemize}
\item $L_i,C_i,R_i$ are union of slices with $L_i<C_i<R_i$ with respect to the ordering,
\item $L_{i+1}=g_{xi}\m R_{i}$, their slices are the same but in reverse order,
\item $C_i$ contains at most $\kappa. \#\cale$ slices.
\end{itemize}
\end{itemize}
\end{thm}

\begin{figure}[htb]
  \centering
\includegraphics[width=6in]{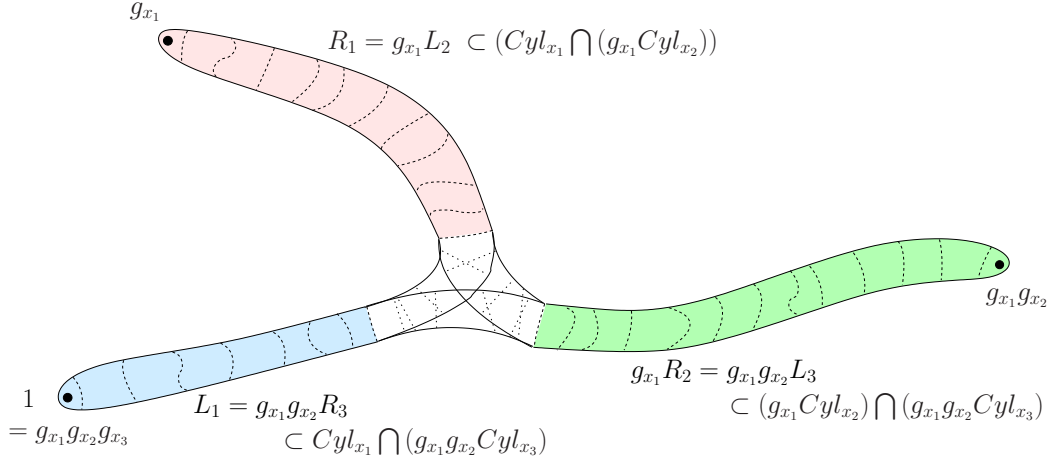}  
  \caption{Three cylinders for a solution to $x_1x_2x_3=1$. The subdivision in slices is drawn. The central part, where the cylinders do not coincide, does not contain more than $3\kappa  \#\cale$ slices.}
  \label{fig_repcan}
\end{figure}

Note that cylinders and slicings depend only on the variable considered whereas the decomposition of the cylinder depends on the equation
in which the variable appears.  
The fact that slices move quasigeodesically is a consequence of their definition in \cite{RiSe_canonical},
see \cite[Prop. 3.8]{Dah_existential} for a proof.
\\

Next step consists in interpreting sliced cylinders as paths in a suitable graph $\K$.
Let us consider the Rips complex $P_{50\delta}(\Gamma)$ whose set of vertices is $\G$, and whose simplices 
are subsets of $\Gamma$ of diameter at most $50\delta$. 
Let us call $\K$ the $1$-skeleton of its barycentric subdivision. 
We identify $\Gamma$ to a subset of the vertices of $\K$, in the obvious way. 
For any path $\pp:\{0,\dots,n\}\ra \calk$, we denote by $\ol\pp$ the reverse path
defined by $\ol\pp(i)=\pp(n-i)$. 
If $\qq$ is a path with $\qq(0)=\pp(n)$, 
we denote by $\pp\cdot \qq$ the concatenation.

\begin{prop}    \label{prop;canpaths}                
Let $\G$ be a hyperbolic group, $\Cay \G$ a Cayley graph,
and $\K$ the $1$-skeleton of the barycentric subdivision of the Rips complex of $\Gamma$ as above.
Consider $\lambda_0= 400\delta m_0$
where $m_0$ is a bound on the cardinality of balls of radius $50\delta$ in $\Cay \G$,
$\mu_0 = 8$,  
and $\kappa$ as in Theorem \ref{thm_can_cyl}.

Let $\cale$ be a triangular system of equations over $X$, and
$\ul g=(g_x)_{x\in X}\in \G^X$ be a solution of $\cale$ and let $g_{\ol x}=g_x\m$.

Then for every $x\in X\cup \ol X$, there exists a $(\lambda_0,\mu_0)$-quasigeodesic path $\pp_x$ 
joining $1$ to $g_x$ in $\K$ such that
\begin{itemize}
\item for each variable $x\in X$, $\pp_{\ol x}=g_x\m \ol\pp_x$
\item for each equation $x_1x_2x_3\in \cale$, and for each $i\in\{1,2,3\mod 3\}$
there is a decomposition (depending on the equation) 
$\pp_{x_i}= \ll_i \cdot \cc_i \cdot \rr_i$,
where
$\ll_{i+1}=g_{x_i}\m \ol{\rr_{i}}$,
and $\cc_i$ has length at most $2\kappa. \#\cale$. 
\end{itemize}
\end{prop}

\begin{proof}
Consider sliced cylinders $C_x$ given by Rips and Sela's Theorem \ref{thm_can_cyl}.
Let $S,S'$ be the slices containing $1$ and $g_x$ respectively,
and $S=S_1,\dots , S_n=S'$ be the set of slices between $S$ and $S'$, in increasing order.
Since $S_i$ of $C_x$ has diameter at most $50\delta$, it defines a simplex of $P_{50\delta}(\Gamma)$, and thus a vertex
$v_i$ of $\K$.
Similarly,  $S_i\cup S_{i+1}$ defines a vertex $u_i$ of $\K$
connected by an edge to $v_i$ and $v_{i+1}$,
and there is an edge in $\calk$ joining $1$ to $v_1$ and $v_n$ to $g_x$.
We can therefore define $\pp_x$ as $1,v_1,u_1,v_2,\dots,u_{n-1},v_n,g_x$.

Let us explain why these paths are quasigeodesic, the other properties immediately follow
from the properties of canonical cylinders.
The fact that slices vary quasigeodesically say that 
the distance in $\Cay\G$ between the $i$-th slice and the $j$-th slice is at least 
$|j-i|/4m_0 -200\delta$. 
Thus, the barycentres of the corresponding simplices in  $P_{50\delta}(\Gamma)$ are at distance 
at least $|j-i|/(200\delta m_0) -4$ so $d_\K(v_i,v_j)\geq |j-i|/(200\delta m_0) -4$. 
Taking into account the fact that $\pp_x$ takes $2|j-i|$ steps to go from $v_i$ to $v_j$, we get
$d_\K(\pp_x(s),\pp_x(t))\geq |s-t|/(400\delta m_0) -4$ for all odd $s,t$.    
Since $u_i$ is at distance $1$ from $v_i$ and $v_{i+1}$,
one gets for all $s,t$
$d_\K(\pp_x(s),\pp_x(t))\geq (|s-t|-2)/(400\delta m_0) -4-2\geq |s-t|/(400\delta m_0)-8$    
as desired.
\end{proof}

\subsubsection{The virtually free group $V$, and canonical representatives for $\Gamma$.}\label{sec;the_group_V}

We now define the virtually free group $V$ as a group of paths.
We consider homotopy classes of edge paths relative to endpoints in the graph $\K$. 
A path is \emph{reduced} if it has no subpath of length $2$ passing twice through the same edge. 
Thus, each homotopy class of paths in $\K$ has a unique reduced representative.

Let $V$ be the set of all homotopy classes of edge paths $\pp$ in $\K$ that start at the vertex $1_\G$, 
and end at a vertex of $\G$. Define $\pi:V\onto \Gamma$ by mapping $\pp$ to its endpoint.
We endow $V$ with a group structure by defining $ \pp \pp'$ to be the
homotopy class of the concatenation $\pp\cdot (\pi(\pp)\pp')$ 
(where $\pi(\pp)\pp'$ is the translate of $\pp'$ by $\pi(\pp)\in\Gamma$). 
Note that $\pp \m=\pi(p)\m \ol\pp$.
It is also clear that $\pi: V \to \G$ is a surjective homomorphism.

\begin{lem}
  The group $V$ is virtually free.
More precisely, it is isomorphic to the fundamental group of the finite graph of finite groups defined by $\K/\G$, 
where vertices and edges are marked by copies of stabilisers of their preimages in $\K$. 
In particular, a presentation of $V$ is computable from a presentation of $\G$.  
\end{lem}

\begin{proof}
  Let $p:T \tto \K$ be the universal cover of the graph $\K$; it is a tree. 
With the choice of a preimage $v_0\in T$ of $1_\G\in
  \K$, the group $V$ admits a natural action on $T$, as we explain now. 
Any path $\pp$ in $\K$ starting at $1_\G$ has a unique
lift $\tilde{\pp}$ in $T$ starting at $v_0$.
On the other hand, for any point $x\in T$, the path $[v_0,x]\subset T$ defines a path $\qq_x$ by projection to $\K$.
Then one can define $\pp.x$ as the endpoint of lift starting at $v_0$ of $\pp\cdot (\pi(p)\qq_x)$.
It is easy to check that $p:T\tto \K$ is $\pi$-equivariant.
In particular, $T/V \simeq \K/\G$ is a finite graph.

Note that $\pi_1(\K,1)\subset V$ is the kernel of $\pi$ and acts freely on $T$.

To show that $V$ is virtually free, we prove that stabilisers of vertices of $T$ are finite.
Let $V_x\subset V$ be the stabiliser of a vertex $x\in T$. Then $\pi(V_x)$ stabilises $p(x)\in\K$.
Since $\ker \pi$ acts freely on $T$, $\pi_{|V_x}$ is injective. 
This proves the finiteness of vertex stabilisers.

To prove the computability of vertex stabilisers, we prove that $\pi$ induces an isomorphism
between $V_x$ and the $\G$ stabiliser of $p(x)$, which is computable.
We need to prove that any $g\in\G$ fixing $p(x)$ lies in $\pi(V_x)$.
Consider the path $\qq'=g\qq_x$ joining $g$ to $p(x)$, and
$\pp=\qq_x\cdot \ol\qq'$ joining $1$ to $g$.
Then $\pp.x$ is the endpoint of the lift of $\pp\cdot (g\qq_x)= \qq_x\cdot \ol\qq'\cdot (g\qq_x)=\qq_x$,
so $\pp.x=x$.

A similar argument shows that $\pi$ induces an isomorphism between edge stabilisers of $T$ and $\K$,
so one can compute a finite graph of finite groups representing $V$.
\end{proof}

\subsection{Lifting equations to $V$}

Consider $(\lambda_1,\mu_1)=(\lambda_0,\mu_0+2+\frac{2}{\lambda_0})$,
so that any concatenation of a $(\lambda_0,\mu_0)$-quasigeodesic with a path of length 1 at each extremity
is a $(\lambda_1,\mu_1)$-quasigeodesic.
Let $\qg_{\lambda_1,\mu_1}(V)\subset V$ be the set of elements such that the corresponding reduced path in $\K$ is
$(\lambda_1,\mu_1)$-quasigeodesic.
Denote by $V_{\leq L}$ the set of elements of $V$ whose corresponding reduced path in $\K$
has length at most $L$.

Interpreting the paths occurring in Proposition \ref{prop;canpaths} in terms of elements of $V$, we get:
\begin{prop}\label{prop_canrep}
  Consider a system of equations in $\Gamma$ as in Proposition \ref{prop;canpaths},
and $(g_x)_{x\in X}\in \G^X$ a solution. Let $\kappa_1=2\kappa.\#\cale+2$.

Then for each variable $x\in X\cup\ol X$, there exists $\Tilde g_x\in \qg_{\lambda_1,\mu_1}(V)$ with $\Tilde g_{\ol x}=(\Tilde g_x)\m$ 
and $\pi(\Tilde g_x)=g_x$,
and for each equation $\eps\in \cale$ representing the equation $x_1x_2x_3=1$ and for each $i\in \{1,2,3\mod 3\}$, there exists
$\Tilde l_{\eps,i}\in\qg_{\lambda_1,\mu_1}(V)$ and $\Tilde c_{\eps,i}\in V_{\leq \kappa_1}$, such that
\renewcommand{\labelenumi}{(\roman{enumi})}
\begin{enumerate*}
\item $\Tilde g_{x_i}= \Tilde l_{\eps,i} \Tilde c_{\eps,i} \Tilde l_{\eps,i+1}\m$ in $V$
\item $\pi (\Tilde c_{\eps,1}\Tilde c_{\eps,2}\Tilde c_{\eps,3})=1$ in $\G$.
\end{enumerate*}

Conversely, given any family of elements of $V$ $(\Tilde g_{x})_{x\in X\cup \ol X}$,
$(\Tilde l_{\eps,i})_{\eps\in\cale,i=1,2,3}$,
$(\Tilde c_{\eps,i})_{\eps\in\cale,i=1,2,3}$ satisfying $g_{\ol x}=g_x\m$ and  $(i)-(ii)$,
the family $g_x=\pi(\Tilde g_x)$ is a solution of $\cale$.
\end{prop}

\begin{figure}[htbp]
  \centering
  \includegraphics{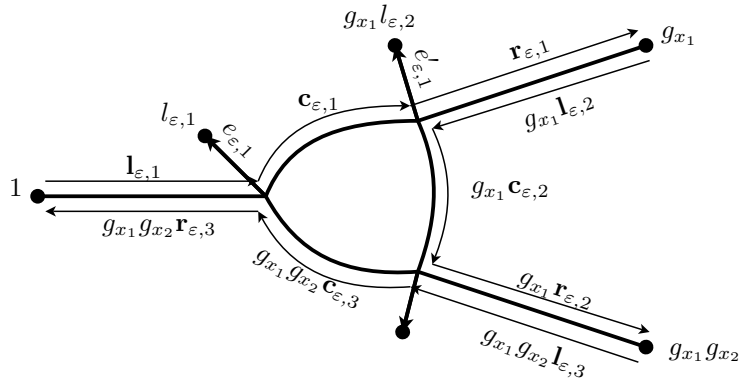}
  \caption{canonical paths in $\K$}
  \label{fig_can_paths}
\end{figure}

\begin{proof}
  The converse implication is obvious as for each equation $\eps$ written $x_1x_2x_3=1$, 
 $\pi(\Tilde g_{x_1}\Tilde g_{x_2}\Tilde g_{x_3})$ simplifies to
$\pi(\Tilde l_{\eps,1}\Tilde c_{\eps,1}\Tilde c_{\eps,2}\Tilde c_{\eps,3}\Tilde l_{\eps,1}\m)=1$.

For the direct implication, apply Proposition \ref{prop;canpaths} to get a path $\pp_x$
for each $x\in X\cup \ol X$, and for each equation $\eps$ written $x_1x_2x_3=1$, 
and each $i\in \{1,2,3\mod 3\}$, three paths $\ll_{\eps,i},\cc_{\eps,i},\rr{\eps,i}$.
Since $\pp_x$ joins $1$ to $g_x$, it represents an element $\Tilde g_x=\pp_x\in V$, and $\pi(\Tilde g_x)=g_x$.
On the other hand $\ll_{\eps,i}$ may fail to end at a vertex of $\Gamma$, so consider $e_{\eps_i}$
an edge of $\K$ joining the endpoint of $\ll_{\eps,i}$ to a point in $l_{\eps,i}\G$.
This allows to define an element of $V$ by $\Tilde l_{\eps,i}=\ll_{\eps,i}\cdot e_{\eps,i}$.
Since $\ll_{\eps,i+1}=g_{x_i}\m \ol{\rr_{\eps,i}}$, the edge $g_{x_i} e_{\eps,i+1}$ joins the initial point of $\rr_{\eps,i}$
to $g_{x_i}l_{\eps, i+1}\in \G$, we denote it by $e'_{\eps,i}$ (see Figure \ref{fig_can_paths}).
One can therefore define the following elements of $V$: 
$\Tilde c_{\eps,i}=l_{\eps,i}\m(\ol e_{\eps,i}\cdot \cc_{\eps,i} \cdot e'_{\eps,i})$,
and $\Tilde r_{\eps,i}=(g_{x_i}l_{\eps i+1})\m(\ol e'_{\eps,i}\cdot \rr_{\eps,i})$.

By construction, $\Tilde g_{x_i}=\Tilde l_{\eps,i}\Tilde c_{\eps,i}\Tilde r_{\eps,i}$.
Since $\ll_{\eps,i+1}=g_{x_i}\m \ol{\rr_{\eps,i}}$, we get
$\Tilde l_{\eps,i+1}=\ll_{\eps,i+1}\cdot e_{i+1}=g_{x_i}\m (\ol{\rr_{\eps,i}}\cdot e'_{\eps,i})=\Tilde r_{\eps,i}\m$.
The fact that $\pi (\Tilde c_{\eps,1}\Tilde c_{\eps,2}\Tilde c_{\eps,3})=1$ then follows from $\pi(\Tilde g_{x_1}\Tilde g_{x_2}.\Tilde g_{x_3})=1_\G$.
The fact that $\Tilde l\in \qg_{\lambda_1,\mu_1}(V)$ and $\Tilde c\in V_{\leq \kappa_1}$ is clear from the analogous facts 
satisfied by $\ll$ and $\cc$.
\end{proof}

Note that statement $(i)$ of Proposition \ref{prop_canrep} is an equation in $V$ while $(ii)$
takes place in $\G$. But since each $\Tilde c_{\eps,i}$ is short, there are only finitely possibilities,
and for each value of $\Tilde c_{\eps,i}$ satisfying $(ii)$,
one can think of $(i)$ as a system equations in $V$ with variables $\Tilde l_{\eps,i}$ and $\Tilde g_x$,
and constants $\Tilde c_{\eps,i}$.
More formally:

\begin{dfn}\label{dfn_eqn_lift}
For each tuple $\ul{\Tilde c}$ of elements $\Tilde c_{\eps,i}\in V_{\leq \kappa_1}$ satisfying $(ii)$,
define $\Tilde \cale(\ul{\Tilde c})$ as the system of equations $(i)$ with unknowns $\Tilde g_x$ and $\Tilde l_{\eps,i}$.
  
For each solution of $\Tilde\cale(\ul{\Tilde c})$, its projection $(g_x)=(\pi(\Tilde g_x))\in \Gamma^X$,
is a solution of $\cale$. We say that the solution $(\Tilde g_x,\Tilde l_{\eps,i})$ is a \emph{lift}
of $(g_x)$ in $V$.
\end{dfn}

Proposition \ref{prop_canrep} says that every solution of $\cale$ has a lift in $V$, and more precisely in 
$\qg_{\lambda_1,\mu_1}(V)\subset V$.

\subsection{Lifting rational constraints to $V$, and proof of the theorem}

Before proving the main theorem of this section, we need to lift 
the \qi embedded rational subsets of $\Gamma$ in $V$.
Let $\Hat S$ be a generating system of $V$.
Let  $\call_{\Hat S}=\lqg_{\nu,\lambda,\mu}\subset \Hfmi$ be the 
the rational language of local quasigeodesics for the metric $\abs{\cdot}_{\Hat S}$ on $\G$.
Let $\call_V$ be its image in $V$.
Given $\lambda_1,\mu_1$ as in previous section, one can compute $\lambda,\mu$ large enough so that
$\call_V$ contains $\qg_{\lambda_1,\mu_1}(V)$.

\begin{lem}\label{lem_full_V}
Let $\calr$ be any \qi embeddable rational subset of a hyperbolic group $G$.
Consider $V$, 
$\pi:V\onto G$ and $\qg_{\lambda_1,\mu_1}(V)\subset\call_V\subset V$ as above.

Then $\Tilde\calr_V=\call_V\cap\pi\m (\calr)$ is a rational subset of $V$, and $\pi(\Tilde \calr_V)=\calr$.
\end{lem}

\begin{proof}
  By Proposition \ref{prop_full},
the language $\Tilde\calr_{\Hat S}=\pi\m(\calr)\cap\call_{\Hat S}\subset \Hfmi$ is regular.
Since $\pi_{\Hat S}:\Hfmi\ra G$ factors through $\pi_V:V\ra G$,
$\Tilde\calr_v$ is the image in $V$ of $\Tilde\calr_{\Hat S}$.
Therefore, it is rational.
\end{proof}

We are now ready to prove that one can decide algorithmically whether a system of equations with \qi-embeddable
rational constraints has a solution.

\begin{proof}[Proof of Theorem \ref{thm_eqn_hyp}]
Since $\Gamma\setminus\{1\}$ is a \qi embeddable rational subset by Corollary \ref{cor_boolean},
and since constants can be encoded by rational constraints consisting of a single element,
we can reduce to the case of a system of equations $\cale$ with \qi embeddable rational constraints $\calr_x\subset \G$ (without constants).

For each $x\in X$, consider $\Tilde\calr_x=\call_V\cap\pi\m(\calr_x)\subset V$ which is rational by Lemma \ref{lem_full_V}.
Compute $\kappa_1$, and enumerate all tuples $\ul{\Tilde c}$ in $V_{\leq \kappa_1}$ satisfying  
$\pi (\Tilde c_{\eps,1}\Tilde c_{\eps,2}\Tilde c_{\eps,3})=1$ (this uses a solution of the word problem in $\G$).
Consider the lifted system of equations $\Tilde \cale(\ul{\Tilde c})$ in $V$ as in definition \ref{dfn_eqn_lift},
and add the rational constraint $\Tilde g_x\in \Tilde \calr_x$.
Denote by $\ul{\Tilde \calr}$ the corresponding set of rational constraints.
For each $\ul{\Tilde c}$, use the  algorithm  of Theorem \ref{thm_vf} for  virtually free groups to decide whether 
$(\Tilde \cale(\ul{\Tilde c}), \ul{\Tilde \calr})$ has a solution.
We claim that $(\cale,\ul\calr)$ has a solution if and only if 
$(\Tilde \cale(\ul{\Tilde c}),\Tilde \calr)$ has a solution for at least one $\ul{\Tilde c}$.

If $(\Tilde \cale(\ul{\Tilde c}),\Tilde \calr)$ has a solution, then $\pi(\Tilde g_x)$ is a solution of $\cale$ by Proposition \ref{prop_canrep},
and since $\Tilde g_x\in \Tilde \calr_x$, $\pi(\Tilde g_x)\in \calr_x$, so the constraints are satisfied.

Conversely, if $(\cale,\ul\calr)$ has a solution $(g_x)$, then  by Proposition \ref{prop_canrep},
$\Tilde \cale(\ul{\Tilde c})$ has a solution lying in $\qg_{\lambda_1,\mu_1}(V)\subset \call_V$
for some tuple $\ul{\Tilde c}$.
Since the solution satisfies the rational constraint $g_x\in\calr_x$,
the corresponding solution of $\Tilde \cale(\ul{\Tilde c})$ 
satisfies $\Tilde g_x\in \call_V\cap\pi\m(\calr_x)=\Tilde \calr_x$, so 
this solution of $\Tilde \cale(\ul{\Tilde c})$  satisfies the rational constraint $\ul{\Tilde \calr}$. 
\end{proof}

{\small
\bibliographystyle{alpha}
\bibliography{published,unpublished}
}

\begin{flushleft}
Fran{\c c}ois Dahmani, Vincent Guirardel\\
Institut de Math\'ematiques de Toulouse  (UMR 5219)\\
Universit\'e Paul Sabatier, Toulouse III\\
31062 Toulouse cedex 9.\\
France.\\
\emph{e-mail: } \texttt{dahmani@math.univ-toulouse.fr}, \texttt{guirardel@math.univ-toulouse.fr}\\
\end{flushleft}

\end{document}